\documentclass[12pt]{article}

\usepackage[margin=2.5cm, top=2.5cm, bottom=2.5cm]{geometry}

\usepackage[utf8]{inputenc} 
\usepackage[T1]{fontenc}    

\usepackage{url}            
\usepackage{booktabs}       
\usepackage{amsfonts}       
\usepackage{nicefrac}       
\usepackage{microtype}      
  
\usepackage{amsmath,amssymb,amsthm,xcolor}
\usepackage{array}
\usepackage{float}

\usepackage{wrapfig}
\usepackage{thmtools}
\usepackage[hidelinks,hypertexnames=false]{hyperref}
\usepackage[capitalise]{cleveref}

\theoremstyle{plain}
\newtheorem{theorem}{Theorem}[section]
\declaretheorem[
  name=Fact,          
  sibling=theorem,      
  refname={Fact,Facts}   
]{fact}
\declaretheorem[name=Lemma,       sibling=theorem, refname={Lemma,Lemmas}]{lemma}
\declaretheorem[name=Corollary,   sibling=theorem, refname={Corollary,Corollaries}]{corollary}

\declaretheorem[name=Proposition, sibling=theorem, refname={Proposition,Propositions}]{proposition}
\declaretheorem[name=Assumption,  sibling=theorem, refname={Assumption,Assumptions}]{assumption}
\declaretheorem[name=Example,     sibling=theorem, refname={Example,Examples}]{example}
\declaretheorem[name=Remark,      sibling=theorem, refname={Remark,Remarks}]{remark}
 \declaretheorem[
  name=Definition,
  sibling=theorem,
  refname={Definition,Definitions},
  style=definition
]{definition}

\usepackage{graphicx}
\usepackage{subcaption}
\usepackage{mdframed,framed}
\usepackage{lipsum}
\usepackage{tikz,calc}
\usepackage{enumitem}
\usetikzlibrary{calc}
\allowdisplaybreaks
\usepackage{pgfplots}
\usetikzlibrary{intersections, pgfplots.fillbetween}
\usepackage{epstopdf}
\usepackage{algorithmic}
\usepackage{booktabs}
\usepackage{changepage}
\usepackage{centernot}
\usepackage{stmaryrd}
\usepackage[nottoc]{tocbibind}
\newcommand{\forae}{%
  \tikz[baseline={(forall.base)}]{
    \node[inner sep=0pt, outer sep=0pt] (forall) {$\forall$};
    \fill[white] (-0.06em,0.05ex) rectangle (0.06em,0.25ex);
    \draw[line width=0.04em] (-0.06em, 0.7ex) -- (0.06em, 0.7ex);
  }
}

\usepackage{empheq}

\usepackage{shadethm}

\newshadetheorem{assume}{Assumption}
\definecolor{shadethmcolor}{HTML}{FFFFFF}
\definecolor{shaderulecolor}{HTML}{000000}
\newcommand{\R}{\mathbb{R}}

\newcommand{\fg}{\mathfrak{g}}

\newcommand{\dom}{\mathrm{dom}}
\newcommand{\rank}{\mathrm{rank}\hspace{.3mm}}
\newcommand{\diag}{\mathrm{diag}}
\newcommand{\gph}{\mathrm{gph}}

\newcommand{\N}{\mathbb{N}}
\newcommand{\lip}{\mathrm{lip}}

\def\O{\mathrm{O}}
\def\GL{\mathrm{GL}}
\def\p{^}
\newcommand{\co}{\mathrm{co}\hspace*{.3mm}}
\def\loc{\mathrm{loc}}
\usepackage{accents}
\def\cf{\accentset{\circ}{f}}
\def\of{\overline{f}}
\def\epi{\mathrm{epi}}
\def\im{\mathrm{Im}\hspace{.3mm}}

\newcommand{\inte}{\mathrm{int}}

\title{\LARGE On the geometry of flat minima}

\begin{document}

\author{\large C\'edric Josz\thanks{\url{cj2638@columbia.edu}, IEOR, Columbia University, New York.}}
\date{}

\maketitle

\begin{center}
    \textbf{Abstract}
    \end{center}
    \vspace*{-3mm}
 \begin{adjustwidth}{0.2in}{0.2in}
 ~~~~ What does it mean to be flat?
 We propose to define it by measuring the maximal variation around a point, or from a dual perspective, the distance to neighboring level sets.
 After developing some calculus rules, we show how flat minima, conservation laws, and symmetries are intertwined. Gradient flows of conserved quantities are of particular interest, due to their flattening properties.
\end{adjustwidth} 
\vspace*{3mm}
\vspace*{3mm}


\tableofcontents

\section{Introduction}
\label{sec:Introduction}

Flat minima were informally introduced by Hochreiter and Schmidhuber \cite{hochreiter1994simplifying,hochreiter1997flat} in the context of deep learning as connected regions of minima where the objective is nearly constant. In a desire to formalize this idea, it was later suggested to use the volume of connected components of sublevel sets \cite[Definition 1]{dinh2017sharp}. Around the same time, an empirical measure of sharpness \cite[Metric 2.1]{keskar2017large} was proposed to analyze the role of batch size in the stochastic subgradient method for training neural networks. For nonnegative functions, it was redefined as the maximal ratio of the function variation over one plus the function value, over a ball of fixed radius \cite[Definition 2]{dinh2017sharp}.

More recently, flat minima of deep matrix factorization were defined as global minima which minimize the trace \cite{gatmiry2023inductive} or the maximum eigenvalue \cite{mulayoff2020unique,liu2021noisy,marion2024deep} of the Hessian of the objective function. A scaled trace of the Hessian tailored for matrix factorization \cite{ding2024flat} was also proposed. Others still use the trace of the Hessian combined with gradient dynamics for $C^4$ functions \cite[Definition 3]{ahn2024escape}.

Evidently, there is no commonly agreed upon definition of flatness. Yet, it has been reported that, when training deep neural networks, algorithms tend to find flat minima \cite{keskar2017large,wu2018sgd}. These may have good generalization properties \cite{hochreiter1997flat,andriushchenko2023modern}. Due to their possible role in deep learning theory, it seems enviable to have a definition that is not problem specific and somehow captures the previous ones. It should also fill in the gaps left by prior work. 

As it stands, being a flat global minimum of an overparametrized ReLU neural network has no meaning. To make things concrete, consider
$$f(x) = (x_2 \mathrm{ReLU}(x_1)+x_3 - 1)\p 2$$
where $\mathrm{ReLU}(t) = \max\{0,t\}$. The objective function is differentiable at its global minima, so its gradient is defined and equal to zero (details in \cref{eg:nn}). But it is not twice differentiable there, so its Hessian is not defined. As for the volumes of sublevel sets, they are generally not finite on ReLU networks \cite[Theorem 2]{dinh2017sharp}, nor on the most simple neural network, i.e., $f(x) = (x_1x_2-1)\p 2$.

Even when the Hessian is available, flatness is currently ill-defined. Minimizing the maximal eigenvalue $\lambda_1 (\nabla \p 2 f(x))$ over the solution set arbitrarily discards higher-order variation. This is problematic. For example,
$$f(x) = x_2\p 2 +x_1\p2x_2\p 4$$
obeys $\nabla f(x_1,0)= (0,0)$ and $\lambda_1 (\nabla\p 2 f(x_1,0)) = 2$ for all $x_1 \in \R$, according to which all the global minima $(x_1,0)$ are allegedly flat. 
But factoring in the fourth-order growth actually suggests that $(0,0)$ is the unique flat minimum (see \cref{eg:4th}). This is not merely a theoretical matter, as gradient descent with sufficiently slowly diminishing step lengths converges to the origin. This can be seen in \cref{fig:level}, and is proved in the sequel \cite{josz2026implicit}.

\begin{figure}
    \centering
    \includegraphics[width=0.8\linewidth]{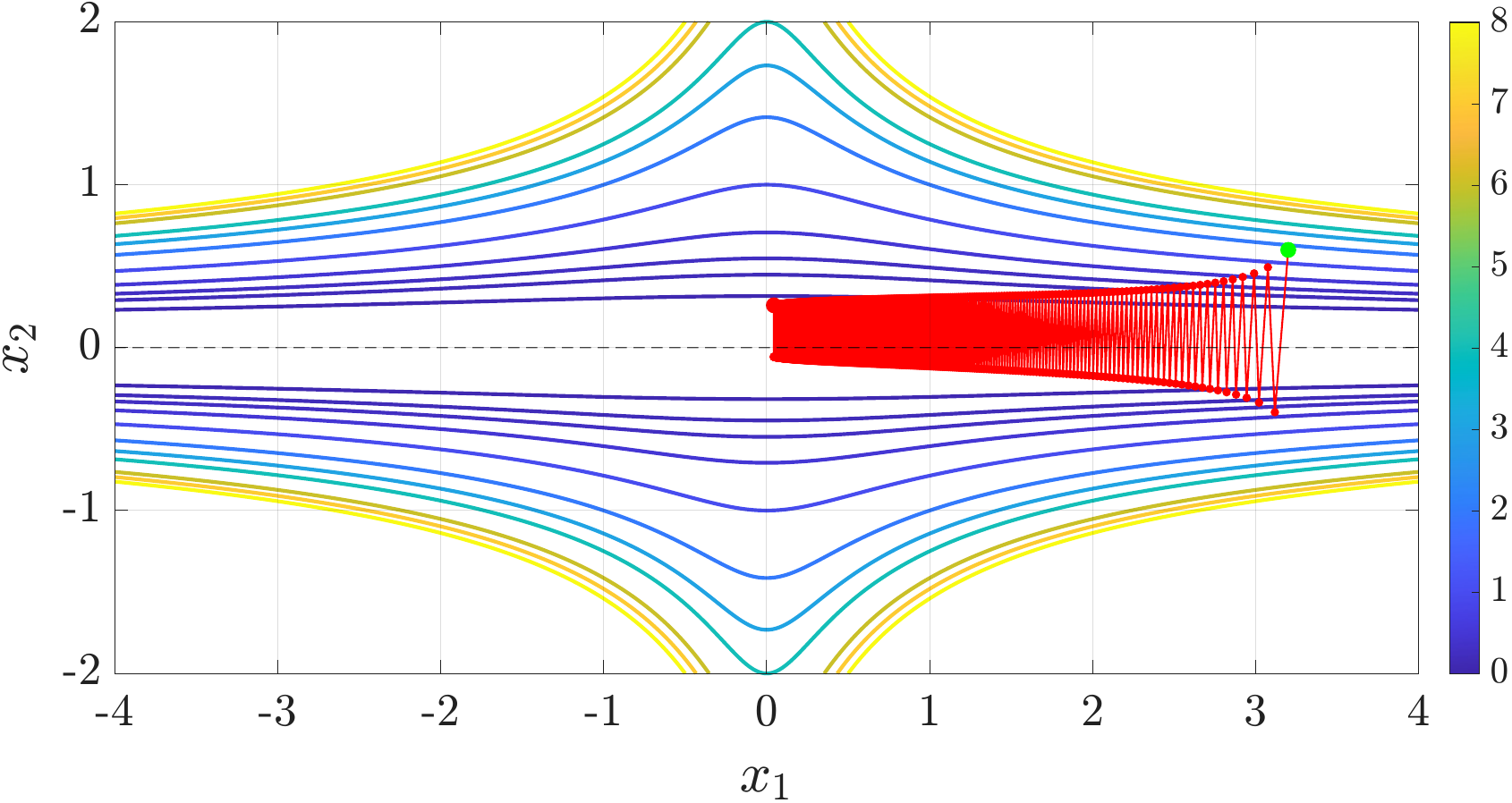}
    \caption{1000 iterations of gradient descent applied to $f(x_1,x_2)=x_2^2+x_1^2x_2^4$ with step length $(k+1)\p{-1/6}$ initialized at $(3.2,0.6)$.}
    \label{fig:level}
\end{figure}

Leaving algorithmic aspects aside, in this work we propose to define flatness by measuring the maximal variation around a point, or equivalently, for a large function class\footnote{The class of locally Lipschitz definable functions with nowhere dense level sets. If one is solely interested in the flatness of global minima, it suffices for the set of global minima to be nowhere dense.}, the distance to neighboring level sets.
It is inspired by topographic maps used in mountain hiking: concentrated contour lines signal significant elevation change, while diffuse ones suggest a flat region like a valley. We actually define a preorder on $\R\p n$, which is a total preorder if the objective is definable in an o-minimal structure on the real field \cite{van1998tame}. This gives a precise meaning to the adjectives flatter and sharper.

At a glance, our results are as follows. We first cover basic aspects: 
the definition of flatness, the properties of maximizing curves, and several calculus rules. The rules rely on the subgradient, gradient, Hessian, and possibly higher-order derivatives. 
Notably, for a local minimum $\overline{x}$ of a smooth function $f$, $\overline{x}$ is flat only if it is a local minimum of $\lambda_1 (\nabla \p 2 f(x))$ subject to $f(x)=f(\overline{x})$. The converse holds if it is a strict local minimum of the constrained problem.

Second, we show how conserved quantities $c(x)$ in subgradient dynamics
$$\dot{x} \in -\overline{\partial} f(x)$$
provide a useful tool for analyzing flatness, a theme of recent interest \cite{zhao2023symmetries}. On the one hand, they enable one to detect situations where level sets are expanding in a certain direction. In that case, gradient trajectories of the conserved quantity, modeled by
$$\dot{x} = -\nabla c(x),$$
flatten over time. In particular, if $x(0)$ is a local minimum of $f$ and $f$ is smooth, then $$\lambda_1 (\nabla \p 2 f(x(t)))\leq e\p{-\omega t}\lambda_1 (\nabla \p 2 f(x(0)))$$ for some constant $\omega >0$. On the other hand, quadratic conserved quantities $C(x)$ arising from linear symmetries of $f$ can help characterize flatness. For instance, if $x$ is a flat minimum of a smooth function $f$, then there exists a maximal eigenvector $v$ of $\nabla\p 2 f(x)$ such that $C(v) = 0$. 


Third, we determine flat minima in a series of examples that were out of reach with previous definitions and tools. They demonstrate that the definition of flatness proposed in this work is both versatile and workable. We also illustrate the connection with symmetry and conservation. 

This work clarifies the picture in overparamatrized matrix factorization
    \begin{equation*}
    \begin{array}{cccc}
         f: & \mathbb{R}^{m\times r}\times \mathbb{R}^{r\times n} & \longrightarrow & \mathbb{R} \\
        & (X,Y) & \longmapsto & \|XY-M\|_F\p 2
    \end{array}
    \end{equation*}
where $M \in \mathbb{R}^{m\times n}$, $\|\cdot\|_F$ is the Frobenius norm, $r\geq \rank(M)$, and $C(X,Y) = X\p T X = Y Y\p T$. Some authors \cite{ding2024flat} report the flat minima are ``balanced'', meaning that $X\p T X = Y Y\p T$, while others \cite{mulayoff2020unique} say that they are ``nearly balanced''. With our definition, we find that balanced global minima are flat, but the converse is false. While $X\p T X - Y Y\p T$ can be nonzero at flat minima, there exists a balanced maximal eigenvector of the Hessian $\nabla \p 2 f(X,Y)$. In passing, we note that flat minima admit a simple characterization. Namely, if $XY=M$, then
$$(X,Y) ~\text{is flat} ~~~\Longleftrightarrow ~~~ \|X\|_2=\|Y\|_2 = \sqrt{\|M\|_2}$$
where $\|\cdot\|_2$ is the spectral norm. 

Our analysis also reveals a new local-global property of matrix factorization. While it is known the every local minimum is a global minimum \cite{baldi1989neural,valavi2020revisiting}, it was not known that locally flat minima are globally flat. A related novelty is that any gradient trajectory of the conserved quantity $\|C\|_F\p 2$ initialized at a global minimum converges to a flat global minimum.

The paper is organized as follows. \cref{sec:Background} contains background material on several branches of mathematics.
\cref{sec:Basic aspects} treats basic aspects of flat minima. \cref{sec:Flatness and conservation} builds on them to forge a link with conservation laws and symmetries. 
Finally \cref{sec:Examples} provides several examples of flat minima.

\section{Background}
\label{sec:Background}

This works relies on several branches of mathematics. We include a vignette of each one after introducing some standard notations. As usual, 
\begin{gather*}
    \mathbb{N} = \{0,1,\hdots\}, ~~~ \N\p* = \N\setminus \{0\}, ~~~\llbracket 1,k\rrbracket = \{1,\hdots,k\}, \\
    \overline{\R}=\R\cup\{\infty\},~~~ \R_-=[0,\infty),~~~ \R_+=[0,\infty).
\end{gather*} 
The sign of a real number $t$ is defined by 
 \begin{equation*}
     \mathrm{sign}(t) = 
     \left\{
     \begin{array}{cl}
          t/|t| & \text{if} ~ t\neq 0, \\
          \left[-1,1\right] & \text{else.} 
     \end{array}  
     \right.
\end{equation*}
The symbol $\land$ means `and', $\lor$ means `or', and $\neg$ means `negation'. Let $\langle \cdot,\cdot\rangle$ and $|\cdot|$ respectively denote the dot product and the Euclidean norm on $\R\p n$. Let $B_r(x)$, $\overline{B}_r(x)$, and $S_r(x)$ respectively denote the open ball, closed ball, and sphere of center $x\in \R\p n$ and radius $r \geq 0$. In particular, $B\p n = \overline{B}_1(0)$ and $S\p {n-1} = S_1(0)$.

Given a matrix $M\in\R\p{m\times n}$, $M\p T$ denotes the transpose. If $A:U\to V$ is a linear map between two finite dimensional inner product spaces, then $A\p *:V\to U$ denotes the adjoint. Let $\|\cdot\|_F$, $\|\cdot\|_1$, $\|\cdot\|_2$, $\|\cdot\|_*$, $\|\cdot\|_{p,q}$, and $\rho(\cdot)$ respectively denote the Frobenius norm, entrywise $\ell_1$-norm, spectral norm, nuclear norm, $(p,q)$-induced norm, and spectral radius. Also, $\langle \cdot,\cdot\rangle_F$ denotes the Frobenius norm. Given a symmetric matrix $M \in \R\p{n\times n}$, $\lambda_1(M)$ and $E_1(M)$ respectively denote its maximal eigenvalue and its associated eigenspace, whose nonzero elements are referred to as maximal eigenvectors. We use $M\succcurlyeq 0$ to denote that $M$ is positive semidefinite.

\subsection{Ordered sets}
\label{subsec:Ordered sets}
    A binary relation $R$ on a set $X$ is a subset of $X\times X$. An element $x\in X$ is related to $y\in X$, denoted $xRy$, if $(x,y)\in R$. A relation $R$ is 
    \begin{enumerate}[label=\rm{(\roman{*})}]
        \item reflexive if $\forall x\in X, ~ xRx$; \vspace{-2mm}
        \item irreflexive if $\forall x\in X, ~ \neg xRx$; \vspace{-2mm}
        \item transitive if $\forall x,y,z\in X, ~ xRy \land yRz \Longrightarrow xRz$; \vspace{-2mm}
        \item antisymmetric if $\forall x,y\in X, ~ xRy \land yRx \Longrightarrow x = y$; \vspace{-2mm}
        \item total if $\forall x,y\in X, x\neq y \Longrightarrow x\leq y \lor y\leq x$.
    \end{enumerate}
    A preorder $\leq$ on $X$ is a binary relation that is reflexive and transitive \cite[Definition 3.1]{harzheim2005ordered}. Let $\nleq$ denote the complementary relation, i.e., $(X\times X) \setminus \leq$. A strict preorder $<$ on $X$ is a binary relation that is irreflexive and transitive. A preorder $\leq$ induces a strict order $<$ defined by $x<y ~ \Longleftrightarrow ~  x\leq y\land y\nleq x$ \cite[Definition 3.2]{harzheim2005ordered}. An order $\leq$ on $X$ is a preorder that is antisymmetric. In that case, $x\leq y\land y\nleq x ~ \Longleftrightarrow ~ x\leq y\land x\neq y$.
   
   Let $S$ be a subset of an ordered set $X$. An element $y\in X$ is a least (resp. greatest) element of $S$ if $y \in S$ and $x \leq y$ (resp. $y\leq x$) for all $x \in S$. A lower (resp. upper) bound of $S$ is an element $y\in X$ such that $y\leq x$ (resp. $x\leq y$) for all $x \in S$. A lower bound of $S$ is an infimum of $S$ if it is a greatest lower bound. If an infimum exists, then it is unique. An infimum is a minimum if it belongs to $S$. The set of infima and minima are respectively denoted $\inf S$ and $\min S$. Suprema and maxima are defined analogously, and denoted $\sup S$ and $\max S$ respectively. 
   
    A totally ordered set is complete if every nonempty subset that has a lower bound, has a greatest lower bound. The real numbers $\R$ is such an example. Let $S \subseteq \R$. If $S$ has no lower bound, then $\inf S = -\infty$. If $S=\emptyset$, then $\inf S = \infty$. 

    \subsection{Differential calculus}
\label{subsec:Differential calculus}

Given some vector spaces $V_1,\hdots,V_k$ and $W$, a map $F:V_1 \times \cdots \times V_k \rightarrow W$ is multilinear if it is linear in each variable taken separately when the others are held fixed. It is symmetric if $V_1=\cdots=V_k$ and $F(x_{\sigma(1)},\hdots,x_{\sigma(k)}) = F(x_1,\hdots,x_k)$ for any permutation $\sigma$ of $\llbracket 1,k\rrbracket$. When it is real-valued, i.e., $W = \mathbb{R}$, $F$ is called a multilinear form. Given a symmetric multilinear map $F:V\times \cdots\times V\to W$ and norms $\|\cdot\|_V,\|\cdot\|_W$ on $V,W$, consider the norm \cite[Theorem A]{bochnak1971polynomials}
$$\|F\| = \sup_{\|v\|_V = 1} \|F(v,\hdots,v)\|_W.$$

A function $f:U \to W$ where $U\subseteq V$ is open is Fr\'echet differentiable \cite{cartan1971differential} at $\overline{x}\in U$, or $D\p 1$ at $\overline{x}$, if there exists a bounded linear map, denoted $f'(\overline{x}):V\to W$, such that
$$f(x) = f(\overline{x}) + f'(\overline{x})(x-\overline{x}) +o(|x-\overline{x}|)$$
where $o:\R_+\to W$ is a function such that $o(t)/t \to 0$ at $t\searrow 0$. We say that $f$ is $C\p 1$ at $\overline{x}$ if $f$ is $D\p 1$ near $\overline{x}$ and $f'$ is continuous at $\overline{x}$. By induction, for any integer $k\geq 2$, we may define $f$ to be $D\p k$ (resp. $C\p k$) at $\overline{x}$ if it is $D\p {k-1}$ (resp. $C\p {k-1}$) near $\overline{x}$ and $f\p{(k-1)}$ is $D\p 1$ (resp. $C\p 1$) at $\overline{x}$. We only define $D\p k$ and $C\p k$ for $k\in \N\p *$. We say $f$ is $D\p k$ (resp. $C\p k$) if it is so at every point in $U$. Also, $f$ is $C\p{k,\ell}$ (resp. $C\p{k,\ell}_L$) if it is $C\p k$ and $f\p{(\ell)}$ is locally Lipschitz continuous (resp. $L$-Lipschitz continuous). If $f$ is $D\p k$ at $\overline{x}$, then $f\p {(k)}$ is a symmetric multilinear map \cite[Proposition C.16]{gilbert2021fragments}. When evaluated on the diagonal, we use the shorthand $f\p{(k)}(x)v\p k = f\p{(k)}(x)(v,\hdots,v)$. 

If $f:\R\p n\to \R\p m$ is $D\p 1$, then we identify $f'$ with the Jacobian $(\partial f_i/\partial x_j)_{i,j}$.
When $f:\R\p n\to \R$ is $D\p 1$ at $\overline{x}$, then there exists a unique vector $\nabla f(\overline{x})\in \R\p n$ such that 
$f'(\overline{x})(v) = \langle \nabla f(\overline{x}),v\rangle$ by the Riesz-Fr\'echet representation theorem. Likewise, there exists a unique matrix $\nabla\p 2 f(\overline{x})\in\R\p{n\times n}$ such that $f''(\overline{x})(v_1,v_2) = \langle \nabla\p 2 f(\overline{x})v_1,v_2\rangle
$ for all $v_1,v_2\in \R\p n$. Given $v_2,\hdots,v_n\in \mathbb{R}^n$, we let $\nabla^k f(\overline{x})(v_2,\hdots,v_n) \in \mathbb{R}^n$ denote the unique vector such that $f^{(k)}(\overline{x})(v_1,v_2,\hdots,v_n) = \langle v_1 , \nabla^k f(\overline{x})(v_2,\hdots,v_n) \rangle$ for all $v_1 \in \mathbb{R}^n$.
When $f:\mathbb{R}^n\rightarrow \R$ is $D^2$ at $\overline{x}$, $f''(\overline{x})$ is a symmetric bilinear form and
\begin{equation*}
    \|f''(\overline{x})\| = \max_{|v|=1} |\langle \nabla^2 f(\overline{x})v , v \rangle| = \max_{|v|=1} |\nabla^2 f(\overline{x}) v| = \|\nabla^2 f(\overline{x})\|_2 = \rho(\nabla^2 f(\overline{x})).
\end{equation*}
If $\overline{x}$ is a local minimum of $f$, then $\|f''(\overline{x})\|= \lambda_1(\nabla^2 f(\overline{x}))$ since $\nabla^2 f(\overline{x})$ is positive semidefinite.

\subsection{Differential geometry}
\label{subsec:Differential geometry}
    An action of a group $G$ with identity $e$ on a set $M$ \cite[p. 161]{lee2012smooth} is a map $\theta:G \times M \rightarrow M$ such that 
\begin{enumerate}[label=\rm{(\roman{*})}]
    \item $\forall g,h\in G, ~ \forall x\in M, ~ \theta(g,\theta (h,x)) = \theta(gh,x)$,
    \item $\forall x\in M, ~ \theta(e,x) = x$.
\end{enumerate}
When an action exists, $M$ is referred to as a $G$-space.
    It is homogeneous if for all $x,y\in M$, there exists $g\in G$ such that $\theta(g,x)=y$. A function $f:M\rightarrow N$ between sets $M$ and $N$ is invariant under an action of a group $G$ on $M$ if $f(\theta(g,x)) = f(x)$ for all $g\in G$ and $x\in M$.
    
    A Lie group $G$ is a smooth manifold and a group whose operations are smooth. A Lie group $G$ acts smoothly on a smooth manifold $M$ if there exists a smooth action $\theta:G \times M \rightarrow M$. A Lie subgroup of a Lie group $G$ is a subgroup of $G$ is endowed with a topology and smooth structure making into a Lie group and an immersed submanifold.
    Topologically closed subgroups of Lie groups are Lie subgroups by the closed subgroup theorem \cite[Theorem 20.12]{lee2012smooth}. Let $\fg$ denote the Lie algebra of a Lie group $G$, which we identify with its tangent space at $e$.
    
    Let $I_n$ denote the identity matrix of order $n$. The set of invertible matrices with real coefficients of order $n$, denoted  $\GL(n)$, is a Lie group. The natural action of a Lie subgroup $G$ of $\GL(n)$ on $\R\p n$ is defined by the matrix vector multiplication $G\times \R\p n \ni (g,x)\mapsto gx\in \R\p n$. The orthogonal group $\O(n) = \{Q\in\R\p{n\times n}: Q\p T Q = I_n\}$ is a Lie subgroup of $\GL(n)$. 
    
    Given a Lipschitz continuous function $F:\R\p n\to\R\p n$ and $x_0\in\R\p n$, consider the ODE
    \begin{equation*}
        \left\{\begin{array}{ccc}
             \dot{x} & = & F(x)  \\
             x(0) & = & x_0.
        \end{array}\right.
    \end{equation*}
Suppose it has a unique global solution $x:\R \to \R\p n$ for every initial point $x_0$. Then one can define the global flow $\theta: \R \times \R\p n \to \R\p n$ by $\theta(t,x_0) = x(t)$. If $F$ is $C\p k$, then $\theta$ is a $C\p k$ smooth action of $\R$ on $\R\p n$. In that case, $\theta_t = \theta(t,\cdot)$ is a $C\p k$ diffeomorphism with inverse $\theta_{-t}$ for any $t\in\R$.

\subsection{Variational analysis}
\label{subsec:Variational analysis}

Given $X\subseteq \R\p n$ and $f:X\to \overline{\R}$, a point $\overline{x} \in X$ where $f(\overline{x})$ is finite is a local minimum (resp. strict local minimum) of $f$ if there exists a neighborhood $U$ of $\overline{x}$ in $X$ such that $f(\overline{x})\leq f(x)$ (resp. $f(\overline{x}) < f(x)$) for all $x \in U\setminus\{\overline{x}\}$. Let $$\arg\loc\min_X f = \{ \overline{x} \in X: \overline{x}~ \text{is a local minimum of}~f\}.$$ In amounts to an abuse of notation, we will alternatively write this as $\arg\loc\min \{ f(x) : x \in X\}$.
A point $\overline{x} \in X$ where $f(\overline{x})$ is finite is a global minimum of $f$ if there exists a neighborhood $U$ of $\overline{x}$ in $\R\p n$ such that $f(\overline{x})\leq f(x)$ for all $x \in \R\p n$. Accordingly,
$\min_X f = \min\{ f(x) : x\in X\}$ and $$\arg\min_X f = \{ x \in X : f(x) = \min_X f\},$$ which we will also denote by
$\arg\min \{ f(x): x\in X\}$. A local minimum is spurious if it is not a global minimum.

It will be convenient to use the generalization of local optimality from points to sets proposed in \cite{joszneurips2018,josz2023certifying}. A nonempty set $S\subseteq \R\p n$ is a local minimum (resp. strict) of $f:\R\p n\to\R$ if there exists a neighborhood $U$ of $S$ such that $f(x)\leq f(y)$ (resp. $f(x)<f(y)$) for all $x\in S$ and $y\in U\setminus S$. In constrast to \cite{josz2023certifying}, we do not assume $S$ to be closed.

Given $x \in \mathbb{R}^n$ and $S \subseteq \R\p n$, let 
    $$d(x,S) = \inf \{|x-y|: y \in S\} ~~~ \text{and} ~~~
    P_S(x) = \arg\min \{ |x-y| : y \in S\}.$$ 
Given $f:\R\p n \to \overline{\R}$ and $\ell\in \R$, let $$[f = \ell] = \{x\in \R\p n : f(x) = \ell\}$$ and define other expressions like $[f \leq \ell]$ similarly. Let $\dom f = \{ x\in \mathbb{R}^n : f(x) < \infty\}$, 
$$\gph f=\{ (x,t)\in \mathbb{R}^n\times \R : f(x)=t\}, ~~~\text{and}~~~ \epi f=\{ (x,t)\in \mathbb{R}^n\times \R : f(x)\leq t\}.$$
The indicator and the support function of a set $C \subseteq \mathbb{R}^n$ are respectively defined by 
\begin{equation*}
    \forall x\in \R\p n ,~ \delta_C(x) = \left\{
    \begin{array}{cl}
        0 & \text{if} ~ x \in C, \\
        \infty & \text{if} ~ x \notin C.
    \end{array}
    \right.~~~\text{and}~~~  \forall v\in \R\p n, ~ \sigma_C(v) = \sup_{w \in C} \langle v,w \rangle.
\end{equation*}
Given $C \subseteq \mathbb{R}^n$ and $\overline{x} \in C$, the tangent cone \cite[Definition 6.1]{rockafellar2009variational} is defined by
\begin{gather*}
     T_C(\overline x)=\{v\in \R^n:~\exists  x_k\xrightarrow[C]{} \overline{x}~,~\tau_k\searrow 0: (x^k-\overline x)/\tau^k\to v\}
\end{gather*}
where $x_k \xrightarrow[C]{} \overline{x}$ is a shorthand for $x_k \rightarrow \overline{x}$ and $x_k \in C$.
A function $f:\mathbb{R}^n\rightarrow\overline{\mathbb{R}}$ is lower semicontinuous at $\overline{x} \in \dom f$ if $\liminf_{x\rightarrow \overline{x}} f(x) \geq   f(\overline{x})$ \cite[Definition 1.5]{rockafellar2009variational}. It is lower semicontinuous if it is so at every point in its domain.
The regular normal cone and normal cone \cite[Definition 6.3]{rockafellar2009variational} are defined by
\begin{gather*}
    \widehat{N}_C (\overline{x}) = \{ v \in \mathbb{R}^n : \langle v , x - \overline{x}\rangle \leq  o(|x-\overline{x}|)~\text{for}~x\in C~\text{near}~\overline{x} \}, \\ 
    N_C(\overline{x}) = \{ v \in \mathbb{R}^n : \exists x_k \xrightarrow[C]{} \overline{x}~\text{and}~ \exists v_k \rightarrow v~\text{with}~v_k \in \widehat{N}_C (x_k) \},
\end{gather*}
where $x_k \xrightarrow[C]{} \overline{x}$ is a shorthand for $x_k \rightarrow \overline{x}$ and $x_k \in C$. Explicitly, the $o$ means that 
\begin{equation*}
    \limsup_{\scriptsize \begin{array}{c}x \xrightarrow[C]{} \overline{x}\\ x\neq \overline{x}\end{array}} \frac{\langle v , x - \overline{x}\rangle}{|x-\overline{x}|} \leqslant  0.
\end{equation*}
A set $C \subseteq \mathbb{R}^n$ is regular \cite[Definition 6.4]{rockafellar2009variational} at one of its points $\overline{x}$ if it is locally closed and $\widehat{N}_C(\overline{x}) = N_C(\overline{x})$. 
The polar set of $C\subseteq \R^n$ \cite[Definition 6.22]{bauschke2017convex} is defined by 
$$   C^* =\{v\in \R^n:~\forall w\in C,~\langle v,w\rangle\leq  0\}.$$
By \cite[Theorem 6.28]{rockafellar2009variational}, the relationship $ \widehat{N}_C (\overline{x})=T_C(\overline{x})^*$ holds.

Given $f:\mathbb{R}^n\rightarrow \overline{\mathbb{R}}$ and a point $\overline{x}\in\mathbb{R}^n$ where $f(\overline{x})$ is finite, the one-sided directional derivative (if it exists in $\R\cup\{\pm\infty\}$) and the subderivative for $w\in \R\p n$ are respectively given by
$$f'(\overline x;w) = \lim_{\tau \searrow 0} \frac{f(\overline x+\tau w)-f(\overline x)}{\tau} ~~~\text{and}~~~
df(\overline{x})(w) = \liminf_{\scriptsize \begin{array}{c}
     \tau \searrow 0  \\
     w' \to w 
\end{array}} \frac{f(\overline x+\tau w')-f(\overline x)}{\tau}. 
$$
According to \cite[Definition 7.20]{rockafellar2009variational}, $f$ is semidifferential at $\overline{x}$ for $w\in\R\p n$ if
$$\lim_{\scriptsize \begin{array}{c}
     \tau \searrow 0  \\
     w' \to w 
\end{array}} \frac{f(\overline x+\tau w')-f(\overline x)}{\tau}\in \R \cup \{\pm \infty\}.$$
If this holds for all $w\in \R\p n$, then $f$ is semidifferentiable at $\overline{x}$.
Given $f:\mathbb{R}^n\rightarrow \overline{\mathbb{R}}$ and a point $\overline{x}\in\mathbb{R}^n$ where $f(\overline{x})$ is finite, the regular subdifferential, subdifferential, horizon subdifferential \cite[Definition 8.3]{rockafellar2009variational}, and Clarke subdifferential of $f$ at $\overline{x}$ \cite[Definition 4.1]{drusvyatskiy2015curves} are respectively given by
\begin{gather*}
    \widehat{\partial} f (\overline{x}) = \{ v \in \mathbb{R}^n : f(x) \geq   f(\overline{x}) + \langle v , x - \overline{x} \rangle + o(|x-\overline{x}|) ~\text{near}~ \overline{x} \}, \\
    \partial f(\overline{x}) = \{ v \in \mathbb{R}^n : \exists (x_k,v_k)\in \gph\hspace*{.5mm}\widehat{\partial} f: (x_k, f(x_k), v_k)\rightarrow(\overline{x}, f(\overline{x}), v) \}, \\[1mm]
    \partial^\infty f(\overline{x}) = \{ v \in \mathbb{R}^n : \exists (x_k,v_k)\in \gph\hspace*{.5mm}\widehat{\partial} f: \exists \tau_k \searrow 0: (x_k, f(x_k), \tau_kv_k)\rightarrow(\overline{x}, f(\overline{x}), v) \}, \\[2mm]
    \overline{\partial} f(\overline{x}) = \overline{\mathrm{co}} [\partial f(\overline{x}) + \partial^\infty f(\overline{x})],
\end{gather*}
where ${\mathrm{co}}$ denotes the convex hull, and $\overline{\mathrm{co}}$ its closure.
The $o$ means that
\begin{equation*}
    \liminf_{\scriptsize \begin{array}{c}x \rightarrow \overline{x} \\ x\neq \overline{x}\end{array}} \frac{f(x) - f(\overline{x}) - \langle v , x - \overline{x} \rangle}{|x-\overline{x}|} \geq   0.
\end{equation*} 
A point $x\in \R\p n$ is critical (resp. Clarke critical) if $0\in \partial f(x)$ (resp. $0\in \overline{\partial} f(x)$). A function $f:\mathbb{R}^n \rightarrow \overline{\mathbb{R}}$ is regular \cite[Definition 7.25]{rockafellar2009variational} at $\overline{x}$ if $f(\overline{x})$ is finite and $\epi f$ is regular at $(\overline{x},f(\overline{x}))$ as a subset of $\mathbb{R}^{n+1}$. The Lipschitz modulus of a function $f:\R^n\to\R$ is defined by \cite[p. 354]{rockafellar2009variational}
\begin{equation*}
    \lip f(\overline{x}) = \limsup_{\scriptsize \begin{array}{c}x,y \rightarrow \overline{x}\\ x\neq y\end{array}} \frac{|f(x)-f(y)|}{|x-y|}.
\end{equation*}
If $f$ is Lipschitz continuous near $\overline{x}$, i.e., $\lip f(\overline{x}) < \infty$, then let
\begin{equation*}
    \overline{\nabla} f(\overline{x}) = \{ v \in \mathbb{R}^n : \exists x_k \xrightarrow[D]{} \overline{x} ~\text{with}~ \nabla f(x_k) \rightarrow v\}
\end{equation*}
where $D$ are the differentiable points of $f$. In what amounts to a slight abuse of notation, we may replace $D$ by some $D' \subseteq D$ such that $D \setminus D'$ has zero Lebesgue measure. Still assuming $\lip f(\overline{x}) < \infty$, by \cite[Theorem 9.13]{rockafellar2009variational} $\partial f(\overline{x})$ is nonempty and compact, and $\lip f(\overline{x}) = \max \{|v| : v \in \partial f(\overline{x})\}$. By \cite[Theorem 9.61]{rockafellar2009variational}, $\overline{\nabla} f(\overline{x})$ is a nonempty compact subset of $\partial f(\overline{x})$ and $\mathrm{co} \overline{\nabla} f(\overline{x}) = \mathrm{co}\hspace*{.3mm} \partial f(\overline{x}) = \overline{\partial} f(\overline{x})$. 
Thus $\arg\max \{ |v|: v\in \partial f(\overline{x})\}= \arg\max \{ |v|: v\in \overline{\nabla} f(\overline{x})\}$, using the following fact.

\begin{fact}
\label{fact:hull}
    For any set $X \subseteq\R\p n$, $\arg\max \{ |x| : x \in X \} = \arg\max \{ |x| : x \in \co X \}$.
\end{fact}
\begin{proof}
Let $\overline{x}\in \arg\max \{ |x| : x \in X \}$ and $x \in \co X$.
There exist finitely many $t_i \geq 0$ with $\sum_i t_i = 1$ and $x_i \in X$ such that $x = \sum_i t_i x_i$. Thus $|x| = \left|\sum_i t_i x_i\right|\leq \sum_i t_i |x_i| \leq \sum_i t_i |\overline{x}| = |\overline{x}|$ and $\overline{x} \in \arg\max \{ |x| : x \in \co X \}$. Conversely, let $\overline{x} \in \arg\max \{ |x| : x \in \co X \}$. There exist finitely many $t_i \geq 0$ with $\sum_i t_i = 1$ and $x_i \in X \subseteq \co X$ such that $\overline{x} = \sum_i t_i x_i$. If $\overline{x} \neq x_j$ for some $j$, then $t_j < 1$ and one obtains the contradiction
\begin{align*}
    |\overline{x}|\p 2 & = \left|(1-t_j)\sum_{i\neq j} \frac{t_i}{1-t_j} x_i + t_j x_j\right|\p 2 < (1-t_j)\left|\sum_{i\neq j} \frac{t_i}{1-t_j} x_i\right|\p 2 + t_j |x_j|\p 2 \\
    & \leq (1-t_j) \sum_{i\neq j} \frac{t_i}{1-t_j} |x_i|\p 2 + t_j |x_j|\p 2 = \sum_i t_i |x_i|\p 2 \leq \sum_i t_i |\overline{x}|\p 2 = |\overline{x}|\p 2
\end{align*}
using the strict convexity of $|\cdot|\p 2$. Thus $\overline x = x_j \in X$ for some $j$.
\end{proof}


\subsection{O-minimal structures}
\label{subsec:O-minimal structures}

O-minimal structures (short for order-minimal) were originally considered by van den Dries, Pillay, and Steinhorn \cite{van1984remarks,pillay1986definable}. They are founded on the observation that many properties of semi-algebraic sets can be deduced from a few simple axioms \cite{van1998tame}. Recall that a subset $A$ of $\mathbb{R}^n$ is semi-algebraic \cite{bochnak2013real} if it is a finite union of basic semi-algebraic sets, which are of the form 
    \begin{equation*}
        \{ x \in \mathbb{R}^n : f_1(x) > 0 , \hdots , f_p(x) > 0, f_{p+1}(x) = 0, \hdots , f_q(x) = 0\}
    \end{equation*}
    where $f_1,\hdots,f_q$ are polynomials with real coefficients. We adopt \cite[Definition p. 503-506]{van1996geometric} below.
    
\begin{definition}
\label{def:o-minimal}
An o-minimal structure on the real field is a sequence $S = (S_k)_{k \in \mathbb{N}}$ such that for all $k \in \mathbb{N}$:\vspace{-2mm}
\begin{enumerate}
    \item $S_k$ is a boolean algebra of subsets of $\mathbb{R}^k$, with $\mathbb{R}^k \in S_k$;\vspace{-2mm}
    \item $S_k$ contains the diagonal $\{(x_1,\hdots,x_k) \in \mathbb{R}^k : x_i = x_j\}$ for $1\leqslant  i<j \leqslant  k$;\vspace{-2mm}
    \item If $A\in S_k$, then $A\times \mathbb{R}$ and $\mathbb{R}\times A$ belong to $S_{k+1}$;\vspace{-2mm}
    \item If $A \in S_{k+1}$ and $\pi:\mathbb{R}^{k+1}\rightarrow\mathbb{R}^k$ is the projection onto the first $k$ coordinates, then $\pi(A) \in S_k$;\vspace{-2mm}
    \item $S_3$ contains the graphs of addition and multiplication;\vspace{-2mm}
    \item $S_1$ consists exactly of the finite unions of open intervals and singletons. 
\end{enumerate}
\end{definition}

A subset $A$ of $\mathbb{R}^n$ is definable in an o-minimal structure $(S_k)_{k\in\mathbb{N}}$ if $A \in S_k$ for some $k\in\mathbb N$. A function $f:\mathbb{R}^n\rightarrow\overline{\mathbb{R}}$ is definable in an o-minimal structure if $\gph f$ is definable in that structure. Throughout this paper, we fix an arbitrary o-minimal structure $(S_k)_{k\in\mathbb{N}}$.

A key property of univariate definable functions is that they satisfy the monotonicity theorem \cite[4.1]{van1996geometric}. It states that on bounded open intervals, for any $p\in \mathbb{N}$ there exist finitely many open subintervals where the function is $C^p$ and either constant or strictly monotone. The extension of the monotonicity theorem to multivariate functions is the cell decomposition theorem \cite[4.2]{van1996geometric}, which is used to prove the definable Morse-Sard theorem \cite[Corollary 9]{bolte2007clarke}. It asserts that that lower semi-continuous definable functions have finitely many Clarke critical values.

\section{Basic aspects}
\label{sec:Basic aspects}

We investigate two dual viewpoints on flatness. To each one naturally corresponds an optimal curve emanating from the point of interest. We study its properties in preparation of future sections and develop calculus rules.

\subsection{Definition of flatness}

Below are two ways to measure the variation of a function around a point.
\begin{definition}
\label{def:cfof}
    Given $f:\R\p n \to \R$, let $\cf,\of:\R\p n \times \R_+ \to \overline{\R}$ be defined by
$$\cf(x,r) = \sup_{\overline{B}_r(x)} |f-f(x)| ~~~\text{and} ~~~ \of(x,\ell) = d(x,[|f-f(x)|\geq\ell]).$$
\end{definition}

The following definition is the central tenet of this paper.

\begin{definition}
    \label{def:flat}
    Consider the binary relation on $\R\p n$ defined by
$$x \preceq y ~~~ \Longleftrightarrow ~~~  \exists \overline{r}>0: ~ \forall r \in (0,\overline{r}], ~~~ \cf(x,r) \leq \cf(y,r).$$ 
We say $x$ is flatter than $y$, or $y$ is sharper than $x$ if $x\prec y$. We say $\overline{x}$ is flat (resp. strictly flat) if there exists a neighborhood $U$ of $\overline{x}$ in $[f=f(\overline{x})]$ such that $\overline{x}\preceq x$ (resp. $\overline{x}\prec x$) for all $x \in U \setminus\{\overline{x}\}$. We say $\overline{x}$ is globally flat if $\overline{x}\preceq x$ for all $x \in [f=f(\overline{x})]$.
\end{definition}

A function $f:\R\p n\to\R$ is constant near $x\in\R\p n$ if it is constant on a neighborhood of $x$. Otherwise, $f$ is nonconstant near $x$, which happens iff $\cf(x,r)>0$ for all $r>0$ iff $x\notin \inte [f=f(x)]$. Given $x$, $f$ is constant iff $\of(x,\ell)=\infty$ for all $\ell>0$, and $f$ is continuous at $x$ iff $\of(x,\ell)>0$ for all $\ell>0$.
A set $X\subseteq \R\p n$ is nowhere dense if its closure has empty interior. If $f$ is continuous and $\alpha\in \R$, then $f$ is nonconstant near every point in $[f=\alpha]$ iff $[f=\alpha]$ is nowhere dense.

A function $\varphi:S\to\overline\R$ where $S\subseteq \R$ is increasing (resp. strictly increasing) if $\varphi(s)\leq \varphi(t)$ (resp. $\varphi(s)< \varphi(t)$) for all $s,t\in S$ such that $s<t$. It is monotone (resp. strictly monotone) if either $\varphi$ or $-\varphi$ is increasing (resp. strictly increasing). Given $x\in\R\p n$, let $\cf_x = \cf(x,\cdot)$ and $\of_x = \of(x,\cdot)$, both of which are increasing.


\begin{fact}
\label{fact:total}
    $\preceq$ is a preorder. If $f$ is definable, then it is a total preorder and
    $$\forall x,y\in \R\p n, ~~~ x \prec y ~~\Longleftrightarrow ~~ \exists \overline{r}>0: ~ \forall r \in (0,\overline{r}], ~~~ \cf(x,r) < \cf(y,r).$$
\end{fact}
\begin{proof}
    $\preceq$ is reflexive and transitive. Let $x,y\in\R\p n$. Since $\cf_x$ is increasing, either $\cf_x(r)=\infty$ for all $r>0$, or $\cf_x(r)<\infty$ for all $r\in(0,\overline{r})$. Likewise at $y$. In the infinite cases, $x$ and $y$ are related since either $\cf(x,r) = \cf(y,r)=\infty$ for all $r>0$, $\cf(x,r) < \cf(y,r)=\infty$ for all $r\in(0,\overline{r})$, or $\cf(y,r) < \cf(x,r)=\infty$ for all $r\in(0,\overline{r})$. 
    Otherwise, assume $\cf_x(r),\cf_y(r)<\infty$ for all $r\in(0,\overline{r})$. Since $f$ is definable, so is $\cf$. Consider the function $\varphi:(0,\overline{r})\to\R$ defined by $\varphi(r) = \cf(y,r)-\cf(x,r)$. By the monotonicity theorem \cite[4.1]{van1998tame}, $\varphi$ is either constant or strictly monotone on $(0,\overline{r})$, after possibly reducing $\overline{r}$. If it is constant, then there exists $c\in \R$ such that $\cf(x,r) = \cf(y,r)+c$ for all $r\in(0,\overline{r})$. Otherwise, without loss of generality, we may assume that it is strictly increasing. If $\lim_{r\searrow 0} \varphi(r)\geqslant 0$, then $\cf(x,r) < \cf(y,r)$ for all $r\in(0,\overline{r})$. Otherwise, there exists $\widehat{r}\in (0,\overline{r}]$ such that $\cf(y,r) < \cf(x,r)$ for all $r\in(0,\widehat{r})$. In all cases, $x$ and $y$ are related. The equivalence follows easily from the above.
\end{proof}

We now gather several useful facts when $f$ is continuous (i.e., Facts \ref{fact:sublevel}-\ref{fact:lsc}).

\begin{fact}
\label{fact:sublevel}
    $\forall (x,\ell) \in \R\p n \times \R_+,~ P_{[|f-f(x)|\geq\ell]}(x) = P_{[|f-f(x)|=\ell]}(x)~\land ~ \of(x,\ell)=d(x,[|f-f(x)|= \ell])$.
\end{fact}
\begin{proof}
    Let $y \in P_{[|f-f(x)|\geq\ell]}(x)$. If $|f(y)-f(x)| > \ell$, then by intermediary value theorem there exists $z\in (x,y)$ such that $|f(z)-f(x)| = \ell$. Since $|x-z|<|x-y|$ and $z \in [|f-f(x)|\geq\ell]$, this contradicts the optimality of $y$. Thus $|f(y)-f(x)| = \ell$ and $y \in P_{[|f-f(x)|=\ell]}(x)$. As for the second equality, $\of(x,\ell)\leq d(x,[|f-f(x)|= \ell])$ is obvious. Let $y_k\in\R\p n$ be such that $|y_k-x|\to \of(x,\ell)$ and $|f(y_k)-f(x)|\geq \ell$. By the intermediary value theorem, there exists $z_k\in[x,y_k]$ such that $|f(z_k)-f(x)|=\ell$. Hence $d(x,[|f-f(x)|= \ell])\leq |z_k-x|\leq |y_k-x|$ and so $d(x,[|f-f(x)|= \ell])\leq \of(x,\ell)$. 
\end{proof}

There exists a natural duality between $\cf$ and $\of$.
\begin{fact}
\label{fact:primal}
    $\forall (x,\ell)\in \R\p n \times \R\p +$, $\of(x,\ell) = \inf \{ r \geq 0: \cf(x,r) \geq \ell\}$.
\end{fact}
\begin{proof} 
    $\of(x,\ell) = d(x,[|f-f(x)|\geq\ell]) = \inf\{ r\geq 0: \exists y \in \overline{B}_r(x): |f(y)-f(x)|\geq\ell \} = \inf\{ r\geq 0: \sup_{y \in \overline{B}_r(x)} |f(y)-f(x)|\geq\ell \} = \inf \{ r \geq 0: \cf(x,r) \geq \ell\}$.
\end{proof}

\begin{fact}
\label{fact:dual}
    $\forall (x,r)\in \R\p n \times \R_+$, $\cf(x,r) = \sup \{ \ell \geq 0: \of(x,\ell) \leq r\}$.
\end{fact}
\begin{proof}
        $\cf(x,r) = \sup_{\overline B_r(x)} |f-f(x)| = \sup \{\ell \geq 0: \exists y \in \overline B_r(x), |f(y)-f(x)| \geq \ell \}
        = \sup \{ \ell \geq 0 : d(x,[|f-f(x)|\geq\ell]) \leq r\} = \sup \{ \ell \geq 0: \of(x,\ell) \leq r\}$.
\end{proof}

\cref{fact:primal} and \cref{fact:dual} actually hold without continuity if one uses strict relations, including in the definition of $\cf$ and $\of$, with the convention that $\cf(x,0)=\of(x,0)=0$.

\begin{fact}
\label{fact:equiv}
If $f$ is nonconstant near $x$, then
the following are equivalent:
\begin{enumerate}[label=\rm{(\roman{*})}]
    \item \label{item:cf} $\exists \overline{r}>0: ~ \forall r \in (0,\overline{r}], ~~ \cf(x,r) \leq \cf(y,r)$;
    \item \label{item:of} $\exists \overline{\ell}>0: ~ \forall \ell \in (0,\overline{\ell}], ~~ \of(x,\ell) \geq \of(y,\ell)$.
\end{enumerate}
\end{fact}
\begin{proof}
\ref{item:cf} $\Longrightarrow$ \ref{item:of} Let $\overline{\ell} = \cf(x,\overline{r})>0$ and $\ell\in(0,\overline{\ell}]$. By \cref{fact:primal},
    \begin{align*}
        \of(x,\ell) & = \inf \{ r \geq 0: \cf(x,r) \geq \ell\} \\
        & = \inf \{ r \in [0,\overline{r}]: \cf(x,r) \geq \ell\} \\
        & \geq \inf \{ r \in [0,\overline{r}]: \cf(y,r) \geq \ell\} \\
        & \geq \inf \{ r \geq 0: \cf(y,r) \geq \ell\}  =\of(y,\ell).
    \end{align*}
\ref{item:of} $\Longrightarrow$ \ref{item:cf} Since $f$ is nonconstant, $\infty > \of(x,\overline{\ell}) \geq \of(y,\overline{\ell})>0$ after possibly reducing $\overline{\ell}$. Let $\overline{r} = \cf(y,\overline{\ell})/2>0$ and $r\in(0,\overline{r}]$. Since $\of(x,\ell) \geq \of(x,\overline{\ell}) \geq \of(y,\overline{\ell}) = 2\overline{r} > r$ for all $\ell\in(\overline{\ell},\infty)$, by \cref{fact:dual},
    \begin{align*}
        \cf(x,r) & = \sup \{ \ell \geq 0: \of(x,\ell) \leq r\} \\
        & = \sup \{ \ell \in [0,\overline{\ell}]: \of(x,\ell) \leq r\} \\
        & \leq \sup \{ \ell \in [0,\overline{\ell}]: \of(y,\ell) \leq r\} \\
        & \leq \sup \{ \ell \geq 0: \of(y,\ell) \leq r\} =\cf(y,r).
    \end{align*}
\end{proof}

\begin{fact}
\label{fact:continuous}
    $\cf$ is continuous.    
\end{fact}
\begin{proof}
    Let $\R\p n \times \R_+ \ni (x_k,r_k)\to (x,r)$. By continuity, there exists $y_k \in \overline{B}_{r_k}(x_k)$ such that $\cf(x_k,r_k) = |f(y_k)-f(x_k)|$, namely, $|f(y_k)-f(x_k)| \geq |f(z)-f(x_k)|$ for all $z\in \overline{B}_{r_k}(x_k)$. By compactness, $y_k \to y$ after taking a subsequence. For all $\epsilon >0$, we eventually have $B_{r-\epsilon}(x) \subseteq\overline{B}_{r_k}(x_k)$ and so $|f(y_k)-f(x_k)| \geq |f(z)-f(x_k)|$ for all $z\in B_{r-\epsilon}(x)$. Passing to the limit yields $|f(y)-f(x)| \geq |f(z)-f(x)|$ for all $z\in B_{r-\epsilon}(x)$. As $\epsilon$ is arbitrary, we have $|f(y)-f(x)| \geq |f(z)-f(x)|$ for all $z\in B_{r} (x)$. By continuity of $f$, this holds for all $z\in \overline{B}_{r}(x)$ and thus $\cf(x_k,r_k) = |f(y_k)-f(x_k)| \to |f(y)-f(x)| = \cf(x,r)$.
\end{proof}

\begin{fact}
\label{fact:lsc}
     $\of$ is lower semicontinuous.    
\end{fact}
\begin{proof}
    Let $(x_k,\ell_k) \to (x,\ell)\in \R\p n \times \R_+$ be such that $r=\lim\inf \of(x_k,\ell_k) < \infty$. There exists $y_k \in \R\p n$ such that $|x_k-y_k|\to r$ and $|f(x_k)-f(y_k)|\geq \ell_k$ up to a subsequence. By compactness, $y_k \to y$ up to another subsequence. Since $f$ is continuous, we may pass to the limit: $|x-y|=r$ and $|f(x)-f(y)|\geq \ell$. Thus $\lim\inf \of(x_k,\ell_k) = r \geq \of(x,\ell)$.
\end{proof}

Since $\cf$ is continuous, one has equality in \cref{fact:primal}, namely, $\of(x,\ell) = \inf \{ r \geq 0: \cf(x,r) = \ell\}$. This holds by the intermediate value theorem as $\cf(x,0)=0$. In contrast, since $\of$ is merely lower semicontinuous, equality does not necessarily hold in \cref{fact:dual} (think of the Cantor function). Likewise, $\cf_x$ need not be strictly increasing. To obtain such properties and others, more assumptions are needed.
\begin{assumption}
    \label{assum:isolated}
    Let $f:\R\p n \to \R$ be locally Lipschitz and nonconstant near $x\in \R\p n$. Suppose $\R \setminus \{f(x)\}$ contains no Clarke critical value of $f$ reached in $\overline{B}_{\overline{r}}(x)$ for some $\overline{r}>0$. Let $\overline{\ell} = \cf(x,\overline{r})$.
\end{assumption}

\begin{proposition}
\label{prop:isolated}
    Under \cref{assum:isolated},
    \begin{enumerate}[label=\rm{(\roman{*})}]
\item \label{item:curve_1} $\cf_x$ and $\of_x$ are continuous and strictly increasing on $[0,\overline{r}]$ and $[0,\overline{\ell}]$ respectively.
\item \label{item:curve_2} $\forall (r,\ell) \in [0,\overline{r}]\times [0,\overline{\ell}],~  (\cf_x \circ \of_x)(r) = r$ and $(\of_x \circ \cf_x)(\ell) = \ell$.
\item \label{item:curve_3} $\forall (r,\ell) \in \gph \cf_x|_{[0,\overline r]}, ~\arg\max_{\overline{B}_r(x)} |f-f(x)| = \arg\max_{S_r(x)} |f-f(x)| = P_{[|f-f(x)|=\ell]}(x)$.
    \end{enumerate}
\end{proposition}
\begin{proof}
    \ref{item:curve_1} \& \ref{item:curve_2} Continuity of $\cf_x$ follows from \cref{fact:continuous}. Clearly $\cf_x$ is increasing on $[0,\overline{r}]$. It is in fact strictly increasing.  Indeed, suppose $\cf_x$ is constant near $r\in (0,\overline{r})$. If $\cf(x,r)=0$, then $f$ is constant near $x$, violating our assumption. Otherwise, by continuity there exists $y \in \overline{B}_r(x)$ such that $\cf(x,r) = |f(y)-f(x)|$, which must be a local maximum of $|f-f(x)|$. Since $\cf(x,r)>0$, $y$ is either a local minimum or a local maximum of $f$, and hence $0 \in \partial f(y) = \co \partial f(y) = \overline{\partial} f(y)$ or $0 \in \partial (-f)(y) \subseteq \overline{\partial} (-f)(y) = -\overline{\partial} f(y)$ by Fermat's rule \cite[Theorem 10.1]{lee2012smooth}. Since $\R \setminus \{f(x)\}$ contains no Clarke critical value of $f$ reached in $\overline{B}_{\overline{r}}(x)$, we have $f(y) = f(x)$. It follows that $\cf(x,r) = 0$, a contradiction. 

By \cref{fact:primal} and the discussion below \cref{fact:lsc}, $\cf_x(\of_x(\ell)) = \ell$ for all $\ell \in [0,\overline{\ell}]$. Thus for all $r\in [0,\overline{r}]$, $ \cf_x((\of_x\circ\cf_x)(r))= (\cf_x \circ \of_x)(\cf_x(r)) = \cf_x(r)$
and $(\of_x\circ\cf_x)(r) = r$. It follows that $\of_x$ is continuous and strictly increasing.

\ref{item:curve_3} Let $r \in [0,\overline{r}]$ and $y \in \arg\max_{\overline{B}_r(x)} |f-f(x)|$. As $\cf_x(r) = |f(y)-f(x)| \leq \cf(x,|y-x|)$ and $\cf_x$ is strictly increasing, $r \leq |x-y|$. Hence $y\in \arg\max_{S_r(x)} |f-f(x)|$. 
Let $\ell = \cf_x(r)$. $y \in \arg\max_{S_r(x)} |f-f(x)|$ iff $\cf_x(r) = |f(y)-f(x)|$ and $|x-y|=r$ iff $\cf_x(|x-y|) = \ell = |f(y)-f(x)| = \cf_x(r)$ iff $\of_x(\ell) = |x-y|$ and $|f(y)-f(x)|=\ell$ iff $y \in P_{[|f-f(x)|=\ell]}(x)$. 
\end{proof}



Flatness can be equivalently defined using $\of$ under minimal assumptions.

\begin{proposition}
    Suppose $f:\R\p n \to \R$ is continuous and $[f = \alpha]$ is nowhere dense for some $\alpha\in \R$. Then
    $$\forall x,y\in [f = \alpha], ~~~ x \preceq y ~~\Longleftrightarrow ~~ \exists \overline{\ell}>0: ~ \forall \ell \in (0,\overline{\ell}], ~~~ \of(x,\ell) \geq \of(y,\ell).$$
    If in addition, $f$ is locally Lipschitz definable, then
    $$\forall x,y\in [f = \alpha], ~~~ x \prec y ~~\Longleftrightarrow ~~ \exists \overline{\ell}>0: ~ \forall \ell \in (0,\overline{\ell}], ~~~ \of(x,\ell) > \of(y,\ell).$$
\end{proposition}
\begin{proof}
    The first point follows from \cref{fact:equiv}, while the second follows from Facts \ref{fact:total}-\ref{fact:continuous}, \cref{prop:isolated}, and the definable Morse-Sard theorem \cite[Corollary 9]{bolte2007clarke}. 
\end{proof}

\subsection{Curve selection}
\label{subsec:Curve selection}

\cite[Definition 3.1]{pham2020local}

With the above definitions in place, it is natural to try and select optimal curves. This will be useful later.

\begin{lemma}
\label{lemma:curve}
Under \cref{assum:isolated},
\begin{enumerate}[label=\rm{(\roman{*})}]
\item \label{item:curve_4} There exist $\gamma:[0,\overline{r}]\to \R\p n$, $\lambda:[0,\overline{r}]\to \R$, and $v:[0,\overline{r}]\to \R\p n$ such that 
 for all $r \in [0,\overline{r}]$,  
$ \gamma(r)=x + \lambda(r)v(r) \in \arg\max_{\overline{B}_r(x)} |f-f(x)|$, $|\gamma(r)-x|=r$, and $v(r) \in \overline{\partial} f(\gamma(r))$.
\item \label{item:curve_5} If $x$ is a local minimum of $f$, then $\lambda(r)>0$ and $v(r) \in -\partial (-f)(\gamma(r))$ for all $r\in (0,\overline{r}]$. If in addition, $f$ is regular, then one may choose $v(r) \in \overline\nabla f(\gamma(r))$ for all $r\in (0,\overline{r}]$.
\item \label{item:curve_6} If $f$ is definable, then $\gamma,\lambda,v$ can be made $C\p k$ definable on $(0,\overline{r}]$ for any $k \in \N$, such that $\gamma$ is $C\p 1$ on $[0,\overline{r}]$, and $v$ continuous on $[0,\overline{r}]$, after possibly reducing $\overline{r}$.
\item \label{item:curve_7} If $x$ is a local minimum of $f$ and $f$ is definable and differentiable, then $\lambda(r)>0$, $\cf_x$ is differentiable on $[0,\overline{r}]$, and $(\cf_x)'(r) = r/\lambda(r)$ for all $r \in (0,\overline{r}]$.
\end{enumerate}
\end{lemma}
\begin{proof}
\ref{item:curve_4} By \cref{prop:isolated} \ref{item:curve_3}, there exists $\gamma(r) \in S_r(x)$ such that $\cf(x,r) = |f(\gamma(r))-f(x)|$. The first-order optimality condition \cite[Theorem 8.15]{rockafellar2009variational} reads $0 \in \partial (-|f-f(x)|)(\gamma(r)) + N_{\overline{B}_r(x)}(\gamma(r))$. If $r = 0$, then let $\lambda(0) = 0$ and $v(0) \in \overline\partial f(x)\neq \emptyset$ by \cite[Theorem 9.13]{rockafellar2009variational}. Otherwise, there exist $w(r) \in \partial (s(r)f)(\gamma(r)) \subseteq \co \partial (s(r)f)(\gamma(r)) = \overline{\partial} (s(r)f)(\gamma(r)) = s(r) \overline{\partial} f(\gamma(r))$ with $s(r) = \mathrm{sign}(f(x)-f(\gamma(r)))$ and $\mu(r) \geq 0$ such that $w(r) + \mu(r)(\gamma(r)-x) = 0$. If $\mu(r) = 0$, then $w(r) = 0$ and so $f(\gamma(r))$ is a Clarke critical value of $f$. By assumption $f(\gamma(r)) = f(x)$ and thus $\cf_x(r) = 0$, i.e., $r = 0$, a contradiction. Thus $\mu(r)> 0$. We may hence define $\lambda(r) = -s(r)/\mu(r)$ and $v(r) = s(r) w(r)$.

\ref{item:curve_5} When $x$ is a local minimum of $f$, $s(r) = -1$. 
Since $\gamma(r)$ is a local minimum of $g = -f+\delta_{B_r(x)}$, $dg(\gamma(r))\geq 0$ by Fermat's rule \cite[Theorem 10.1]{rockafellar2009variational}. 
As $f$ is locally Lipschitz regular, 
by \cite[Theorem 9.16]{rockafellar2009variational} $f$ is semidifferentiable and $df(\cdot) = \sigma_{\partial f (\cdot)}$. Thus $-f$ is semidifferentiable and $d(-f) = -df$. Since $-f$ is semidifferentiable, $dg = d(-f+\delta_{B_r(x)}) = d(-f)+d\delta_{B_r(x)}$. Thus  $dg(\gamma(r)) = -\sigma_{\partial f (\gamma(r))}+\sigma_{N_{B_r(x)}(\gamma(r))}\geq 0$ and $\sigma_{\partial f (\gamma(r))} \leq \sigma_{N_{B_r(x)}(\gamma(r))}$, namely $\partial f (\gamma(r)) \subseteq N_{B_r(x)}(\gamma(r))= \R_+(\gamma(r)-x)$. But $\partial f$ is nonempty convex compact valued by \cite[Theorem 9.61]{rockafellar2009variational}, and so $\partial f (\gamma(r)) = [a(r),b(r)](\gamma(r)-x)$ for some $0< a(r)\leq b(r)$ (recall that $0\notin \overline{\partial}f(\gamma(r))$). Also $\co \overline\nabla f(\gamma(r)) = \partial f(\gamma(r))$, hence $b(r)(\gamma(r)-x) \in \nabla f(\gamma(r))$. Now, simply let $v(r) = b(r)(\gamma(r)-x)$ and $\lambda(r)=1/b(r)>0$.

\ref{item:curve_6} When $f$ is definable, $[0,\overline{r}] \ni r \rightrightarrows \arg\max_{\overline{B}_r(x)}|f-f(x)|$ is definable, so there exists a definable selection $\gamma$ \cite[4.5]{van1996geometric}. Consequently $[0,\overline{r}] \ni r \rightrightarrows \{ v \in \overline{\partial} f(\gamma(t)) : \exists \lambda \in \R : \gamma(t) = x + \lambda v \}$ is definable, yielding a definable selection $v$. Finally, $\lambda(r) = \langle \gamma(r)-x,v(r)\rangle /|v(r)|\p 2$ is definable. The monotonicity theorem ensures $\gamma,\lambda,v$ are $C\p k$ for any $k\in \N$ on $(0,\overline{r}]$ after possibly reducing $\overline{r}$. Also $|\gamma_i(r)|\leq |\gamma(r)| = r$ so $|\gamma'_i(r)| \leq 1$ on $(0,\overline{r}]$ after possibly reducing $\overline{r}$ and $|\gamma'(r)|=\sqrt{|\gamma'_1(r)|\p 2 + \cdots + |\gamma'_n(r)|\p 2} \leq \sqrt{n}$. Hence $\lim_{r\searrow 0}\gamma'(r)$ exists. By the mean value theorem, $\gamma_i'(0) = \lim_{r\searrow 0} (\gamma_i(r)-x_i)/r = \lim_{r\searrow 0}\gamma_i'(r)$. Finally, one can take $v(0)=\lim_{r\searrow 0} v(r)\in \overline{\partial} f(\gamma(0))$, where the limit exists because $v(r)$ is bounded by \cite[Theorem 9.13]{rockafellar2009variational} and \cite[Proposition 3 p. 42]{aubin1984differential}.

\ref{item:curve_7} With $-\nabla f(\gamma(r)) + \mu(r) (\gamma(r)-x)=0$, we have \begin{align*}
    \cf_x(r+\epsilon)-\cf_x(r) & = f(\gamma(r+\epsilon)) - f(\gamma(r)) \\
    & = \langle \nabla f(\gamma(r)),\gamma(r+\epsilon)-\gamma(r)\rangle + o(|\gamma(r+\epsilon)-\gamma(r)|) \\
    & = \mu(r) \langle \gamma(r)-x,\gamma(r+\epsilon)-\gamma(r)\rangle + o(|\gamma(r+\epsilon)-\gamma(r)|) \\
    & = \mu(r) (|\gamma(r+\epsilon)|\p 2-|\gamma(r)|\p 2)/2 + o((\mu(r)+1)|\gamma(r+\epsilon)-\gamma(r)|) \\
    & = \mu(r) ((r+\epsilon)\p 2-r\p 2)/2 + o((\mu(r)+1)|\gamma(r+\epsilon)-\gamma(r)|) \\
    & = \mu(r) (r\epsilon + \epsilon\p 2/2) + o((\mu(r)+1)|\gamma(r+\epsilon)-\gamma(r)|) \\
    & = r \mu(r) \epsilon + o(\epsilon).
\end{align*}
\end{proof}

\subsection{Calculus rules}
\label{subsec:Calculus rules}
While the definition of flatness is quite general, determining flat minima directly from the definition is not always easy.
We thus develop some simple calculus rules. We begin with the nonsmooth case.
\begin{lemma}
\label{lemma:regular}
    If $f:\R\p n \to \R$ is Lipschitz continuous and regular near $x\in \R\p n$, then 
        $$
    \lip f(x) = \max_{|u|=1} ~  \lim_{\tau \searrow 0} \frac{|f(x+\tau u)-f(x)|}{\tau} = \limsup_{\scriptsize \begin{array}{c} y \rightarrow x\\ y \neq x\end{array}} \frac{|f(y)-f(x)|}{|y-x|}.
$$
\end{lemma}
\begin{proof}
Since $f$ is Lipschitz continuous near $x$, by \cite[Theorem 9.13]{rockafellar2009variational} $\partial f(x)$ is nonempty and compact, and $\lip f(x) = \max \{|v| : v \in \partial f(x)\}$.
Since $f$ is Lipschitz continuous and regular near $x$, by \cite[Theorem 9.16]{rockafellar2009variational} we have 
$\lim_{\tau \searrow 0} |f(x+\tau u)-f(x)|/\tau = \max \{ \langle u,v \rangle : v \in \partial f(x) \}$.
Thus
$$
    \max_{|u|=1}   \lim_{\tau \searrow 0} \frac{|f(x+\tau u)-f(x)|}{\tau} = \max_{|u|=1}   \max_{v \in \partial f(x)}   \langle u,v \rangle 
     =  \max_{v \in \partial f(x)} \max_{|u|=1}   \langle u,v \rangle =  \max_{v \in \partial f(x)}  |v| = \lip f(x).$$
\end{proof}

\begin{proposition}
\label{prop:regular}
    If $f:\R\p n \to \R$ is regular near $x\in \R\p n$ and $\lip f(x) < \infty$, then 
        $
        \cf(x,r) = \lip f(x)r + o(r)$.
\end{proposition}
\begin{proof}
    On the one hand, 
    \begin{align*}
        \frac{\cf(x,r)}{r} & = \sup_{y \in B_r(x)\setminus \{x\}} \frac{|f(y)-f(x)|}{r}
        \leq \sup_{y \in B_r(x)\setminus \{x\}} \frac{|f(y)-f(x)|}{|y-x|} \\
        & \leq \sup_{y,z \in B_r(x), y \neq z} \frac{|f(y)-f(z)|}{|y-z|}
         = \lip f(x)+o(1)
    \end{align*}
    for $r > 0$ near $0$. On the other hand, let 
    $\overline{u} \in \arg\max_{|u|=1} \lim_{\tau \searrow 0} |f(x+\tau u)-f(x)|/\tau$.
    By \cref{lemma:regular}, we have 
    $\cf(x,r) \geq |f(x+ r \overline{u}) - f(x)| = \lip f(x) r + o(r)$
    for $r \geq 0$ near $0$.
\end{proof}

Without regularity, \cref{prop:regular} may fail.
\begin{example}
    If $f(x,y) = \min\{x\p2,|y|\}$, then $\lip f(0,0) = 1$ and yet $\cf((0,0),r) =(\sqrt{1+4r\p 2} - 1)/2 = r\p 2 + o(r\p 2) \neq r+o(r)$.
\end{example}

We next turn to the differentiable case.

\begin{proposition}
\label{prop:smooth}
    If $f:\R\p n \to \R$ is $D\p k$ at $x\in \R\p n$ and $f\p{(i)}(x) = 0$ for all $i\in \llbracket 1,k-1\rrbracket$, then $\cf(x,r) = \|f\p{(k)}(x)\| r\p k/k! +o(r\p k)$.
\end{proposition}
\begin{proof}
    Since
    $f(y) - f(x) = f\p{(k)}(x)(y-x)\p k/k! + o(|y-x|\p k)$, we have 
    $$|f(y) - f(x)| = \frac{1}{k!} |f\p{(k)}(x)(y-x)\p k| + o(|y-x|\p k)\leq \frac{1}{k!} \|f\p{(k)}(x)\| |y-x|\p k(1+o(1)).$$
    On the one hand,
    $$ \frac{\cf(x,r)}{r\p k} =  \sup_{y \in B_r(x) \setminus\{x\}} \frac{|f(y) - f(x)|}{r\p k}
        \leq \sup_{y \in B_r(x) \setminus\{x\}} \frac{|f(y) - f(x)|}{|y-x|\p k}
        \leq \frac{1}{k!} \|f\p{(k)}(x)\| + o(1).$$
    On the other hand, with $\overline{u} \in  \arg\max \{ |f\p{(k)}(x) u \p k| : |u|=1\}$, we have
    $ f(x+r\overline{u}) = f(x) + r\p k f\p{(k)}(x) \overline{u}\p k + o(r\p k)$
    for $r \geq 0$ near $0$ and thus
    $\cf(x,r) \geq |f(x+r \overline{u})-f(x)| = r\p k\|f\p{(k)}(x)\|(1+o(1))/k!$.
\end{proof}

\cref{prop:regular} and \cref{prop:smooth} admit the following corollaries.

\begin{corollary}
\label{cor:regular}
    Suppose $f:\R\p n \to \R$ is regular near $\overline{x} \in \R\p n$ and $\lip f(x) < \infty$ for all $x \in [f=f(\overline{x})]$ near $\overline{x}$. 
    \begin{enumerate}[label=\rm{(\roman{*})}]
    \item If $\overline{x}$ is flat, then $\overline{x}$ is a local minimum of $\lip f + \delta_{[f=f(\overline{x})]}$. \label{item:necessary_flat}
    \item If $\overline{x}$ is a strict local minimum of $\lip f + \delta_{[f=f(\overline{x})]}$, then $\overline{x}$ is strictly flat. \label{item:sufficient_flat}
\end{enumerate} 
\end{corollary}
\begin{proof}
    \ref{item:necessary_flat} There exists a neighborhood $U$ of $\overline{x}$ in $[f=f(\overline{x})]$ such that $$\forall x \in U,~ \exists \overline{r}>0: ~  \forall r \in (0,\overline{r}], ~~  \cf(\overline{x},r) \leq \cf(x,r),$$ and 
in particular $\cf(\overline{x},r)/r \leq \cf(x,r)/r$ for all $r\in (0,\overline{r}]$. Passing to the limit as $r\searrow 0$ yields $\lip f(\overline{x}) \leq \lip f(x)$ by \cref{prop:regular}.

\ref{item:sufficient_flat} There exists a neighborhood $U$ of $\overline{x}$ in $[f=f(\overline{x})]$ such that $\lip f(\overline{x}) < \lip f(x)$ for all $x \in U \setminus \{\overline{x}\}$. By \cref{prop:regular}, $\cf(x,r) = \lip f(x)r + o(r)$ for all $x \in U$ after possibly reducing $U$. Thus, for all $x \in U \setminus \{\overline{x}\}$, there exists $\overline{r}>0$ such that $\cf(\overline{x},r)< \cf(x,r)$ for all $r\in (0,\overline{r}]$.
\end{proof}

\begin{corollary}
\label{cor:smooth}
    Suppose $f:\R\p n \to \R$ is $D\p k$ near $\overline{x}\in \R\p n$ and for all $x \in [f=f(\overline{x})]$ near $\overline{x}$, $f\p{(i)}(x) = 0$ for all $i\in \llbracket 1,k-1\rrbracket$. 
    \begin{enumerate}[label=\rm{(\roman{*})}]
    \item If $\overline{x}$ is flat, then $\overline{x}$ is a local minimum of $\|f\p{(k)}\| + \delta_{[f=f(\overline{x})]}$.
    \item If $\overline{x}$ is a strict local minimum of $\|f\p{(k)}\| + \delta_{[f=f(\overline{x})]}$, then $\overline{x}$ is strictly flat.
\end{enumerate} 
\end{corollary}
\begin{proof}
    One argues as for \cref{cor:smooth} using \cref{prop:smooth}.
\end{proof}

When $\|f\p {(i)}\|$ are merely constant on $[f=f(\overline{x})]$, it is neither necessary nor sufficient for $\overline{x}$ to be flat that it be a local or strict local minimum of $\|f\p{(k)}\| + \delta_{[f=f(\overline{x})]}$. 
\begin{example}
\label{eg:flat_not_imply_strict}
    All the minima of $f(x) = x_1\p 2 + x_2\p 4(1+x_3\p 2)$ are flat even though $\|f\p {(1)}(0,0,x_3)\| = 0$, $\|f\p {(2)}(0,0,x_3)\| = 2$, $\|f\p {(3)}(0,0,x_3)\| = 0$, and $\|f\p {(4)}(0,0,x_3)\| = 24(1+x_3\p 2)$.    
\end{example}
\begin{proof}
    For all $x_3\in \R$ and $r\in[0,1/(1+x_3\p 2)]$, $\cf(0,0,x_3,r) = r\p 2$.
\end{proof}
\begin{example}
\label{eg:strict_not_imply_flat}
    $(0,0,1)$ is the unique flat minimum of $f(x) = x_1\p 2 + x_2\p 4(1+x_3\p 2) + x_1\p 6(1+(x_3-1)\p 2)$ yet $\|f\p {(1)}(0,0,x_3)\| = 0$, $\|f\p {(2)}(0,0,x_3)\| = 2$, $\|f\p {(3)}(0,0,x_3)\| = 0$, and $\|f\p {(4)}(0,0,x_3)\| = 24(1+x_3\p 2)$.
\end{example}
\begin{proof}
    For all $x_3\in \R$, there exists $\overline{r}>0$ such that $\cf(0,0,x_3,r) = r\p 2 + (1+(x_3-1)\p 2)r\p 6$ for all $r\in[0,\overline{r}]$.
\end{proof}

Notwithstanding, there is one situation where local optimality can be sufficient for flatness.

\begin{fact}
    \label{fact:orthogonal}
    If $f:\R\p n \to \R$ is invariant under the natural action of a Lie subgroup $G$ of $\O(n)$, then $\cf(x,r) = \cf(gx,r)$ for
    $(x,r,g)\in\R\p n \times \R_+ \times G$.
\end{fact}
\begin{proof}
$\cf(x,r) = \sup_{y\in \overline{B}_r(x)} |f(y)-f(x)| = \sup_{gy\in \overline{B}_r(gx)} |f(gy)-f(gx)| = \sup_{z\in \overline{B}_r(gx)} |f(z)-f(gx)| = \cf(gx,r)$.
\end{proof}

\begin{corollary}
\label{cor:orthogonal_reg}
        Let $f:\R\p n \to \R$ be locally Lipschitz regular and invariant under the natural action of a Lie subgroup $G$ of  $\O(n)$. Let $\overline{x}\in\R\p n$. If $G\overline{x}$ is a strict local minimum of $\lip f + \delta_{[f=f(\overline{x})]}$, then every point in $G\overline{x}$ is flat.
\end{corollary}
\begin{proof}
\cref{fact:orthogonal} implies that $x\preceq y$ and $y\preceq x$ for all $x,y\in G\overline{x}$. By \cref{prop:regular}, $\cf(x,r) = \lip f(x)r + o(r)$ for all $x \in \R\p n$. There exists a neighborhood $U$ of $G\overline{x}$ in $[f=f(\overline{x})]$ such that $\lip f(x)< \lip f(y)$ for all $x\in G\overline{x}$ and $y\in U \setminus G\overline{x}$. Hence, for any such $x$ and $y$, there exists $\overline{r}>0$ such that $\cf(x,r)< \cf(y,r)$ for all $r\in (0,\overline{r}]$.
\end{proof}

An analogous result holds in the differentiable case.

\begin{corollary}
\label{cor:orthogonal_smooth}
        Let $f:\R\p n \to \R$ be $D\p k$ and invariant under the natural action of a Lie subgroup $G$ of $\O(n)$. Let $\overline{x} \in\R\p n$. Suppose that for all $x \in [f=f(\overline{x})]$, $f\p{(i)}(x) = 0$ for all $i\in \llbracket 1,k-1\rrbracket$. 
        If $G\overline{x}$ is a strict local minimum of $\|f\p {(k)}\| + \delta_{[f=f(\overline{x})]}$, then every point in $G\overline{x}$ is flat.
\end{corollary}

When applying the calculus rules, a composite structure can help.

    \begin{fact}{\normalfont \cite[Theorem 10.6]{rockafellar2009variational}}
    \label{fact:chain_reg}
    Let $f = g\circ F$ where $g:\R\p m\to \R$ is locally Lipschitz regular and $F:\R\p n \to \R\p m$ is $C\p 1$. Then $f$ is regular and
    $$\forall x\in \R\p n,~~~ \partial f(x) = F'(x)\p * \partial g(F(x)).$$
    In particular, if $F(x)=0$ and $g=|\cdot|_1$, then $\lip f(x) = \|F'(x)\p *\|_{\infty,2}$.
    \end{fact}

    \begin{fact}
    \label{fact:chain_smooth}
    Let $f = g\circ F$ where $g:\R\p m\to \R$ and $F:\R\p n \to \R\p m$ are $D\p k$, and $g\p {(i)}(0) = 0$ for all $i\in \llbracket 1,k-1\rrbracket$. Let $x\in \R\p n$ be such that $F(x)=0$. Then
        $$\forall i\in \llbracket 1,k-1\rrbracket, ~ f\p {(i)}(x) = 0 ~~\land~~ f\p {(k)}(x)v\p k = g\p {(k)}(x)(F'(x)v)\p k.$$
    In particular, if $g=|\cdot|\p k/k$, then $\|f\p {(k)}(x)\| = \|F'(x)\|_2\p k$.
    \end{fact}
    
Sometimes the composite structure helps with subdifferentiation, but not with determining flat minima.

\begin{remark}
\label{rem:jacobian}
    Let $f = g\circ F$ where $g:\R\p m\to \R$ is $C\p 1$ such that $\nabla g(0)= 0$ and $F:\R\p n \to \R\p m$ is locally Lipschitz. The Jacobian chain rule \cite[Theorem 2.6.6]{clarke1990} yields $\overline{\partial} f(x) =  \overline{\partial} F(x) \p T \nabla g(F(x))$. Hence, if $F(x) = 0$, then $\overline{\partial} f(x) = \{0\}$. Since the Clarke subdifferential is a singleton, by \cite[Theorem 9.18]{rockafellar2009variational} $f$ is strictly differentiable \cite[Definition 9.17]{rockafellar2009variational} at $x$, $\nabla f(x) = 0$, and $\lip f(x) = 0$. 
    \end{remark}

Two additional rules are useful in the presence of symmetries.

\begin{fact}
    \label{fact:chain_orbit_lsc}
    Let $f:\R\p n\to \overline\R$ be lower semicontinuous and invariant under the natural action of a Lie subgroup of $\GL(n)$. For all $\overline{x}\in \dom f$ and $g\in G$, $\partial f(g^{-1}\overline{x}) = g^T\partial f(\overline{x})$.
\end{fact}
\begin{proof}
    Applying the chain rule \cite[Exercise 10.7]{rockafellar2009variational} to $f(x) = f(gx)$ at $g^{-1}\overline{x}$ for any $g\in G \subseteq \mathrm{GL}(n,\mathbb{R})$ yields the result.
\end{proof}

\begin{fact}
    \label{fact:chain_orbit_smooth}
    Let $f:\R\p n\to \R$ be $D\p k$ and invariant under the natural action of a Lie subgroup of $\GL(n)$. For all $\overline{x},v_1,\hdots,v_k\in\R\p n$,
    $f^{(k)}(g^{-1}\overline{x})(v_1,\hdots,v_{k}) = f^{(k)}(\overline{x})(gv_1,\hdots,gv_{k})$.
\end{fact}
\begin{proof}
    Differentiating $f(x) = f(gx)$ at $g^{-1}\overline{x}$ in the direction $v$ yields $f'(g^{-1}\overline{x})(v) = f'(\overline{x})(gv)$. Assuming that $f^{(i)}(g^{-1}\overline{x})(v_1,\hdots,v_{i}) = f^{(i)}(\overline{x})(gv_1,\hdots,gv_{i})$, deriving with respect to $\overline{x}$ in the direction $gv_{i+1}$ yields $f^{(i+1)}(g^{-1}\overline{x})(v_1,\hdots,v_{i+1}) = f^{(i+1)}(\overline{x})(gv_1,\hdots,gv_{i+1})$.
\end{proof}

\section{Flatness and conservation}
\label{sec:Flatness and conservation}

Having established basic properties of flat minima, we now show how conserved quantities in subgradient dynamics provide a useful tool for analyzing them in more detail.

\subsection{Flattening trajectories}

A curve is a function $x:I\to \R$ where $I$ is an interval of $\R$.

\begin{definition}
    A trajectory of $F:\R\p n \rightrightarrows\R\p n$ is an absolutely continuous curve $x:I\to\R\p n$ where $I$ is an interval of $\R$ such that
$$\forae t\in I, ~~~ x'(t) \in F(x(t)).$$
\end{definition}
If $F$ is singleton-valued and continuous, then its trajectories are $C\p 1$.
We refer to solutions to the differential inclusion $\dot{x}\in F(x)$ as trajectories of $F$ over an interval $I = [0,T)$ for some $T\in(0,\infty]$ or $I = [0,T]$ for some $T\in(0,\infty)$. A solution $x:I\to\R\p n$ is maximal if for any other solution $y:J\to\R\p n$ such that $I\subseteq J$ and $x(t)=y(t)$ for all $t\in I$, we have $I=J$. A solution is globally defined if $I = [0,\infty)$. 

\begin{definition}
    \label{def:flattening}
    Given $f:\R\p n\to \R$, a curve $x:I\to \R\p n$ is flattening, or flattens over time (resp. strictly), if $x(t)\preceq x(s)$ (resp. $x(t)\prec x(s)$) for all $s,t\in I$ such that $s<t$. A curve $x:I\to \R$ sharpens (resp. strictly) over time if $I\ni t\mapsto x(-t)\in\R\p n$ flattens over time (resp. strictly).
\end{definition}
We seek to construct such flattening curves using the following assumption. It is satisfied in \cref{eg:monomial}, \cite[Examples 4.4, 4.5]{josz2025lyapunov}, and \cite[Examples 4.2, 4.3, ..., 4.8]{josz2026implicit}.


\begin{assumption}
    \label{assum:conserved}
    Let $f:\R\p n \to \R$ be locally Lipschitz definable and $c:\R\p n \to \overline{\R}$ be $C\p 4$ on a bounded convex open set $U\subseteq \R\p n$ such that
    $$\exists \omega> 0:~ \forall x\in U,~ \forall v\in \overline{\nabla} f(x), ~~~ \langle \nabla c(x) , v \rangle = 0 ~~\text{and}~~ ~ \langle \nabla\p2 c(x) v , v \rangle \leq - \omega|v|\p 2.$$
\end{assumption}

The first-order condition in \cref{assum:conserved} induces conservation laws.

\begin{proposition}
    \label{prop:conserved_mutual}
    Let $f:\R\p n \to \R$ be locally Lipschitz definable and $c:\R\p n \to \overline{\R}$ be $D\p 1$ on an open set $U\subseteq \R\p n$ such that $\langle \nabla c(x) , v \rangle = 0$ for all
    $x\in U$ and $v\in \overline{\partial} f(x)$.
    \begin{enumerate}[label=\rm{(\roman{*})}]
    \item \label{item:conserved_1} If $x:I\to U$ is a trajectory of $\overline{\partial}f$, then $c$ is conserved. 
    \item \label{item:conserved_2} If $x:I\to U$ is a trajectory of $\nabla c$, then $f$ is conserved.
    \end{enumerate}
\end{proposition}
\begin{proof}
    \ref{item:conserved_1} Since $x(\cdot)$ is absolutely continuous, it is differentiable almost everywhere. Thus, for almost every $t\in I$, $(c\circ x)'(t) = \langle \nabla c(x),x'(t)\rangle = 0$ and $I \ni t\to c(x(t))$ is constant.
    
    \ref{item:conserved_2} Since $f$ is definable and Lipschitz continuous, by the chain rule in \cite[Corollary 5.4]{drusvyatskiy2015curves},
    for almost every $t \in I$, $$\forall v \in \overline{\partial} f(x(t)),~~~ (f\circ x)'(t) = \langle v,x'(t)\rangle = \langle v,\nabla c(x(t))\rangle = 0$$
and $I \ni t\to f(x(t))$ is constant.
\end{proof}

Under \cref{assum:conserved}, gradient trajectories of the conserved quantity $c$ strictly sharpen over time. This is because the level sets are contracting. To see why, we begin with a simple lemma.

\begin{lemma}
\label{lemma:c}
If $c:\R\p n \to \overline{\R}$ is $C\p 4$ on a bounded convex open set $U\subseteq \R\p n$, then there exists $M>0$ such that
$$\forall x,y \in U,~~~ \left|\langle \nabla c(x) - \nabla c(y) , x-y \rangle - \langle \nabla\p2 c(x) (x-y),x-y \rangle\right| \leq M|x-y|\p 3.$$    
\end{lemma}
\begin{proof}
    Consider the function $g:\R\p n \times \R\p n\to \R$ defined by $g(x,y) = \langle \nabla c(x) - \nabla c(y) , x-y \rangle$. Since $c$ is $C\p 4$ on $U$, $g$ is $C\p 3$ on $U$. By \cite[Lemma 1.2.4]{nesterov2003introductory}, there exists $M>0$ such that for all $x,y \in U$, $u \in U +\{-x\}$, and $v \in U +\{-y\}$ we have
$$\left|g(x+u,y+v) - g(x,y) - \langle \nabla g(x,y), (u,v) \rangle - \langle \nabla\p2 g(x,y) (u,v),(u,v) \rangle/2\right| \leq M(|u|\p 3+|v|\p 3).$$
In particular,
$$\left|g(x,x+v) - g(x,x) - \langle \nabla g(x,x), (0,v) \rangle - \langle \nabla\p2 g(x,x) (0,v),(0,v) \rangle/2\right| \leq M|v|\p 3.$$
Fix $x \in U$ and consider the function $h: \R\p n\to \R$ be defined by $h(y) = g(x,y) = c'(x)(x-y)-c'(y)(x-y)$. We have
$$\forall u \in \R\p n,~~~  h'(y)u = -c'(x)u-c''(y)(x-y,u)+c'(y)u,$$
$$\forall u,v \in \R\p n,~~~ h''(y)(u,v) = -c'''(y)(x-y,u,v)+c''(y)(v,u)+c''(y)(u,v).$$
In particular, $\nabla h(x) = 0$ and $\nabla \p 2 h(x) = \nabla\p 2 c(x)$. For all $x\in U$ and $v \in U +\{-x\}$, it follows that
$|\langle \nabla c(x+v) - \nabla c(x) , v \rangle - \langle \nabla\p2 c(x) v,v \rangle| \leq M|v|\p 3$, yielding the desired inequality.
\end{proof}

The following result is one of the main findings of this paper. 
\begin{theorem}
    \label{thm:contracting}
    Under \cref{assum:conserved}, for all $x_0 \in U$ near which $f$ is nonconstant, there exists a solution $x:[0,\overline{t})\to \R\p n$ to
    $$\left\{ \begin{array}{ccc}
        \dot{x} & = & \nabla c(x)  \\
        x(0) & = & x_0 
    \end{array} \right.$$
    such that the two equivalent conditions hold: 
    \begin{enumerate}[label=\rm{(\roman{*})}]
    \item $\exists \overline{r}>0, ~ \forall r \in [0,\overline{r}), ~ \forall t \in [0,\overline{t}), ~~~  \cf(x_0,r) \leq \cf(x(t),e\p {-\omega t/2} r)$,
    \item $\exists \overline{\ell}>0, ~ \forall \ell \in [0,\overline{\ell}), ~ \forall t \in [0,\overline{t}), ~~~  \of(x(t),\ell) \leq e\p {-\omega t/2} \of(x_0,\ell)$.
    \end{enumerate}
    \end{theorem}
\begin{proof}
Since $\nabla c$ is Lipschitz continuous on $U$, by \cite[Theorem D.1]{lee2012smooth}, there exist $\overline{t},\overline{r}>0$ such that the ODE 
$$\left\{ \begin{array}{ccc}
        \dot{z} & = & \nabla c(z)  \\
        z(0) & = & z_0
    \end{array} \right.$$
admits a unique $C\p 1$ solution $z: [0,\overline{t})\to U$ for all $z_0 \in B_{\overline{r}}(x_0)$. Since $\nabla c$ is $C\p 2$, by \cite[Theorem D.1]{lee2012smooth}, the semiflow $\theta:[0,\overline{t})\times B_{\overline{r}}(x_0) \to \R\p n$ defined by $\theta(t,z_0) = z(t)$ is $C\p 2$. By \cite[Lemma 1.2.4]{nesterov2003introductory}, there exists $C>0$ such that for all $t,\widetilde t \in [0,\overline{t})$ and $z_0,\widetilde z_0 \in B_{\overline{r}}(x_0)$, we have  
\begin{equation}
    \label{eq:theta}
    | \theta(t,z_0) - \theta(\widetilde t, \widetilde z_0) - \theta'(\widetilde t,\widetilde z_0)(t - \widetilde t, z_0 - \widetilde z_0)|\leq C\left[(t - \widetilde t)\p 2 + |z_0 - \widetilde z_0|\p 2\right].
\end{equation}

Since $f$ is locally Lipschitz, regular, definable, and nonconstant near $x_0$, by \cref{lemma:curve} there exist definable curves $\gamma:[0,\overline{r}]\to \R\p n$, $\lambda:[0,\overline{r}]\to \R$, and $v:[0,\overline{r}]\to \R\p n$ such that $\gamma$ is $C\p 1$ on $[0,\overline{r}]$ and 
 $$\forall r \in [0,\overline{r}], ~~ \gamma(r)= x_0 + \lambda(r)v(r) \in \arg\max_{\overline{B}_r(x_0)} |f-f(x_0)|,~~ |\gamma(r)-x_0|=r,~~ v(r) \in \overline{\nabla} f(\gamma(r)),$$
after possibly reducing $\overline{r}$.
As $c$ is $C\p 4$ on $U$, there exists $L>0$ such that
$$\forall x,y \in U,~~~|\nabla\p 2 c(x) - \nabla\p 2 c(y) |\leq L|x-y|.$$ Let $M>0$ be given by \cref{lemma:c}.
For all $(t,r) \in [0,\overline{t}) \times [0,\overline{r})$, let $x_r(t) = \theta(t,\gamma(r))$ and $x(t) = x_0(t) = \theta(t,\overline{x})$. As $\gamma$ is continuous, so is the composition $[0,\overline{t}) \times [0,\overline{r}) \ni (t,r)\to x_{r}(t)\in \R\p n$.
In particular, $|x_r(t)-x_0| \leq \omega/(8\max\{L,M\})$ for all $t \in [0,\overline{t})$ and $r \in [0,\overline{r})$
after possibly reducing $\overline{t}$ and $\overline{r}$. It follows that
$|x_r(t)-\gamma(r)| \leq |x_r(t)-x_0| + |x_0-\gamma(r)| \leq \omega/(8L)+\omega/(8L) = \omega/(4L)$
and
$|x_r(t)-x(t)| \leq |x_r(t)-x_0| + |x_0-x_0(t)| \leq \omega/(4M)$.

The main idea of the proof is now as follows. Using \cref{prop:conserved_mutual}, we have
$$ \cf(x_0,r) = \max_{\overline{B}_r(x_0)} |f-f(x_0)|
    = |f(x_r(0))-f(x(0))| 
     = |f(x_r(t))-f(x(t))|
    \leq \cf(x(t),e\p {-\omega t/2} r)$$
where the last inequality is due to $|x_r(t)-x(t)| \leq e\p {-\omega t/2}|x_r(0)-x(0)| = e\p {-\omega t/2} r$.
The step we need to justify is the contraction. To that avail, let $y_r:[0,\overline{t})\to \R\p n$ be defined by $y_r(t) = x_r(t) - x(t)$. Initially $y_r(0)= x_r(0)-x(0) = \gamma(r) - x_0 = \lambda(r)v(r) \neq 0$ for all $r \in (0,\overline{r})$. Since $v(r) \in \overline{\nabla} f(\gamma(r))$, by \cref{assum:conserved} we have $\langle \nabla\p 2 c(\gamma(r))v(r),v(r) \rangle \leq -\omega |v(r)|\p 2$. Together with $\lambda(r)\neq 0$, we find
$\langle \nabla\p 2 c(\gamma(r))\lambda(r)v(r),\lambda(r)v(r) \rangle \leq -\omega |\lambda(r)v(r)|\p 2$.
In other words,
\begin{equation}
\label{eq:initial}
\forall r \in (0,\overline{r}), ~~~ \left\langle \nabla\p 2 c(\gamma(r))\frac{y_r(0)}{|y_r(0)|},\frac{y_r(0)}{|y_r(0)|} \right\rangle \leq -\omega.
\end{equation}
By \eqref{eq:theta}, for all $(t,r) \in [0,\overline{t}) \times (0,\overline{r})$ we have
$\left| \theta(t,\gamma(r)) - \theta(t,x_0) - \theta'(t,x_0)(0, \gamma(r)-x_0)\right|\leq C|\gamma(r)-x_0|\p 2$
and thus
\begin{equation}
    \label{eq:gamma}
    \left| \frac{\theta(t,\gamma(r)) - \theta(t,x_0)}{|\gamma(r)-x_0|} - \frac{\partial \theta}{\partial z_0}(t,x_0)\frac{\gamma(r)-x_0}{|\gamma(r)-x_0|}\right|\leq C|\gamma(r)-x_0|.
\end{equation}
Each entry of $(\gamma-x_0)/|\gamma-x_0|$ is definable and bounded. By the monotonicity theorem \cite[4.1]{van1998tame}, it is monotone near $0$ and thus convergent. Thus there exists $u \in S\p {n-1}$ such that 
$(\gamma(r)-x_0)/|\gamma(r)-x_0| \to u$ as $r \searrow 0$.
Since $\theta$ is $C\p 1$, it follows that
$$\frac{\partial \theta}{\partial z_0}(t,x_0)\frac{\gamma(r)-x_0}{|\gamma(r)-x_0|} \xrightarrow[(t,r)\searrow (0,0)]{} \frac{\partial \theta}{\partial z_0}(0,x_0) u = u,$$
where the equality holds because $\theta(0,z_0) = z_0$ for all $z_0 \in B_{\overline{r}}(x_0)$. From \eqref{eq:gamma}, we deduce
$$\frac{y_r(t)}{|\gamma(r)-x_0|} = \frac{x_r(t)-x(t)}{|\gamma(r)-x_0|} = \frac{\theta(t,\gamma(r)) - \theta(t,x_0)}{|\gamma(r)-x_0|} \xrightarrow[(t,r)\searrow (0,0)]{} u$$
and $$\frac{y_r(t)}{|y_r(t)|} = \frac{y_r(t)/|\gamma(r)-x_0|}{|y_r(t)|/|\gamma(r)-x_0|} \xrightarrow[(t,r)\searrow (0,0)]{} u.$$
Since $\gamma$ and $\nabla c$ are continuous, passing to the limit in \eqref{eq:initial} yields $\langle \nabla\p 2 c(x_0) u, u \rangle \leq -\omega$. As a result,
$$\forall t \in [0,\overline{t}),~ \forall r \in (0,\overline{r}), ~~~ \left\langle \nabla\p 2 c(\gamma(r))\frac{y_r(t)}{|y_r(t)|},\frac{y_r(t)}{|y_r(t)|} \right\rangle \leq -\frac{\omega}{2}$$
after possibly reducing $\overline{t}$ and $\overline{r}$. Now
\begin{align*}
    \frac{d|y_r|\p 2}{dt} & = 2\langle \dot{y}_r , y_r \rangle \\
    & = 2\langle \dot{x}_r - \dot{x} , y_r \rangle \\
    & = 2\langle \nabla c(x_r) - \nabla c(x) , y_r \rangle \\
    & \leq 2\langle \nabla\p 2 c(x_r)y_r , y_r \rangle + 2M|y_r|\p 3 \\
    & = 2\langle \nabla\p 2 c(\gamma(r))y_r , y_r \rangle + 2\langle [\nabla\p 2 c(x_r)-\nabla\p 2 c(\gamma(r))]y_r , y_r \rangle + 2M|y_r|\p 3 \\
    & \leq 2\langle \nabla\p 2 c(\gamma(r))y_r , y_r \rangle + 2|\nabla\p 2 c(x_r)-\nabla\p 2 c(\gamma(r))||y_r|\p 2 + 2M|y_r|\p 3 \\
    & \leq 2\langle \nabla\p 2 c(\gamma(r))y_r , y_r \rangle + 2L |x_r-\gamma(r)||y_r|\p 2 + 2M|y_r|\p 3 \\
    & \leq -\omega |y_r|\p 2 + 2L |x_r-\gamma(r)||y_r|\p 2 + 2M|y_r|\p 3 \\
    & = -(2\omega - 2L |x_r-\gamma(r)|-2M|y_r|)|y_r|\p 2 \\
    & \leq -(2\omega - \omega/2- \omega/2)|y_r|\p 2 \\
    & \leq -\omega |y_r|\p 2.
\end{align*}
The ODE comparison theorem \cite[Theorem D.2]{lee2012smooth} implies that $|y_r(t)|\p 2 \leq e\p {-\omega t}|y_r(0)|\p 2$, namely,
$|x_r(t)-x(t)| \leq e\p {-\omega t/2}|x_r(0)-x(0)|$ for all $t \in [0,\overline{t})$ and $r \in [0,\overline{r})$. 
As for the equivalence, it is obtained via \cref{lemma:curve} and the change of variables $\ell = \cf(x_0,r)$: $\cf(x_0,r) \leq \cf(x(t),e\p{-\omega t/2}r)$ iff $\of(x(t),\cf(x_0,r)) \leq \of(x(t),\cf(x(t),e\p{-\omega t/2}r))$ iff $\of(x(t),\ell) \leq e\p{-\omega t/2}\of(x_0,\ell)$.  
\end{proof}

Due to the local nature of \cref{thm:contracting}, it does not guarantee that reversing time yields strictly flattening trajectories from a given initial point. In order to do so, we will add some assumptions and rely on the following fact.


\begin{fact}
    \label{fact:exp}
    Let $g:(a,b)\to\R$ be continuous such that for all $s\in (a,b) \subseteq \R$, there exists $s'\in (s,b)$ such that $g(t) \geq e\p {t-s} g(s)$ for all $t \in (s,s')$. Then $g(t) \geq e\p {t-s} g(s)$ for all $s,t\in (a,b)$.
\end{fact}
\begin{proof}
    Let $a<\overline{s}<\overline{t}<b$ and suppose $t_0 = \inf \{ t \in [\overline{s},\overline{t}] : g(t) < e\p {t-\overline{s}} g(\overline{s}) \}< \infty$. By assumption $t_0 > \overline{s}$. By definition of $t_0$, we have $g(t) \geq e\p {t-\overline{s}} g(\overline{s})$ for all $t\in [\overline{s},t_0)$. Since $g$ is continuous, $g(t_0) \geq e\p {t_0-\overline{s}} g(\overline{s})$. By assumption, there exists $t_1\in (t_0,b)$ such that $g(t) \geq e\p {t-t_0}g(t_0)$ for all $t \in (t_0,t_1)$, and so $g(t) \geq e\p {t-t_0} e\p {t_0-\overline{s}} g(\overline{s}) = e\p {t-\overline{s}} g(\overline{s})$. Hence $g(t) \geq e\p {t-t_0}g(t_0)$ for all $t \in [\overline{s},t_1]$, contradicting the optimality of $t_0$.
\end{proof}

\begin{corollary}
\label{cor:flattening}
Under \cref{assum:conserved}, let $x:I\to U$ be such that $\dot{x} = -\nabla c(x)$ on a compact interval $I$ of $\R$ containing $t_0$. Let $s,t\in I$ be such that $s<t$.
    \begin{enumerate}[label=\rm{(\roman{*})}]
        \item If $f$ is locally Lipschitz regular, then $\lip f(x(t)) \leq e\p{-\omega(t-s)/2} \lip f(x(s))$. \label{item:exp_reg}
        \item If $c$ is $C\p {k+1}$, $f$ is $D\p k$, and $f\p{(i)}(x(t_0))=0$ for all $i\in \llbracket 1,k-1\rrbracket$, then $\|f\p {(k)}(x(t))\| \leq e\p{-k\omega(t-s)/2} \|f\p {(k)}(x(s))\|$. \label{item:exp_smooth}
    \end{enumerate}
    If $\lip f(x(t_0))>0$ in \ref{item:exp_reg} or $\|f\p {(k)}(x(t_0))\| >0$ in \ref{item:exp_smooth}, then $x(t) \prec x(s)$.
\end{corollary}
\begin{proof}
\ref{item:exp_reg} Let $t\in I$. By \cref{thm:contracting} and \cref{prop:regular},
$$\lip f(x(t)) + o(1)=\cf(x(t),r)/r \leq \cf(x(s),e\p {-\omega (t-s)/2} r)/r = e\p {-\omega (t-s)/2}\lip f(x(s)) + o(1)$$ for all $s \in I \cap (-\infty,t]$ sufficiently close to $t$. \cref{fact:exp} enables one to extend it to any $s\in I \cap (-\infty,t]$.

\ref{item:exp_smooth} Assuming the existence of a solution $x:I\to U$ to the ODE passing through $x_0 = x(t_0)$ at time $t_0$ implies the existence of solutions on an open interval $J\supseteq I$ for initial points in a neighborhood $U_0$ of $x_0$. This holds because $\nabla c$ is Lipschitz continuous on $U$, which also implies uniqueness of solutions (using the ODE comparison theorem \cite[Theorem D.2]{lee2012smooth}). Thus the flow $\theta:J \times U_0\to U$ is well defined and $C\p k$. For any $t \in J$, $\theta_t = \theta(t,\cdot)$ defines a diffeomorphism from $U_0$ to $\theta_t(U_0)$ \cite[p. 209]{lee2012smooth}, which is a neighborhood of $x(t)$. Conservation implies that $f = f \circ \theta_t$. By the chain rule, $(f\circ \theta_t)'(x) = f'(\theta_t(x))\circ \theta_t'(x)$. In particular, 
$$\forall v\in R\p n, ~~~ 0 = f'(x_0)(v) = (f\circ \theta_t)'(x_0)(v) = f'(\theta_t(x_0))(\theta_t'(x_0)v),$$
so that $f'(\theta_t(x_0)) = 0$. By induction, $(f\circ \theta_t)\p {(i)}(x_0)v\p i = f\p{(i)}(\theta_t(x_0))(\theta_t'(x_0)v)\p i$ for all $i\in\llbracket 1,k\rrbracket$ and thus $f\p{(i)}(x(t)) = 0$ for all $i\in\llbracket 1,k-1\rrbracket$. It follows that
\begin{align*}
\|f\p {(k)}(x(t))\|/k! + o(1) & =\cf(x(t),r)/r\p k \leq \cf(x(s),e\p {-\omega (t-s)/2} r)/r\p k \\
& = e\p {-k\omega (t-s)/2}\|f\p {(k)}(x(s))\|/k! + o(1)
\end{align*}
for all $s\in I \cap (-\infty,t]$ sufficiently close to $t$ by \cref{thm:contracting} and \cref{prop:smooth}. One again concludes by \cref{fact:exp}.
\end{proof}

\begin{remark}
    \label{rem:ineq}
    One can actually weaken the equality $\langle \nabla c(x) , v \rangle = 0$ for all $v\in \overline{\nabla} f(x)$ in \cref{assum:conserved} to an inequality $\langle \nabla c(x) , v \rangle \geq 0$. In exchange, one must assume that $x(\cdot)$ is defined on $[0,\overline{t}]$ (at no cost) and $x(\overline{t})\in\arg\min \{f(x): x \in U\}$ in \cref{thm:contracting}. In the proof, $f$ instead increases along the trajectories of $\nabla c$, but is still constant on $x(\cdot)$, so that $|f(x_r(0))-f(x(0))| \leq |f(x_r(t))-f(x(t))|$. 
    
    Similarly, one must assume that $x(t_0)\in\arg\min \{f(x): x \in U\}$ in \cref{cor:flattening}. In the proof, instead of the chain rule, one uses Taylor expansions: $f(y) = f(x_0) + \langle \nabla f(x_0),y-x_0\rangle + o(|y-x_0|)$ and $f(\theta_t(y)) = f(\theta_t(x_0)) + \langle \nabla f(\theta_t(x_0)),\theta_t(y)-\theta_t(x_0)\rangle + o(|\theta_t(y)-\theta_t(x_0)|)$. Since $f(x_0) = f(\theta_t(x_0))\leq f(\theta_t(y))\leq f(y)$, this yields $\langle \nabla f(\theta_t(x_0)),\theta_t(y)-\theta_t(x_0)\rangle + o(|\theta_t(y)-\theta_t(x_0)|) = o(|y-x_0|)$ and so $\langle \nabla f(\theta_t(x_0)),\theta_t'(x_0)v\rangle = 0$ for all $v\in \R\p n$, namely $\nabla f(x(t))=0$. The rest follows by induction.

    With these new assumptions, it is sufficient to satisfy the second inequality in \cref{assum:conserved}, i.e., $\langle \nabla\p 2 c(x) v , v \rangle \leq - \omega |v|\p 2$ for all $v\in \overline{\nabla}f(x)$ and $x\in U$, merely for all $x \in U \setminus \arg\min \{f(x): x \in U\}$.
\end{remark}

\subsection{Linear symmetries}

Linear symmetries give rise to a conservation law in subgradient dynamics, as shown in \cite{josz2025subdifferentiation}. We next recall it and compute the Hessian of the conserved quantity in some directions, since it appears in \cref{assum:conserved}. Note that a conservation law in gradient dynamics was proposed earlier in \cite[Proposition 5.1]{zhao2023symmetries}

\begin{assumption}
\label{assum:symmetry}
    Let $f:\mathbb{R}^n\rightarrow \mathbb{R}$ be locally Lipschitz and invariant under the natural action of a Lie subgroup $G$ of $\mathrm{GL}(n,\mathbb{R})$, $C(x) = P_{\mathrm{s}(\mathfrak{g})}(xx^T)$, and $\overline{x} \in \R\p n$.
\end{assumption}

\begin{proposition}
\label{prop:conserved}
Let \cref{assum:symmetry} hold and $c(x) = \|C(x)-C(\overline{x})\|_F\p 2/4$ for some $\overline{x}\in\R\p n$. Then 
    $$\forall x \in \R\p n,~ \forall v \in \overline{\partial} f(x), ~~~  \langle \nabla c(x), v \rangle = 0 ~~~\text{and}~~~ \langle \nabla\p 2 c(x) v , v \rangle  = \langle C(x) - C(\overline{x}) , C(v) \rangle_F.$$
\end{proposition}
\begin{proof}
    By \cite[Corollary 5.1]{josz2025subdifferentiation}, for all $x \in \mathbb{R}^n$, $v\in\overline{\partial} f(x)$, and $\alpha\in \mathbb{R}$, we have
    $C(x + \alpha v) = C(x) + \alpha^2 C(v)$ and thus 
    \begin{align*}
        c(x+\alpha v) & = \|C(x+\alpha v)-C(\overline{x})\|_F\p 2/4
        = \|C(x)+\alpha\p 2 C(v)-C(\overline{x})\|_F\p 2/4 \\
        & = \|C(x)-C(\overline{x})\|_F\p 2/4 + \langle C(x)-C(\overline{x}) , C(v) \rangle_F \alpha\p 2/2 + \alpha\p 4 \|C(v)\|_F\p 2/4 \\
        & = c(x) + \langle \nabla c(x) ,v \rangle \alpha + \langle \nabla\p 2 c(x) v ,v \rangle \alpha\p 2/2 + o(\alpha\p 3). \qedhere
    \end{align*}
\end{proof}

Our focus is now on deriving necessary conditions for flatness using the conserved quantity. The following result can be seen as a warm up. It will be useful later.

\begin{proposition}
\label{prop:min_norm}
Under \cref{assum:symmetry}, $\overline{x} \in \arg\loc\min \{ |x|: x \in G\overline{x}\} \Longrightarrow C(\overline{x}) = 0$.
\end{proposition}
\begin{proof}
Let $\gamma:(-\epsilon,\epsilon)\rightarrow G$ be a smooth curve such that $\gamma(0) = I_n$ where $\epsilon>0$. Since $0$ is a local minimum of $(-\epsilon,\epsilon) \ni t \to |\gamma(t)\overline{x}|\p 2$, $\langle \gamma(0)\overline{x},\gamma'(0)\overline{x}\rangle = \langle \overline{x},\gamma'(0)\overline{x}\rangle = \langle \gamma'(0), \overline{x}\hspace{.5mm}\overline{x}\p T\rangle = 0$. Thus $\langle v, \overline{x}\hspace{.5mm}\overline{x}\p T\rangle = 0$ for all $v \in \fg$ and $C(\overline{x}) = P_{\mathrm{s}(\fg)} (\overline{x}\hspace{.5mm}\overline{x}\p T)=0$.    
\end{proof}

We have arrived at our second main result.

\begin{theorem}
\label{thm:flat_0}
Under \cref{assum:symmetry}, $$\overline{x}\in \arg\loc\min \{ \lip f(x) : x \in G\overline{x}\} ~\Longrightarrow ~ \exists \overline{v} \in
 \arg\max \{|v| : v \in \partial f(\overline{x})\}: ~C(\overline{v}) = 0.$$
\end{theorem}
\begin{proof}
By assumption, $I_n$ is a local solution to
$$\inf_{g\in G} ~ \lip f(g^{-1}\overline{x})^2 = \inf_{g\in G} \sup_{v \in \partial f(g^{-1}\overline{x})} |v|^2
    = \inf_{g\in G} \sup_{v \in \partial f(\overline{x})} |g^Tv|^2
     = \inf_{g\in G} - \inf_{v \in \mathbb{R}^n} \phi(v,g)
     = -\sup_{g\in G} \varphi(g)$$
where the second equality holds by \cref{fact:chain_orbit_lsc}. Also, $\phi : \mathbb{R}^n \times G \rightarrow \overline{\mathbb{R}}$ and $\varphi : G \rightarrow \mathbb{R}_-$ are defined by
$$\phi(v,g) = \delta_{\partial f(\overline{x})}(v) - |g^Tv|^2 ~~~\text{and}~~~\varphi(g) = \inf_{v \in \mathbb{R}^n} \phi(v,g).$$
Let $\gamma:(-\epsilon,\epsilon)\rightarrow G$ be a smooth curve such that $\gamma(0) = I_n$ where $\epsilon>0$. Consider the function $\tau:\mathbb{R}^n \times \mathbb{R} \rightarrow \overline{\mathbb{R}}$ defined by $$\tau(v,t)= \phi(v,\gamma(t)) + \delta_{[-\epsilon/2,\epsilon/2]}(t).$$ We have 
$$\forall t\in [-\epsilon/2,\epsilon/2], ~~~ (\varphi \circ \gamma)(t) = \inf_{v \in \mathbb{R}^n} \tau(v,t).$$
The function $\tau$ is lower semicontinuous and continuous on its compact domain $\partial f(\overline{x}) \times [-\epsilon/2,\epsilon/2]$ since $f$ is Lipschitz continuous near $\overline{x}$. Thus $\varphi \circ \gamma$ is real-valued on $[-\epsilon/2,\epsilon/2]$. Since $0$ is a local maximum of $\varphi \circ \gamma$, we have $0 \in \partial(\varphi \circ \gamma)(0)$ by Fermat's rule \cite[Theorem 10.1]{rockafellar2009variational}. As the sum of the indicator of a closed convex set and a smooth function, $\tau$ is regular \cite[Example 7.28]{rockafellar2009variational}. \cite[Corollary 10.11]{rockafellar2009variational} then implies that for all $(v,t) \in \mathrm{dom}\tau$, one has
\begin{equation*}
    \partial \tau (v,t) \subseteq \begin{pmatrix}
        N_{\partial f(\overline{x})}(v) - 2 \gamma(t)\gamma(t)^Tv \\
        N_{[-\epsilon/2,\epsilon/2]}(t) -2\langle \gamma(t)^T v , \gamma'(t)^Tv \rangle
    \end{pmatrix}.
\end{equation*}
Due to its bounded domain, the function $\tau(v,t)$ is level-bounded in $v$ locally uniformly in $t$ \cite[1.16 Definition]{rockafellar2009variational}. We may thus apply \cite[Theorem 10.13]{rockafellar2009variational} on parametric subdifferentiation. For all $\overline{t} \in [-\epsilon/2,\epsilon/2]$, it yields
\begin{equation*}
    \partial (\varphi \circ \gamma)(\overline{t}) \subseteq \bigcup_{\overline{v} \in\arg\min\limits_{v\in \mathbb{R}^n} \tau(v,\overline{t})} M(\overline{v},\overline{t})~~~\text{where}~~~M(\overline{v},\overline{t})= \{ y \in \mathbb{R} : (0,y) \in \partial \tau (\overline{v},\overline{t})\}.
\end{equation*}
In particular, since $0 \in \partial (\varphi \circ \gamma)(0)$, there exists $$\overline{v}\in \arg\min \{ \tau(v,0) : v\in \mathbb{R}^n\} = \arg\max \{ |v|: v\in \partial f(\overline{x})\}$$ such that
\begin{equation*}
    (0,0) \in \partial \tau (\overline{v},0) \subseteq 
    \begin{pmatrix}
        N_{\partial f(\overline{x})}(\overline{v}) - 2\overline{v} \\
        -2\langle \overline{v} , \gamma'(0)\overline{v} \rangle
    \end{pmatrix}.
\end{equation*}
Hence $\langle \overline{v} , u\overline{v} \rangle = \langle \overline{v}\hspace{.3mm}\overline{v}^T , u \rangle = 0$ for all $u \in \mathfrak{g}$, and so $C(\overline{v}) = P_{\mathrm{s}(\mathfrak{g})} (\overline{v}\hspace{.3mm}\overline{v}^T)=0$.
\end{proof}

In \cref{thm:flat_0}, one can also take a maximal Bouligand subdifferential by \cref{fact:hull}. The converse of \cref{thm:flat_0} do not always hold.

\begin{example}
    \label{eg:zero_C}
    The function $f(x) = |2x_1\p 2 - x_2\p 2|$ is invariant under the natural action of $G = \{ A \in \GL(2) : A\p T D A = D\}$ where $D = \diag(2,-1)$. The only flat minimum is $(0,0)$ even though $C(x)=P_{\mathrm{s}(\fg)}(xx\p T)= 0$ for all $x\in \R\p 2$.
\end{example}
\begin{proof}
Observe that $f(x) = \langle x , Dx\rangle$. For all $A \in \GL(2)$, $f(Ax)  = \langle Ax,DAx\rangle = \langle x,A\p TDAx\rangle  = \langle x,Dx\rangle = f(x)$. Thus $f$ is invariant under the natural action of $G$, whose Lie algebra is
    $$\fg=\left\{ 
    B \in \R\p{2\times 2}
    :~B^TD+DB=0\right\}=\left\{\begin{pmatrix}
    0 & 2t \\
    -t & 0
\end{pmatrix}:~t\in \R\right\}.
$$ 
Thus $\mathrm{s}(\fg) = \{0\}$ and $C(x)=P_{\mathrm{s}(\fg)}(xx\p T)= 0$ for all $x\in \R\p 2$. By \cref{fact:chain_reg}, $\lip f(x)=2\sqrt{4x_1\p 2+x_2\p 2}$ and thus $(0,0)$ is the only flat minimum by \cref{cor:regular}.
\end{proof}

The converse of \cref{thm:flat_0} does hold in $\ell_1$-matrix factorization, as we show in \cref{subsec:Matrix factorization}.

\begin{corollary}
\label{cor:flat_0}
Under \cref{assum:symmetry}, if $\overline{x}$ is flat and $f$ is regular near $\overline{x}$, then there exists $\overline{v} \in
 \arg\max \{|v| : v \in \partial f(\overline{x})\}$ such that $C(\overline{v}) = 0$.
\end{corollary}
\begin{proof}
    This a consequence of \cref{fact:chain_orbit_lsc}, \cref{cor:regular}, and \cref{thm:flat_0}.
\end{proof}

\cref{thm:flat_0} can be generalized to higher orders, as follows.

\begin{theorem}
\label{thm:flat_k}
    Suppose $f$ is $D\p k$ near $\overline{x} \in \R\p n$ and $f^{(k)}(\overline{x})\neq 0$. Then
\begin{gather*}
    \overline{x} \in \arg\loc\min \{ \|f\p{(k)}(x)\| : x \in G\overline{x} \} \\
    \Longrightarrow \\
    \exists \overline{v} \in \arg\max\{ |f^{(k)}(\overline{x})(v,\hdots,v)| : |v|=1\}: ~ C(\overline{v}) = 0.
\end{gather*}
\end{theorem}
\begin{proof}
The proof naturally mirrors that of \cref{thm:flat_0}. 
By assumption, $I_n$ is a local solution to
\begin{align*}
    \inf_{g\in G} \|f^{(k)}(g^{-1}\overline{x})\|^2 
     & = \inf_{g\in G} \sup_{|v|=1} (f^{(k)}(g^{-1}\overline{x})(v,\hdots,v))^2 
     = \inf_{g\in G} \sup_{|v|=1} (f^{(k)}(\overline{x})(gv,\hdots,gv))^2 \\ 
    & = \inf_{g\in G} - \inf_{v \in \mathbb{R}^n} \phi(v,g)
    = -\sup_{g\in G} \varphi(g)
\end{align*}
where the second equality holds by \cref{fact:chain_orbit_smooth}.
Also, $\phi : \mathbb{R}^n \times G \rightarrow \overline{\mathbb{R}}$ and $\varphi : G \rightarrow \mathbb{R}_-$ are defined by
$$\phi(v,g) = \delta_{B\p n}(v) - (f^{(k)}(\overline{x})(gv,\hdots,gv))^2 ~~~ \text{and} ~~~ \varphi(g) = \inf_{v \in \mathbb{R}^n} \phi(v,g).$$ 
Let $\gamma:(-\epsilon,\epsilon)\rightarrow G$ be a smooth curve such that $\gamma(0) = I_n$ where $\epsilon>0$. Consider the function $\tau:\mathbb{R}^n \times \mathbb{R} \rightarrow \overline{\mathbb{R}}$ defined by $$\tau(v,t)= \phi(v,\gamma(t)) + \delta_{[-\epsilon/2,\epsilon/2]}(t).$$ It holds that 
$$\forall t\in [-\epsilon/2,\epsilon/2], ~~~ (\varphi \circ \gamma)(t) = \inf_{v \in \mathbb{R}^n} \tau(v,t).$$ 
Thus for all $(v,t) \in \mathrm{dom}\tau$, one has $\partial \tau (v,t) \subseteq $
\begin{equation*}
    \begin{pmatrix}
        N_{B\p n}(v) - 2k f^{(k)}(\overline{x})(\gamma(t)v,\hdots,\gamma(t)v) \gamma(t)^T \nabla^k f(\overline{x})(\gamma(t)v,\hdots,\gamma(t)v) \\
        N_{[-\epsilon/2,\epsilon/2]}(t) - 2k f^{(k)}(\overline{x})(\gamma(t)v,\hdots,\gamma(t)v) f^{(k)}(\overline{x})(\gamma'(t)v,\gamma(t)v,\hdots,\gamma(t)v)
    \end{pmatrix}
\end{equation*}
and for all $\overline{t} \in [-\epsilon/2,\epsilon/2] \subseteq \mathrm{dom}(\varphi \circ \gamma)$, 
\begin{equation*}
    \partial (\varphi \circ \gamma)(\overline{t}) \subseteq \bigcup_{\overline{v} \in\arg\min\limits_{v\in \mathbb{R}^n} \tau(v,\overline{t})} M(\overline{v},\overline{t}) ~~~\text{where}~~~M(\overline{v},\overline{t})= \{ y \in \mathbb{R} : (0,y) \in \partial \tau (\overline{v},\overline{t})\}.
\end{equation*}
In particular, since $0 \in \partial (\varphi \circ \gamma)(0)$, there exists $$\overline{v}\in \arg\min \{ \tau(v,0) : v\in \mathbb{R}^n\} = \arg\max \{ (f^{(k)}(\overline{x})(v,\hdots,v))^2 : |v|=1 \}$$ such that
\begin{equation}
\label{eq:param}
    (0,0) \in \partial \tau (\overline{v},0) \subseteq 
    \begin{pmatrix}
        N_{B\p n}(\overline{v}) - 2k f^{(k)}(\overline{x})(\overline{v},\hdots,\overline{v}) \nabla^{k} f(\overline{x})(\overline{v},\hdots,\overline{v}) \\
        - 2k f^{(k)}(\overline{x})(\overline{v},\hdots,\overline{v}) f^{(k)}(\overline{x})(\gamma'(0)\overline{v},\overline{v},\hdots,\overline{v})
    \end{pmatrix}.
\end{equation}
Consider the Lagrangian $$L(v,\lambda) = (f^{(k)}(\overline{x})(v,\hdots,v))^2 - \lambda (|x|^{2k}-1),$$ following Lim \cite[Section 4]{lim2005singular}. Since the constraint is qualified by gradient independence, the first-order optimality condition ensures the existence of $\overline{\lambda} \in \mathbb{R}$ such that 
$$\nabla_v L(\overline{v},\overline{\lambda}) = 2k f^{(k)}(\overline{x})(\overline{v},\hdots,\overline{v}) \nabla^{k} f (\overline{x})(\overline{v},\hdots,\overline{v}) - 2k \overline{\lambda} \overline{v} = 0,$$
that is to say,
$$f^{(k)}(\overline{x})(\overline{v},\hdots,\overline{v}) \nabla^{k} f(\overline{x})(\overline{v},\hdots,\overline{v}) = \overline{\lambda} \overline{v}.$$
Taking the inner product with $\overline{v}$ yields
$$ 0 \neq |f^{(k)}(\overline{x})|^2 = (f^{(k)}(\overline{x})(\overline{v},\hdots,\overline{v}))^2 = \overline{\lambda} |\overline{v}|^2 = \overline{\lambda}.$$
The inclusion in \eqref{eq:param} then implies that for all $u \in \mathfrak{g}$, we have
   \begin{align*}
       0 & = f^{(k)}(\overline{x})(\overline{v},\hdots,\overline{v}) f^{(k)}(\overline{x})(u\overline{v},\overline{v},\hdots,\overline{v}) \\
       & = f^{(k)}(\overline{x})(\overline{v},\hdots,\overline{v}) \langle u \overline{v} , \nabla^k f(\overline{x})(\overline{v},\hdots,\overline{v}) \rangle \\
      & = \langle u \overline{v} , \overline{\lambda}\overline{v} \rangle 
      =  \overline{\lambda} \langle u , \overline{v}\hspace{.4mm}\overline{v}^T \rangle.
   \end{align*}
As in the proof of \cref{thm:flat_0}, we conclude that $C(\overline{v}) = 0$.
\end{proof}

\begin{corollary}
\label{cor:flat_k}
    Suppose $f$ is $D\p k$ near $\overline{x} \in \R\p n$ with $k \in \mathbb{N}^*$, $f\p{(i)}(\overline{x})=0$ for all $i\in \llbracket 1,k-1\rrbracket$, and $f^{(k)}(\overline{x})\neq 0$. If $\overline{x}$ is flat, then there exists $\overline{v} \in \arg\max\{ |f^{(k)}(\overline{x})(v,\hdots,v)| : |v|=1\}$ such that $C(\overline{v}) = 0$.
\end{corollary}
\begin{proof}
    This a consequence of \cref{fact:chain_orbit_smooth}, \cref{cor:smooth}, and \cref{thm:flat_k}.
\end{proof}


\subsection{Matrix factorization}
\label{subsec:Matrix factorization}

We seek to establish converse results to those developed in the previous section, in the case of matrix factorization, in a desire to characterize flat minima. 
Given $m,n,r\in\N\p *$ and $M\in \R\p{m\times n}$, the map 
    \begin{equation*}
    \begin{array}{cccc}
         F: & \mathbb{R}^{m\times r}\times \mathbb{R}^{r\times n} & \longrightarrow & \mathbb{R} \\
        & (X,Y) & \longmapsto & XY-M
    \end{array}
    \end{equation*}
is invariant under the action of $\GL(r)$ on $\R\p{m\times r}\times \R\p{r\times n}$ defined by $$(X,Y,A)\mapsto (XA,A\p {-1}Y).$$ As shown in \cite[Example 5.3]{josz2025subdifferentiation}, for any locally Lipschitz function $g:\R\p{m\times n}\to \R$, $g\circ F$ then admits the conserved quantity $$C(X,Y) = X\p T X - YY\p T$$ (see \cite{arora2018optimization,du2018algorithmic,josz2023global,marcotte2024abide} for other derivations). A pair $(X,Y)$ is balanced if $X\p T X = Y Y\p T$ \cite[Chapter 6]{helmke1994optimization}. We successively treat the cases where $g = \|\cdot\|_1$ and $g=\|\cdot\|_F\p 2/2$. We begin with a simple fact, which is a special case of \cite[Proposition 1]{ding2014introduction}.

\begin{fact}
\label{fact:commute}
    Let $D = \diag (d_1 I_{k_1},\dots, d_\ell I_{k_\ell})$ with $d_1 > \cdots > d_\ell$ and $k_1+\cdots + k_\ell = n$, and let $P \in \O(n)$. Then $$D = PDP\p T~~~\Longleftrightarrow ~~~ \forall i\in\llbracket 1,\ell\rrbracket, ~ \exists P_{k_i} \in \O(k_i):~ P = \diag (P_{k_1},\dots, P_{k_\ell}).$$ 
\end{fact}
\begin{proof}
    One direction is obvious. Suppose $D = PDP\p T$ and let $p_i = (p_{ij})_{j}$ denote the $i$\textsuperscript{th} column of $P$. Then $(d_{j} p_{ij})_j = Dp_i = PDP\p T p_i = \sum_{j} d_{j} p_jp_j\p T p_i = d_{i} p_i$, i.e., $(d_{j} - d_{i})p_{ij} = 0$.
\end{proof}

\cref{fact:commute} implies that for any diagonal matrix $D$ of order $n$ with nonnegative entries and $P\in \O(n)$, one has $DP = PD$ iff $D\p {1/2}P = PD\p {1/2}$. In particular, given $A\succcurlyeq 0$, there is a unique $B\succcurlyeq 0$ such that $A = B\p 2$, called the square root of $A$. (Indeed if $B\p 2=\widetilde B\p 2$ with $B,\widetilde B\succcurlyeq 0$, then given some eigendecompositions $B = P D P\p T$ and $\widetilde B = \widetilde P \widetilde D \widetilde P\p T$, we have $D = \widetilde D$ and $\widetilde P\p T P D\p 2 = D\p 2 \widetilde P\p T P$. Thus $\widetilde P\p TP D = D \widetilde P\p T P$ and $B=\widetilde B$.)

The optimal value in the lemma below is already known \cite[Lemma 1]{srebro2005rank}, but the local-global property appears to be new. Also, the characterization of global minima via the balance condition on $X$ and $Y$ was stated in \cite[Lemma 2.2]{ding2024flat}, although the proof seems invalid. A similar result constrained to a single orbit was obtained in \cite[Theorem 3.7]{helmke1994optimization}.


\begin{lemma}
    \label{lemma:nuclear} 
    If $r\geq \rank(M)$, then $2\|M\|_* = \min \{ \|X\|_F\p 2+\|Y\|_F\p 2 : XY = M \}$ and a feasible point $(X,Y)$ is a global minimum iff it is a local minimum iff $X\p T X = Y Y\p T$.
\end{lemma}
\begin{proof}
    Consider a singular value decomposition $M = U\Sigma V\p T$ and $p=\rank M$. Let $x_i\p T$ denote the rows of $X$ and $y_j$ the columns of $Y$. If $XY = \Sigma$, then
    $$\|\Sigma\|_* = \sum_{i=1}\p p \langle x_i,y_i\rangle \leq \sum_{i=1}\p p |x_i||y_i| \leq \sqrt{\sum_{i=1}\p p |x_i|\p 2} \sqrt{\sum_{i=1}\p p |y_i|\p 2} \leq \|X\|_F \|Y\|_F \leq (\|X\|_F\p 2 + \|Y\|_F\p 2)/2.$$
    If $XY=M$, then $U\p T X Y V = \Sigma$ and $2\|M\|_* = 2\|\Sigma\|_* \leq \|U\p T X\|_F\p 2 + \|YV\|_F\p 2 = \|X\|_F\p 2 + \|Y\|_F\p 2$, with equality when $(X,Y)$ equals $(U\Sigma\p{1/2},\Sigma\p {1/2} V\p T)$ up to some zero rows/columns. If $(X,Y)$ is a local minimum then $XY = M$ and $X\p T X = Y Y\p T$ by \cref{prop:min_norm}. Conversely, consider some compact singular value decompositions $X = U_X\Sigma_X V_X\p T$ and $Y = U_Y \Sigma_Y V_Y\p T$. By \cref{lemma:balanced} below, $\Sigma_X = \Sigma_Y = \sqrt{\Sigma}$. Thus $\|X\|_F\p 2 + \|Y\|_F\p 2 = 2\|M\|_*$ and $(X,Y)$ is globally optimal.
\end{proof}

The following result can also be deduced from \cite[Lemma 3.1]{ouyang2025burer}.

\begin{lemma}
\label{lemma:balanced}
    If $XY=M$ and $X\p TX = Y Y\p T$, then $\Sigma_X = \Sigma_Y = \sqrt{\Sigma}$ given some compact singular value decompositions $X = U_X\Sigma_X V_X\p T$ and $Y = U_Y \Sigma_Y V_Y\p T$
\end{lemma}
\begin{proof}
    Plugging in the decompositions yields $V_X \Sigma_X\p 2 V_X\p T = U_Y \Sigma_Y\p 2 U_Y\p T$, so that $\Sigma_X = \Sigma_Y$ and $\Sigma_X\p 2 = (V_X\p T U_Y) \Sigma_X\p 2 (V_X\p T U_Y)\p T$. In the proof of \cite[Lemma 2.2]{ding2024flat}, one concludes that $V_X\p T U_Y = I_r$, but that seems untrue (think of $\Sigma_X = I_r$). In fact, $\Sigma_X\p 2 (V_X\p T U_Y) = (V_X\p T U_Y) \Sigma_X\p 2$ and thus $\Sigma_X (V_X\p T U_Y) = (V_X\p T U_Y) \Sigma_X$ by \cref{fact:commute}. Using a compact singular value decomposition $M = U\Sigma V\p T$, this yields $XY = U_X \Sigma_X V_X\p T U_Y \Sigma_X V_Y\p T = U_X V_X\p T U_Y \Sigma_X\p 2 V_Y\p T = U \Sigma V\p T$, and so $\Sigma_X = \Sigma_Y = \sqrt{\Sigma}$.
\end{proof}

Using \cref{lemma:nuclear}, we can show that the converse of \cref{thm:flat_0} holds in $\ell_1$-matrix factorization.

\begin{proposition}
    \label{prop:mf_1}
    Given $M \in \R\p{m\times n}$, let $f:\R\p{m\times r}\times \R\p{r\times n}\to \R$ be defined by $f(X,Y) = \|XY-M\|_1$. Let $\overline{X}\hspace{.4mm}\overline{Y} = M$. The following are equivalent:
    \begin{enumerate}[label=\rm{(\roman{*})}]
    \item \label{item:mf_1_global} $(\overline{X},\overline{Y}) \in \arg\min \{ \lip f(X,Y) : XY = M\}$;
    \item \label{item:mf_1_local} $(\overline{X},\overline{Y}) \in \arg\loc\min \{ \lip f(X,Y) : XY = M\}$;
    \item \label{item:mf_1_balance} $\exists(\overline{H},\overline{K}) \in \arg\max
 \{ \|H\|_F\p 2+\|K\|_F\p 2 : (H,K) \in \partial f(\overline{X},\overline{Y})\} : ~ \overline{H}\p T \overline{H} = \overline{K}\hspace{.4mm}\overline{K}\p T$.
    \end{enumerate}
    There exist global minima since for all $XY = M$,
    $\lip f(X,Y) \geq \sqrt{\|X\|_F\p 2 + \|Y\|_F\p 2}/\sqrt{m+n}$.
\end{proposition}
\begin{proof}
\ref{item:mf_1_global} $\Longrightarrow$ \ref{item:mf_1_local} is obvious. \ref{item:mf_1_local} $\Longrightarrow$ \ref{item:mf_1_balance} is due to \cref{thm:flat_0}. \ref{item:mf_1_balance} $\Longrightarrow$ \ref{item:mf_1_global} Observe that $f(X,Y)=\|F(X,Y)\|_1$ where $F(X,Y) = XY-M$. We thus compute $F(X,Y)(H,K) = XK+HY$ and $F'(X,Y)\p *(\Lambda)=(X\p T \Lambda,\Lambda Y\p T)$. By \cref{fact:chain_reg}, 
$$\partial f(X,Y) = \left\{ \begin{pmatrix}
    \Lambda Y^T \\
            X^T \Lambda
\end{pmatrix} : \Lambda \in \mathrm{sign}(XY-M)\right\}.$$
By assumption, there exists $\overline{\Lambda}\in [-1,1]\p {m\times n}$ such that $(\overline{H},\overline{K}) = (\overline{\Lambda}\hspace{.4mm} \overline{Y}^T, \overline{X}^T \overline{\Lambda})$. Hence $\overline{H}\hspace{.4mm}\overline{K} = \overline{\Lambda}\hspace{.4mm} \overline{Y}^T \overline{X}^T \overline{\Lambda} = \overline{\Lambda}\hspace{.4mm} M\p T \overline{\Lambda}$. Since $\overline{H}\p T \overline{H}=\overline{K}\hspace{.4mm}\overline{K}\p T$, by \cref{lemma:nuclear},
$$(\overline{H},\overline{K}) \in \arg\min\{ \|H\|_F\p 2+\|K\|_F\p 2 : HK = \overline{\Lambda}\hspace{.4mm} M\p T \overline{\Lambda} \}.$$
Thus, whenever $XY=M$, we have
$$\lip f(X,Y) \geq \sqrt{\|\overline{\Lambda} Y^T\|_F^2+\|X^T\overline{\Lambda}\|_F^2} \geq \sqrt{\|\overline{H}\|_F\p 2+\|\overline{K}\|_F\p 2} = \lip f(\overline{X},\overline{Y}).$$
Finally, let $x_i\p T$ denote the rows of $X$ and $y_j$ denote the columns of $Y$. We have 
\begin{align*}
    \lip f(X,Y) & = \max_{\Lambda \in [-1,1]^{m\times n}} \sqrt{\|\Lambda Y^T\|_F^2+\|X^T\Lambda\|_F^2} \geq 
    \max\{|x_1|,\hdots,|x_m|,|y_1|,\hdots,|y_n|\} \\
    & \geq \sqrt{|x_1|\p 2+\cdots+|x_m|\p 2+|y_1|\p 2+\cdots+|y_n|\p 2}/\sqrt{m+n} \\
    & = \sqrt{\|X\|_F\p 2 + \|Y\|_F\p 2}/\sqrt{m+n}.
\end{align*}
To see this, successively take $ \Lambda $ with a single nonzero entry, equal to one, to generate the rows of $X$ and the columns of $Y$. 
\end{proof}

We now consider matrix factorization with the Frobenius norm. It is known to have no spurious local minima, as proved in  \cite{baldi1989neural} when $m=n$, and in \cite{valavi2020revisiting} for any $m,n$. We will show that it also has no spurious flat minima, i.e., flat minima that are not globally flat. To see why, it is useful to recall when equality holds in $\|XY\|_2 = \|X\|_2\|Y\|_2$ where $(X,Y)\in \R\p{m\times r }\times \R\p{r\times n}$. 

Given a real-valued matrix $A$ with a singular value decomposition $A=U_A\Sigma_A V_A\p T$, let $S_A = \{ v\in \R\p n : |Av|=\|A\|_2|v|\}$ and $d_A=\dim S_A$, i.e., the multiplicity of the maximal singular value of $A$. Let $\widehat U_A$ (resp. $\widehat{V}_A$) denote the first $d_A$ columns of $U_A$ (resp. $V_A$), in other words, the maximal left (resp. right) singular vectors of $A$.

\begin{fact}
\label{fact:product_spectral}
    $\|XY\|_2 = \|X\|_2\|Y\|_2 \Longleftrightarrow S_X \cap YS_Y\neq \{0\} \Longleftrightarrow \im\widehat{V}_X\cap \im \widehat{U}_Y\neq \{0\}$.
\end{fact}
\begin{proof}
    For all $v \in \R\p n$, $|XYv|\leq \|XY\|_2|v|\leq \|X\|_2\|Y\|_2|v|$ and $|XYv|\leq \|X\|_2 |Yv|\leq \|X\|_2\|Y\|_2|v|$. Thus $S_{XY} \supseteq S_X \cap YS_Y$ and, if $\|XY\|_2=\|X\|_2\|Y\|_2$, then $S_{XY} \subseteq S_X \cap YS_Y$. Lastly, $S_X = \im\widehat{V}_X$ and $Y S_Y = \im \widehat{U}_Y$.
\end{proof}

The optimal value in the next result is known in the multilayer case \cite{mulayoff2020unique}, a fortiori in the two layer case. On the other hand, the local-global property again appears to be new.

\begin{lemma}
    \label{lemma:spectral} 
    If $r\geq \rank(M)$, then $2\|M\|_2 = \min \{ \|X\|_2\p 2+\|Y\|_2\p 2 : XY = M \}$ and 
    a feasible point $(X,Y)$ is a global minimum iff it is a local minimum iff $\|XY\|_2 = \|X\|_2\|Y\|_2$ and $\|X\|_2=\|Y\|_2$ iff $\|X\|_2=\|Y\|_2 = \sqrt{\|M\|_2}$.
\end{lemma}
\begin{proof}
    Given a singular value decomposition $M = U\Sigma V\p T$, one has 
    $$2\|M\|_2 = 2\|XY\|_2 \leq 2\|X\|_2\|Y\|_2 \leq \|X\|_2\p 2 + \|Y\|_2\p 2 $$
    with equality exactly when $\|XY\|_2 = \|X\|_2\|Y\|_2$ and $\|X\|_2=\|Y\|_2$. This happens in particular when $(X,Y)$ equals to $(U\Sigma\p{1/2},\Sigma\p {1/2} V\p T)$ up to some zero rows/columns. If $\|X\|_2\neq\|Y\|_2$, say $\|X\|_2<\|Y\|_2$, then $\|tX\|_2\p2+\|t\p{-1}Y\|_2\p2<\|X\|_2\p2+\|Y\|_2\p2$ for all $t\in (1,\|Y\|_2\p{1/2}\|X\|_2\p{-1/2}]$. Suppose $\|XY\|_2 < \|X\|_2\|Y\|_2$, then consider some singular value decompositions $X=U_X\Sigma_XV_X\p T$ and $Y = U_Y\Sigma_YV_Y\p T$ and let $d_X$ denote the multiplicity of the maximal singular value of $X$. Consider $D_t = \diag(tI_{d_X},I_{r-d_X})$ where $t>0$. Let $X_t = XV_XD_tV_X\p T$ and $Y_t = V_X D_t\p{-1} V_X\p T Y$. For all $t\in(0,1)$, 
    \begin{align*}
        \|X_t\|_2\p 2 + \|Y_t\|_2\p 2 & = \|XV_XD_tV_X\p T\|_2\p 2+\|V_X D_t\p{-1} V_X\p T Y\|_2\p 2 \\
        & = \|U_X\Sigma_X V_X\p T V_XD_tV_X\p T\|_2\p 2+\|V_X D_t\p{-1} V_X\p T U_Y\Sigma_Y V_Y\p T\|_2\p 2 \\
        & = \|\Sigma_X D_t\|_2\p 2+\|D_t\p{-1} V_X\p T U_Y\Sigma_Y\|_2\p 2 \\
        & = t\p 2\|\Sigma_X\|_2\p 2+\|D_t\p{-1} V_X\p T U_Y\Sigma_Y\|_2\p 2 \\
        & < t\p 2\|\Sigma_X\|_2\p 2+\|D_t\p{-1} V_X\p T\|_2\p 2 \|U_Y\Sigma_Y\|_2\p 2\\
        & = t\p 2\|X\|_2\p 2+t\p{-2} \|Y\|_2\p 2\\
        & < \|X\|_2\p 2+\|Y\|_2\p 2
    \end{align*}
    by \cref{fact:product_spectral}.
\end{proof}

\cref{lemma:nuclear}, \cref{lemma:balanced}, and \cref{lemma:spectral} immediately imply the following.

\begin{corollary}
\label{cor:inclusion}
    $\arg\min \{ \|X\|_F\p 2+\|Y\|_F\p 2 : XY = M \}\subseteq \arg\min \{ \|X\|_2\p 2+\|Y\|_2\p 2 : XY = M \}$.
\end{corollary}
The inclusion can be strict in \cref{cor:inclusion}, as we will see in \cref{eg:unbalanced}. In contrast to the multilayer case \cite{mulayoff2020unique}, it is possible to determine the maximal eigenvalue and eigenspace of the Hessian at any global minimum. In fact, shortly after submitting this work on arXiv, an expression for the maximal eigenvalue in the multilayer case was proposed using Kronecker products \cite[Theorem 5]{kamber2025sharpness}, implying the expression below \cite[Corollary 6]{kamber2025sharpness}.

\begin{proposition}
\label{prop:mf_hessian}
    Given $M \in \R\p{m\times n}$, let $f:\R\p{m\times r}\times \R\p{r\times n}\to \R$ be defined by $f(X,Y) = \|XY-M\|_F\p 2/2$. Suppose $XY=M$. Then $$\lambda_1(\nabla\p 2 f(X,Y)) = \|X\|_2\p 2+\|Y\|_2\p 2~~~\text{and}$$
    $$E_1(\nabla\p 2 f(X,Y)) = \left\{ 
    \left( U_X \begin{pmatrix}
        A & 0 \\
        0 & 0
    \end{pmatrix} U_Y\p T
    ,
    V_X \begin{pmatrix}
        B & 0 \\
        0 & 0
    \end{pmatrix} V_Y\p T \right) : 
    \|X\|_2A = \|Y\|_2B \in \R\p {d_X\times d_Y}
    \right\},$$
    given any singular value decompositions $X=U_X\Sigma_XV_X\p T$ and $Y=U_Y\Sigma_YV_Y\p T$, where $d_X$ (resp. $d_Y$) denotes the multiplicity of the maximal singular value of $X$ (resp. $Y$). Also, 
    $$\exists(H,K) \in E_1(\nabla\p 2 ((X,Y)) : ~ \im H\p T \cap \im K \neq \{0\} ~~\Longrightarrow~~ \|XY\|_2=\|X\|_2\|Y\|_2.$$
\end{proposition}
\begin{proof}
    Observe that $f(X,Y)=\|F(X,Y)\|_F\p 2/2$ where $F(X,Y) = XY-M$. By \cref{fact:chain_smooth}, $\lambda_1(\nabla f(X,Y)) = \|F'(X,Y)\|_2\p 2$ so we compute $F'(X,Y)(H,K) = HY+XK$. 
    One has  
    \begin{align*}
        \|HY+XK\|_F &\leq \|HY\|_F+\|XK\|_F\leq \|H\|_F\|Y\|_2+\|X\|_2\|K\|_F \\
        & \leq \sqrt{\|H\|_F\p 2+\|K\|_F\p 2}\sqrt{\|X\|_2\p 2 + \|Y\|_2\p 2}
    \end{align*}
    with equality when $H = U_X E\p{mr}_{11} U_Y\p T$ and $K = V_X E\p {rn}_{11} V_Y\p T$ (where $E\p{pq}_{11} \in \R\p{p\times q}$ with only one nonzero entry, $(1,1)$, equal to 1). To determine exactly when equality holds, first reduce to the diagonal case:
        \begin{align*}
        \|F'(X,Y)\|_2 & = \max_{\|H\|_F\p 2+\|K\|_F\p 2\leq 1} \|HY+XK\|_F
        = \max_{\|H\|_F\p 2+\|K\|_F\p 2\leq 1} \|HU_Y\Sigma_YV_Y\p T+U_X\Sigma_XV_X\p TK\|_F \\
        & = \max_{\|H\|_F\p 2+\|K\|_F\p 2\leq 1} \|U_X\p T HU_Y\Sigma_Y+\Sigma_XV_X\p TKV_Y\|_F
        = \max_{\|H\|_F\p 2+\|K\|_F\p 2\leq 1} \|H\Sigma_Y+\Sigma_XK\|_F.
    \end{align*}
    Equality in the Cauchy-Schwarz inequality holds iff $(\|H\|_F,\|K\|_F)$ and $(\|X\|_2,\|Y\|_2)$ are positively colinear.
    Let $h_j$ denote the columns of $H$, and $k_i\p T$ the rows of $K$. One has $\|H\Sigma_Y\|_F=\|H\|_F\|\Sigma_Y\|_2$ iff $\sum_j(\Sigma_Y)_{jj}\p 2|h_j|\p 2=\sum_j(\Sigma_Y)_{11}\p 2|h_j|\p 2$ iff $\sum_j[(\Sigma_Y)_{jj}\p 2-(\Sigma_Y)_{11}\p 2]|h_j|\p 2 = 0$ iff $[(\Sigma_Y)_{jj}\p 2-(\Sigma_Y)_{11}\p 2]|h_j|\p 2 = 0$ iff $h_j = 0$ for all $j\notin \llbracket 1,d_Y\rrbracket$. Likewise, $\|\Sigma_X K\|_F=\|\Sigma_X\|_2 \|K\|_F$ iff $k_i= 0$ for all $i\notin \llbracket 1,d_X\rrbracket$. Next, $\|H\Sigma_Y+\Sigma_X K\|_F=\|H\Sigma_Y\|_F+\|\Sigma_X K\|_F$ iff $H\Sigma_Y$ and $\Sigma_X K$ are positively colinear. Hence all equalities hold iff
    $\|X\|_2H=\|Y\|_2K$ and $H_{ij} = K_{ij} = 0$ for all $(i,j)\notin \llbracket 1,d_X\rrbracket\times \llbracket 1,d_Y\rrbracket$.

    Let $\widehat V_{X}$ (resp. $\widehat U_{Y}$) denote the first $d_{X}$ (resp. $d_{Y}$) columns of $V_{X}$ (resp. $U_{Y}$). If $(H,K) \in E_1(\nabla\p 2 ((X,Y))$ and $\im H\p T \cap \im K\neq \{0\}$, then $\im \widehat U_Y A\p T \cap \im \widehat V_Y B \neq \{0\}$, $\im \widehat U_Y \cap \im \widehat V_Y \neq \{0\}$, and $\|XY\|_2=\|X\|_2\|Y\|_2$ by \cref{fact:product_spectral}. 
\end{proof}

It is now possible to completely characterize flat minima in matrix factorization. Recall that $\im AA\p T = \im A$ for any matrix $A\in\R\p {m\times n}$.

\begin{proposition}
    \label{prop:mf_2_flat}
    Given $M \in \R\p{m\times n}$, let $f:\R\p{m\times r}\times \R\p{r\times n}\to \R$ be defined by $f(X,Y) = \|XY-M\|_F\p 2/2$. Suppose $\overline{X}\hspace{.4mm}\overline{Y}=M$. The following are equivalent:
    \begin{enumerate}[label=\rm{(\roman{*})}]
        \item \label{item:global_flat} $(\overline{X},\overline{Y})$ is globally flat;
        \item \label{item:flat} $(\overline{X},\overline{Y})$ is flat;
        \item \label{item:local} $(\overline{X},\overline{Y}) \in \arg\loc\min \{ \lambda_1(\nabla\p 2 f(X,Y)) : XY = M\}$;
        \item \label{item:balanced} $\|\overline{X}\|_2 = \|\overline{Y}\|_2 ~\land~ \exists(\overline{H},\overline{K}) \in E_1(\nabla\p 2 f(\overline{X},\overline{Y}))\setminus\{0\} : ~ \overline{H}\p T \overline{H} = \overline{K}\hspace{.4mm}\overline{K}\p T$;
        \item \label{item:global} $(\overline{X},\overline{Y}) \in \arg\min \{ \lambda_1(\nabla\p 2 f(X,Y)) : XY = M\}$;
        \item \label{item:min} $\|\overline{X}\|_2\p 2+ \|\overline{Y}\|_2\p 2 = 2\|M\|_2$;
        \item \label{item:simple} $\|\overline{X}\|_2=\|\overline{Y}\|_2 = \sqrt{\|M\|_2}$.
    \end{enumerate}
\end{proposition}
\begin{proof}
\ref{item:global_flat} $\Longrightarrow$ \ref{item:flat} follows by \cref{def:flat}.

\ref{item:flat} $\Longrightarrow$ \ref{item:local} is due to \cref{cor:smooth}. 

\ref{item:local} $\Longrightarrow$ \ref{item:balanced} If $\|\overline X\|_2\neq\|\overline Y\|_2$, say $\|\overline X\|_2<\|\overline Y\|_2$, then $\|t\overline X\|_2\p2+\|t\p{-1}\overline Y\|_2\p2<\|\overline X\|_2\p2+\|\overline Y\|_2\p2$ for all $t\in (1,\|\overline Y\|_2\p{1/2}\|\overline X\|_2\p{-1/2}]$. If $(\overline X,\overline Y)=(0,0)$, then take $\overline H\p T$ and $\overline K$ to be the rectangular identities. Otherwise, $\nabla\p 2 f(\overline{X},\overline{Y})\neq 0$ so \cref{thm:flat_k} applies. 

\ref{item:balanced} $\Longrightarrow$ \ref{item:global} 
By \cref{lemma:spectral}, $\|\overline X\|_2=\|\overline Y\|_2$. Since $\overline{H}\p T \overline{H} = \overline{K}\hspace{.4mm}\overline{K}\p T\neq 0$, $\im \overline{H}\p T = \im\overline{H}\p T \overline{H} = \im \overline{K}\hspace{.4mm}\overline{K}\p T = \im \overline{K}\neq\{0\}$ and so $\|\overline X\hspace{.4mm}\overline Y\|_2 = \|\overline X\|_2\|\overline Y\|_2$ by \cref{prop:mf_hessian}.
    
\ref{item:global} $\Longrightarrow$ \ref{item:global_flat}
For any $XY=M$ such that $\|X\|_2 = \|Y\|_2$, we have
    \begin{align*}
        \cf(X,Y,r) & = \max_{\|H\|_F\p 2+\|K\|_F\p 2\leq r\p 2} \|HY+XK+HK\|_F\p 2/2 \\
        & = \max_{\|H\|_F\p 2+\|K\|_F\p 2\leq r\p 2} \|HU_Y\Sigma_YV_Y\p T+U_X\Sigma_XV_X\p TK+HK\|_F\p 2 /2 \\
        & = \max_{\|H\|_F\p 2+\|K\|_F\p 2\leq r\p 2} \|U_X\p T HU_Y\Sigma_Y+\Sigma_XV_X\p TKV_Y + U_X\p T H K V_Y\|_F\p 2 /2 \\
        & = \max_{\|H\|_F\p 2+\|K\|_F\p 2\leq r\p 2} \|H\Sigma_Y+\Sigma_XK+HU_Y\p T V_X K\|_F\p 2 /2 \\
        & = \max_{\|H\|_F\p 2+\|K\|_F\p 2\leq r\p 2} (\|H\|_F\|Y\|_2+\|X\|_2\|K\|_F+\|H\|_F \|K\|_F)\p 2 /2 \\
        & = (\sqrt{2}\|X\|_2r+r\p 2/2)\p 2 /2 \\
        & = (2\|X\|_2\p 2 r\p 2 + \sqrt{2}\|X\|_2 r\p 3 + r\p 4/4)/2 \\
        & = (\|X\|_2\p 2+\|Y\|_2\p 2)r\p 2/2 + \sqrt{2} (\|X\|_2+\|Y\|_2) r\p 3/4 + r\p 4/8
    \end{align*}
    using the same direction to obtain the equalities as in first part of the proof of \cref{prop:mf_hessian}. Let $XY=M$. Since $f$ is $D\p 2$ and $\nabla f(X,Y) = 0$, by \cref{cor:smooth}, $\cf(X,Y,r) = \lambda_1(\nabla\p 2 f(X,Y))r\p 2/2 + o(r\p 2)$. Hence, if $\lambda_1(\nabla\p 2 f(\overline X,\overline Y))<\lambda_1(\nabla\p 2 f(X,Y))$, then there exists $\overline{r}>0$ such that $\cf(\overline X,\overline Y,r)<\cf(X,Y,r)$ for all $r\in(0,\overline{r}]$. Otherwise, by \cref{lemma:spectral}, $\|X\|_2=\|Y\|_2=\sqrt{\|M\|_2}=\|\overline X\|_2=\|\overline Y\|_2$ and so $\cf(\overline X,\overline Y,r) = \cf(X,Y,r)$ for all $r\geq 0$. In either case, $(\overline X , \overline Y) \preceq (X,Y)$, hence $(\overline X , \overline Y)$ is globally flat.

    \ref{item:global} $\Longleftrightarrow$ \ref{item:min} $\Longleftrightarrow$ \ref{item:simple} This was already shown in \cref{lemma:spectral}.
\end{proof}

One immediately obtains the following corollary, as announced in the introduction.

\begin{corollary}
    \label{cor:balanced_flat}
    Given $M \in \R\p{m\times n}$, let $f:\R\p{m\times r}\times \R\p{r\times n}\to \R$ be defined by $f(X,Y) = \|XY-M\|_F\p 2/2$. If $XY=M$ and $X\p T X=Y Y\p T$, then $(X,Y)$ is flat.
\end{corollary}
\begin{proof}
    This follows from \cref{lemma:nuclear}, \cref{cor:inclusion}, \cref{prop:mf_hessian}, and \cref{prop:mf_2_flat}.
\end{proof}

While the converse of \cref{cor:balanced_flat} is false, as noted in  \cite{mulayoff2020unique}, \cref{prop:mf_2_flat} implies that when $(X,Y)$ is a flat global minimum, there must exist a maximal eigenvector $(H,K)$ of the Hessian such that $H\p T H = K K\p T$. This appears to be new. We give a counterexample to the converse of \cref{cor:balanced_flat} using a well-known fact.


\begin{fact}
\label{fact:exact}
    Given $M\in\R\p{m\times n}$, consider a compact singular value decomposition $M=U\Sigma V\p T$ and let $r = \rank M$. For all $(X,Y)\in\R\p{m\times r}\times \R\p{r\times n}$,  
    $$XY=M ~~~\Longleftrightarrow~~~ \exists A\in \GL(r): (X,Y)=(U\Sigma\p{1/2}A,A\p{-1}\Sigma\p{1/2}V\p T).$$
\end{fact}
\begin{proof}
    First consider the  diagonal case $$M = \begin{pmatrix}
        \Sigma & 0 \\
        0 & 0
    \end{pmatrix}, ~~~  
    X = \begin{pmatrix}
        X_1 \\
        X_2
    \end{pmatrix}, ~~~\text{and} ~~~
    Y = \begin{pmatrix}
        Y_1 & Y_2
    \end{pmatrix}.
    $$
    Then $XY = M$ iff $X_1Y_1 = \Sigma$, $X_1Y_2 = 0$, $X_2Y_1 = 0$, and $X_2Y_2 =0$. But $X_1Y_1 = \Sigma$ iff $\Sigma \p{-1/2}X_1Y_1\Sigma \p{-1/2} = I_r$ iff there is $A \in \GL(r)$ such that $\Sigma \p{-1/2}X_1 = A$ and $Y_1\Sigma \p{-1/2} = A\p{-1}$ iff $X_1 = \Sigma \p{1/2} A$ and $Y_1 = A\p{-1}\Sigma \p{1/2}$. Now $X_1Y_2 = 0$ implies $Y_2 = 0$ and $X_2Y_1 = 0$ implies $X_2 = 0$. Thus $XY = M$ iff $X_1 = \Sigma \p{1/2} A$ and $Y_1 = A\p{-1}\Sigma \p{1/2}$ and $(X_2,Y_2) = (0,0)$.

    Now consider the general case
    $$M = \begin{pmatrix}
        U & \widetilde U
    \end{pmatrix}
    \begin{pmatrix}
        \Sigma & 0 \\
        0 & 0
    \end{pmatrix}
    \begin{pmatrix}
        V & \widetilde V
    \end{pmatrix}\p T.
    $$
    Then $XY = M$ iff  
        $$\begin{pmatrix}
        U & \widetilde U
    \end{pmatrix}\p T XY \begin{pmatrix}
        V & \widetilde V
    \end{pmatrix} =
    \begin{pmatrix}
        \Sigma & 0 \\
        0 & 0
    \end{pmatrix}
    $$
    iff
    \begin{equation*}
    X = \begin{pmatrix}
        U & \widetilde U
    \end{pmatrix} 
    \begin{pmatrix}
        \Sigma\p {1/2} A \\ 0
    \end{pmatrix} = U \Sigma\p {1/2} A
    ~~~\text{and}~~~ 
    Y =
    \begin{pmatrix}
        A\p{-1} \Sigma\p {1/2} & 0
    \end{pmatrix}
    \begin{pmatrix}
        V & \widetilde V
    \end{pmatrix}\p T = A\p{-1} \Sigma\p {1/2} V\p T. \qedhere
    \end{equation*}
\end{proof}

\begin{example} 
\label{eg:unbalanced}
Given $a>b>0$, the flat minima of $\R\p{2\times 2}\times \R\p{2\times 2}\ni(X,Y)\mapsto \|XY-M\|_F\p 2\in \R$ with $M=\diag(a,b)$ are of the form
    $$(X,Y) = \left(\begin{pmatrix}
        \sqrt{a} & 0 \\
        0 & \sqrt{b}t
    \end{pmatrix}Q, Q\p T\begin{pmatrix}
        \sqrt{a} & 0 \\
        0 & \sqrt{b}/t
    \end{pmatrix}\right)$$
    where $\sqrt{a/b}\leq t \leq \sqrt{a/b}$ and $Q \in \O(2)$. They are globally but not strictly flat, and satisfy $X\p TX-YY\p T = b\hspace{.4mm}\diag(0,t\p 2 - 1/t\p 2)$.
\end{example}
\begin{proof}
    By \cref{prop:mf_2_flat} and \cref{fact:exact}, $(X,Y)$ is a flat minimum iff $\|X\|_2=\|Y\|_2=\sqrt{\|M\|_2}$ and $(X,Y)=(M\p{1/2}A,A\p{-1}M\p{1/2})$ for some $A\in \GL(2)$. Consider a decomposition $A=DNQ$ where $D=\diag(\alpha,\beta)$ with $\alpha,\beta>0$, $N$ is upper triangular of order 2 with ones on the diagonal with $N_{12}=\gamma\in\R$, and $Q \in\O(2)$. We have $\|X\|_2=\|M\p{1/2}A\|_2 =\|M\p{1/2}DNQ\|_2 =\|M\p{1/2}DN\|_2$ and $\|Y\|_2 = \|AM\p{1/2}\|_2 = \|Q\p{-1}N\p{-1}D\p{-1}M\p{1/2}\|_2 = \|N\p{-1}D\p{-1}M\p{1/2}\|_2$. Observe that $\|M\p{1/2}DN\|_2 = \|N\p{-1} D\p{-1} M\p{1/2}\|_2 = \|M\p{1/2}\|_2=\sqrt{a}$ iff $\alpha = 1$, $\sqrt{b/a}\leq \beta\leq \sqrt{a/b}$, and $N=I_2$. Indeed,
        $$ M\p {1/2} D N =
    \begin{pmatrix}
        \sqrt{a} & 0 \\
        0 & \sqrt{b}
    \end{pmatrix}
    \begin{pmatrix}
        \alpha & 0 \\
        0 & \beta
    \end{pmatrix}
    \begin{pmatrix}
        1 & \gamma \\
        0 & 1
    \end{pmatrix}
    =
    \begin{pmatrix}
        \sqrt{a}\alpha & \sqrt{a}\alpha\gamma \\
        0 & \sqrt{b}\beta
    \end{pmatrix} ~~~ \text{and}
    $$
    $$ N\p{-1} D\p{-1} M\p{1/2} =
    \begin{pmatrix}
        1 & -\gamma \\
        0 & 1
    \end{pmatrix}
    \begin{pmatrix}
        1/\alpha & 0 \\
        0 & 1/\beta
    \end{pmatrix}
    \begin{pmatrix}
        \sqrt{a} & 0 \\
        0 &\sqrt{b}
    \end{pmatrix}
    =
    \begin{pmatrix}
        \sqrt{a}/\alpha & -\gamma\sqrt{b}/\beta \\
        0 & \sqrt{b}/\beta
    \end{pmatrix}.
    $$
    Bear in mind that for any scalars $u$ and $v$, one has
    \begin{equation*}
    \left\|\begin{pmatrix}
        1 & u \\
        0 & v
    \end{pmatrix}\right\|_2 = \max_{x\p 2 + y\p 2 = 1} \sqrt{(x+uy)\p 2 + v\p 2 y\p 2} \geq \max_{x\p 2 + y\p 2 = 1} |x+uy| = \sqrt{1+u\p 2}.\qedhere
    \end{equation*}
\end{proof}

\cref{cor:balanced_flat} means that one can compute a flat global minimum easily using linear algebra: just take a compact singular value decomposition $M=U\Sigma V\p T$ and let $(X,Y)$ equal to $(U\Sigma\p {1/2},\Sigma\p {1/2} V\p T)$ up to some zero rows/columns. Below, we consider two complementary dynamical systems converging to balanced solutions, which are therefore flat. The first concerns gradient trajectories of the objective function.

\begin{proposition}
    \label{prop:f_flow}
    Given $M \in \R\p{m\times n}$ and $r\geq \rank (M)$, let $f:\R\p{m\times r}\times \R\p{r\times n}\to \R$ be defined by $f(X,Y) = \|XY-M\|_F\p 2/2$. There exist differentiable functions $(X,Y):\R\to\R\p{m\times r}\times \R\p{r\times n}$ such that
    $$(\dot{X},\dot{Y}) = -\nabla f(X,Y)$$ and $$ (0,0)\xleftarrow[-\infty \leftarrow t]{}(X(t),Y(t))\xrightarrow[t\to\infty]{} (X_\infty,Y_\infty)$$ where $(X_\infty,Y_\infty)$ is a balanced global minimum. 
\end{proposition}
\begin{proof}
    Without loss of generality, we may restrict ourselves to the diagonal case
    $$M = \begin{pmatrix}
        \Sigma & 0 \\
        0 & 0
    \end{pmatrix}.$$
    Let $(X_0,Y_0) = (\diag(\Sigma,0),\diag(\Sigma,0))/2 \in \R\p{m\times r}\times \R\p{r\times n}$
    and consider the initial value problem
    $$\left\{
    \begin{array}{ccc}
        (\dot{X},\dot{Y}) & = & -\nabla f(X,Y), \\
        (X(0),Y(0)) & = & (X_0,Y_0).
    \end{array}
    \right.$$
    It yields $\rank(M)$ decoupled systems
    $$\forall i\in\llbracket 1,\rank(M)\rrbracket,~~~ \left\{
    \begin{array}{ccc}
        (\dot{X}_{ii},\dot{Y}_{ii}) & = & -\nabla \varphi_i(X_{ii},Y_{ii}), \\
        (X_{ii}(0),Y_{ii}(0)) & = & (\Sigma_{ii}\p{1/2},\Sigma_{ii}\p{1/2})/2,
    \end{array}
    \right.$$
    where $\varphi_i(x,y) = (xy-\Sigma_{ii})\p 2/2$ and all other entries are fixed to zero throughout time. This reduces the problem to the scalar case, i.e., $f(x,y) = (xy-1)\p 2$, and to its dynamics on the diagonal $x=y$, i.e., on the function $g(x) = (x\p 2 -1)\p 2$. We deduce that 
    $$ (0,0)\xleftarrow[-\infty \leftarrow t]{}(X_{ii}(t),Y_{ii}(t))\xrightarrow[t\to\infty]{} (\Sigma_{ii}\p{1/2},\Sigma_{ii}\p{1/2})$$
    and conclude by \cref{cor:balanced_flat}.
\end{proof}

The second concerns gradient trajectories of the conserved quantity $\|C\|_F\p 2$. We recall a simple fact.

\begin{fact}
    \label{fact:spectral}
    Let $\phi:\R\p {m\times n} \to \R$ be defined by $\phi(A) = \|A\|_2$. Given a singular value decomposition $A=U\Sigma V\p T$ where $u_i$ and $v_i$ resp. denote the columns of $U$ and $V$, $\partial \phi(A) = \co\{ u_iv_i\p T : i\in\llbracket 1,d_A\rrbracket\}$. 
\end{fact}
\begin{proof}
    We have $\partial \phi(A) = U\partial \phi(\Sigma)V\p T$ and
    $\phi(\Sigma+tH) = \sup \{|(\Sigma+tH)w|:w\in S\p{n-1}\} = \|\Sigma\|_2+t\max\{|H_{ii}|:i\in\llbracket 1, d_A\rrbracket\} = \phi(\Sigma)+t\sigma_{\{E_{ii}\}_{i\in \llbracket 1, d_A\rrbracket}}(H)$ where $E_{ij}$ denotes the canonical basis of $\R \p{m\times n}$. By the max formula \cite[Theorem 23.4]{rockafellar1970convex}, $\phi'(\Sigma,H) = \sigma_{\partial \phi(\Sigma)}(H) = \sigma_{\{E_{ii}\}_{i\in \llbracket 1, d_A\rrbracket}}(H)$ and so $\partial\phi(\Sigma)= \co \{E_{ii}:i\in \llbracket 1, d_A\rrbracket\}$.
\end{proof}

\begin{proposition}
\label{prop:c_flow}
    Given $M \in \R\p{m\times n}$, let $f,c,\varphi:\R\p{m\times r}\times \R\p{r\times n}\to \R$ be defined by 
    \begin{align*}
    f(X,Y) & = \|XY-M\|_F\p 2/2, \\
    c(X,Y) & = \|X\p T X - YY\p T\|_F\p 2/4, \\
    \varphi(X,Y) & = \|X\|_2\p 2-2\|XY\|_2+\|Y\|_2\p 2.
    \end{align*}
    The following hold: 
    \begin{enumerate}[label=\rm{(\roman{*})}]
        \item $f+c$ is coercive; \label{item:coercive}
        \item For any $XY=M$,~ $\nabla c(X,Y) = 0 ~\Longleftrightarrow~ X\p T X = Y Y \p T$; \label{item:gradient}
        \item For any $(X,Y)$,~ $\varphi'(X,Y;-\nabla c(X,Y)) \leq -\varphi(X,Y)\p 2$;
        \label{item:directional}
        \item For any $X_0Y_0= M$, 
    $$
    \left\{
    \begin{array}{ccc}
        (\dot{X},\dot{Y}) & = & -\nabla c(X,Y), \\
        (X(0),Y(0)) & = & (X_0,Y_0),
    \end{array}
    \right.
    $$
    has a unique flattening global solution converging to a balanced global minimum of $f$ such that 
    $$\frac{d}{dt}\lambda_1(\nabla\p 2f(X,Y))\leq -(\lambda_1(\nabla\p 2f(X,Y))-2\|M\|_2)\p 2,$$
    $$\forall t>0, ~~~ \lambda_1(\nabla\p 2f(X(t),Y(t))) \leq 2\|M\|_2 + 1/t.$$
    \label{item:ode}
    \end{enumerate} 
\end{proposition}
\begin{proof}
\ref{item:coercive} Following the derivation in \cite[Example 1]{josz2023global},
    \begin{align*}
        \|X\|_2\p 4 + \|Y\|_2\p 4 & \leq \|X\p T X\|_F\p 2 + \|Y Y\p T\|_F\p 2 \\
        & = \|X\p T X -Y Y\p T\|_F\p 2 + 2\langle X\p TX, YY\p T\rangle \\ 
        & = \|X\p T X -Y Y\p T\|_F\p 2 + 2\|XY\|_F\p 2 \\
        & \leq \|X\p T X -Y Y\p T\|_F\p 2 + 2(\|XY-M\|_F+\|M\|_F)\p 2 \\
        & = 2c(X,Y)+2\left(\sqrt{2f(X,Y)}+\|M\|_F\right)\p 2 \\
        & \leq 2(f+c)(X,Y)+2\left(\sqrt{2(f+c)(X,Y)}+\|M\|_F\right)\p 2.
    \end{align*}
    \ref{item:gradient} For any point $(X,Y)$, we have
    \begin{equation}
    \label{eq:connection_c}
    \nabla c(X,Y) = \begin{pmatrix}
         X(X\p T X - Y Y\p T)  \\
         -(X\p T X - Y Y\p T)Y 
    \end{pmatrix}. 
    \end{equation}
    If $XY=M$ and $\nabla c(X,Y) = 0$, then $XX\p T X = XY Y\p T= M Y\p T$ and $Y Y\p TY = X\p T XY = X\p T M$. Thus $X\p T XX\p T X = X\p T M Y\p T = Y\p T Y Y\p TY$ and $X\p T X = Y Y\p T$ by \cref{fact:commute}.

    \ref{item:directional} By \cref{fact:spectral},
    $\partial \varphi(X,Y)$ is the convex hull of the points of the form $2(s_x u_x v_x\p T,s_y u_y v_y\p T)$
    where $s_x = \|X\|_2$ and $s_y = \|Y\|_2$ and $u_x,v_x,u_y,v_y$ are some singular vectors of $X$ and $Y$ resp. The max formula \cite[Theorem 23.4]{rockafellar1970convex} states that $\varphi'(X,Y;D) = \max \{ \langle H, D\rangle : H\in \partial \varphi(X,Y)\}$. We thus compute
    \begin{align*}
        & ~ \langle (s_x u_x v_x\p T,s_y u_y v_y\p T) , \nabla c(X,Y)\rangle \\
        = & ~ \langle s_x u_x v_x\p T , X(X\p T X - Y Y\p T)\rangle - \langle (X\p T X - Y Y\p T)Y,s_y u_y v_y\p T\rangle \\
        = & ~ s_x u_x\p T X(X\p T X - Y Y\p T)v_x - s_y u_y\p T (X\p T X - Y Y\p T)Y v_y \\
        = & ~ s_x\p 2 v_x\p T(X\p T X - Y Y\p T)v_x - s_y\p 2 u_y\p T (X\p T X - Y Y\p T)u_y \\
        = & ~ s_x\p 4 - s_x\p 2 |Y\p Tv_x|\p 2 - s_y\p 2 |Xu_y|\p 2 +s_y\p 4 \\
        = & ~ s_x\p 4 - |Y\p TX\p T u_x|\p 2 - |XYv_y|\p 2 +s_y\p 4 \\
        \geq & ~ s_x\p 4 - 2s_{xy}\p 2 +s_y\p 4 \\
        \geq & ~ (s_x\p 2 - 2s_{xy} +s_y\p 2)\p 2/2
    \end{align*}
    where the last inequality is due to $s_{xy} =\|XY\|_2\leq \|X\|_2\|Y\|_2= s_xs_y$. Indeed, 
    \begin{align*}
        2s_x\p 4 - 4s_{xy}\p 2 +2s_y\p 4 - (s_x\p 2 - 2s_{xy} +s_y\p 2)\p 2 & = s_x\p 4 + s_y\p 4 - 8s_{xy}\p 2 +4 s_x\p 2 s_{xy} + 4s_{xy} s_y\p 2 - 2s_x\p 2 s_y\p 2 \\
        & = (s_x\p 2 - s_y\p 2)\p 2 + 4s_{xy}(s_x\p 2 + s_y\p 2-2s_{xy}) \\
        & \geq (s_x\p 2 - s_y\p 2)\p 2 + 4s_{xy}(s_x\p 2 + s_y\p 2-2s_{x}s_{y}) \\
        & = (s_x\p 2 - s_y\p 2)\p 2 + 4s_{xy}(s_x- s_y)\p 2.
    \end{align*}
    
    \ref{item:ode} Since $XY$ and $f$ are conserved along trajectories of $-\nabla c$ by \cref{prop:conserved_mutual}, $f+c$ decreases along them. By \ref{item:coercive}, they are bounded and hence globally defined according to \cite[Proposition 2]{josz2023global}. They must converge to balanced points by virtue of \ref{item:gradient}. Let $\mu = \varphi(X,Y) = \lambda_1(\nabla\p 2f(X,Y))-2\|M\|_2$, where the second equality is due to \cref{prop:mf_hessian}. Since $\varphi$ agrees with a convex function along any fixed trajectory, by \cite[Proposition 2]{bolte2020conservative} the chain rule holds almost surely:
    $$\dot{\mu} = \sup_{H\in \partial \varphi(X,Y)} \langle H ,(\dot{X},\dot{Y})\rangle = \sup_{H\in \partial \varphi(X,Y)} \langle H ,-\nabla c(X,Y)\rangle = \varphi'(X,Y;H) \leq - \mu\p 2$$
    where the last inequality is due to \ref{item:directional}.
    Integrating $-\dot{\mu}\mu\p {-2} \geq 1$ yields $1/\mu(t)\geq 1/\mu(t)-1/\mu(0)\geq t$ if $\mu(0)\neq 0$, otherwise $\mu(t) = 0$ for all $t\geq 0$.
\end{proof}

\cref{prop:c_flow} complements the analysis of Riemannian gradient dynamics in \cite[Theorem 5.1]{helmke1994optimization}. Given $X_0Y_0=M$ with $\rank(X_0) = \rank(Y_0)=r$, it considers 
$$\left\{
\begin{array}{ccc}
    (\dot{X},\dot{Y}) & = & - \mathrm{grad}\hspace{.3mm}\Phi, \\
    (X(0),Y(0)) & = & (X_0,Y_0),
\end{array}
\right. $$
where $\Phi: \GL(r)(X_0,Y_0)\to\R$ is defined by $\Phi(X,Y) = \|X\|_F\p 2 + \|Y\|_F\p 2$, and the orbit is equipped with the normal Riemannian metric \cite[Equation (5.3) p. 186]{helmke1994optimization}. \cite[Theorem 5.1 (a)]{helmke1994optimization} shows that $\mathrm{grad}\hspace{.3mm}\Phi = (X(X\p T X - Y Y\p T),-(X\p T X - Y Y\p T)Y )$, but no connection is made with the conserved quantity. According to \eqref{eq:connection_c}, we in fact have $\mathrm{grad}\hspace{.3mm}\Phi = \nabla c$, which considerably simplifies the analysis and does not require the rank condition.
If that condition holds, then each trajectory of $-\nabla c$ remains in the orbit of the initial point.

\section{Examples}
\label{sec:Examples}

Some examples are in order. 

\begin{example}
    \label{eg:nn}
    A point $x$ is a flat minimum of $f(x) = (x_2 \mathrm{ReLU}(x_1)+x_3 - 1)\p 2$ iff $(x_1 < 0 \land x_3 = 1)\lor(x_1= 0 \land |x_2| \leq 1 \land x_3= 1)$.
\end{example}
\begin{proof}
     Observe that $f$ fails to be regular at $(0,x_2,x_3)$ when $x_2(x_3-1)<0$, so the calculus rule in \cref{cor:regular} does not apply. Regardless, by \cref{rem:jacobian}, $\lip f(x) = 0$ whenever $f(x) = 0$, so it wouldn't be of much use. We simply resort to the definition. If $(x_1 < 0 \land x_3 = 1)\lor(x_1= 0 \land |x_2| \leq 1 \land x_3= 1)$, then $\cf(x,r) = r\p 2$ for $r$ near $0$. If $x_1= 0 \land |x_2| > 1 \land x_3= 1$, then $\cf(x,r) = x_2\p 2 r\p 2>r\p 2$ near $0$. Otherwise, $x_1>0$ and $f(x) = (x_2x_1+x_3-1)\p 2$, in which case
     $$\nabla f (x) = 2(x_2x_1+x_3-1)\begin{pmatrix}
         x_2 \\ x_1 \\ 1
     \end{pmatrix} ~~~ \text{and}~~~ \nabla\p 2 f (x) = 2 \begin{pmatrix}
         x_2\p 2 & 2x_2x_1 + x_3-1 & x_2 \\
         2x_2x_1 + x_3-1 & x_1\p 2 & x_1 \\
         x_2 & x_1 & 1
     \end{pmatrix} $$
     If $f(x) = 0$, then
     $$\nabla\p 2 f (x) = 2 \begin{pmatrix}
         x_2\p 2 & x_2x_1 & x_2 \\
         x_2x_1 & x_1\p 2 & x_1 \\
         x_2 & x_1 & 1
     \end{pmatrix} = \begin{pmatrix}
         x_2 \\ x_1 \\ 1
     \end{pmatrix} \begin{pmatrix}
         x_2 \\ x_1 \\ 1
     \end{pmatrix} \p T$$
     and so $\lambda_1(\nabla \p 2 f(x)) = x_2\p2 + x_1\p 2 +1$. 
     By \cref{prop:smooth}, $\cf(x,r) = (x_2\p 2+x_1\p 2 + 1)r\p 2/2 + o(r\p 2) >r\p 2$.
\end{proof}

\begin{example}
    \label{eg:mf4}
    The flat minima of $f(x) = (x_1x_2-1)^4$ are $\pm(1,1)$.
\end{example} 
\begin{proof}
By \cref{fact:chain_smooth}, $\|f\p{(4)}(x)\| = 4|F'(x)|_2\p 4$ when $F(x) = x_1x_2-1 = 0$, i.e., $\|f\p{(4)}(x)\| = 4(x_1\p 2+x_2\p 2)\p 2$. This quantity is strictly minimized at $\pm(1,1)$ over the solution set, so one concludes by \cref{cor:smooth}.
Alternatively, one notices that the flat minima are the same as that of $|x_1x_2-1|$ by using \cref{def:flat} and the fact that $t\in \R_+ \mapsto t\p 4$ is strictly increasing. Since the new function is regular and locally Lipschitz, by \cref{cor:regular} it suffices to compute its Lipschitz modulus when $x_1x_2=1$, i.e., $x_1\p 2 +x_2\p 2$.
\end{proof}

\begin{example}
\label{eg:4th}
    The origin is the sole flat minimum of $f(x) = x_2\p 2 +x_1\p2x_2\p 4$.
\end{example}
\begin{proof}
Compute 
$$F'(x) = 2\begin{pmatrix}
    x_1  x_2\p 4 \\
    x_2 +2x_1\p2x_2\p 3
\end{pmatrix} ~~~\text{and}~~~ \nabla\p 2 f(x) = 2 \begin{pmatrix}
    x_2\p 4 & 4x_1 x_2\p 3 \\
    4x_1 x_2\p 3 & 1 + 6 x_1\p 2 x_2\p 2
\end{pmatrix}.$$
Thus $\lambda_1(\nabla\p 2 f(x_1,0)) = 2$, as claimed in the introduction. By virtue of Examples \ref{eg:flat_not_imply_strict} and \ref{eg:strict_not_imply_flat}, it would seem futile to compute higher-order derivatives. But in this special case, the maximal eigenvectors of $f\p{(2)}(x_1,0)$ and $f\p{(4)}(x_1,0)$ align (with $(0,1)$) while $f\p{(3)}(x_1,0)=0$. Since $f\p{(4)}(x_1,0) = 24x_1\p 2$, we deduce that $\cf(x_1,0,r) = r\p 2 + x_1\p 2r\p4 +o(r\p 4)$ and conclude by \cref{def:flat}.

\end{proof}


\begin{example}
    \label{eg:orthgonal}
    A global minimum of $f(x) = (a_1x_1\p 2+\cdots+a_nx_n\p 2-1)\p 2$ where $a \in \R\p n$ is flat iff $a_ix_i = 0$ for all $i\notin I =\arg\min \{ a_i : a_i > 0\}$.
\end{example}
\begin{proof}
    Without loss of generality, assume $a_i\neq 0$ for all $i\in\llbracket 1,n\rrbracket$. If $I=\emptyset$, then $f(x) = (|a_1|x_1\p 2+\cdots+|a_n|x_n\p 2+1)\p 2$ and every global minimum is flat. Otherwise, let $m$ be the cardinal of $I$. The objective $f$ is invariant under the natural action of $G = \mathrm{diag} (O(m) , I_{n-m})$ once we reorder the indices so that those in $I$ come first. By \cref{fact:chain_smooth}, $\lambda_1(\nabla\p 2 f(x)) = 2|F'(x)|_2\p 2$ when $F(x) = a_1x_1^2+\dots+a_n x_n^2-1=0$. Since $F'(x) = 2(a_1x_1,\dots,a_nx_n)$, we seek to minimize $|F'(x)|\p 2=4(a_1^2 x_1^2+\dots+a_n^2 x_n^2)$ subject to $a_1 x_1^2+\dots+a_n x_n^2=1$. Consider the Lagrangian $$L(x,\lambda)= \sum_{i=1}\p n a_i\p 2 x_i^2 + \lambda\left(1-\sum_{i=1}\p n a_ix_i\p 2\right).$$ Since the constraint is qualified, at optimality $a_i\p 2 x_i - \lambda a_ix_i=0$, namely, $a_ix_i(a_i-\lambda)=0$. If $\lambda \neq a_i$ for all $i$, then $0=a_1 x_1^2+\dots+a_n x_n^2=1$, a contradiction. If $\lambda = a_i<0$ for some $i$, then $0>a_1 x_1^2+\dots+a_n x_n^2=1$, a contradiction. If $\lambda = a_i>0$ for some $i$, then $a_1^2 x_1^2+\dots+a_n^2 x_n^2 = \lambda$ and so $i\in I$. Thus, for any global minimum $\overline{x}$ of $f$ such that $\overline{x}_i = 0$ for all $i\notin I$, $G\overline{x}$ is a strict global minimum of $\lambda_1(\nabla\p 2 f(x))+\delta_{[f=f(\overline{x})]}$. It follows that every point in $G\overline{x}$ is flat by \cref{cor:orthogonal_smooth}. No other global minimum is flat by \cref{cor:smooth}.
\end{proof}

\begin{example}
    \label{eg:mf1_ab}
    The flat global minima of $f(x) = |x_1x_3-a|+|x_2x_3-b|$ are 
    $$x = \pm \left( a\sqrt{ \frac{\sqrt{2}}{|a|+|b|} }, b\sqrt{ \frac{\sqrt{2}}{|a|+|b|} } ,\sqrt{ \frac{|a|+|b|}{\sqrt{2}} }\right)$$
    if $(a,b)\neq (0,0)$, else $(0,0,0)$.
\end{example}

    \begin{proof} By \cref{fact:chain_reg}, $f$ is regular,
    \begin{equation*}
        \partial f(x) = \left\{ \begin{pmatrix}
            \lambda_1 x_3 \\
            \lambda_2 x_3 \\
            \lambda_1 x_1 + \lambda_2 x_2
        \end{pmatrix}
        ~,\hspace{3mm} \lambda \in \begin{pmatrix}
        \mathrm{sign}(x_1x_3-a) \\ \mathrm{sign}(x_2x_3-b)
        \end{pmatrix}
        \right\},
    \end{equation*}
    and $\lip f(x) = \sqrt{2|x_3|\p 2 + (|x_1|+|x_2|)\p 2}$. Thus $(0,0,0)$ is the sole flat global minimum when $(a,b)=(0,0)$ by \cref{cor:regular}.
    When $(a,b)\neq (0,0)$, $\mathrm{argmin} f = \{(at,bt,1/t) : t\neq 0\}$. Accordingly, given $t\neq 0$, let $x_t = (at,bt,1/t)$ and compute
    \begin{align*}
        \lip f(x_t)^2 & = \max \{ |v|^2 : v \in \partial f(x_t) \} \\
        & = \max \{ (\lambda_1/t)^2 + (\lambda_2/t)^2 + (\lambda_1 at + \lambda_2 bt)^2 : \lambda_1,\lambda_2 \in [-1,1] \} \\
        & = \max \{ (\lambda_1^2+\lambda_2^2)/t^2 + (\lambda_1 a + \lambda_2 b)^2t^2 : \lambda_1,\lambda_2 \in [-1,1] \} \\
        & = 2/t^2 + (|a| + |b|)^2t^2.
    \end{align*}
    When $ab\neq 0$, the last max is reached exactly at $\pm(\mathrm{sign}(a),\mathrm{sign}(b))$, whence
    \begin{equation*}
        \mathrm{argmax} \{ |v| : v \in \partial f(x_t) \} = \left\{ 
        \pm\begin{pmatrix}
            \mathrm{sign}(a)/t \\
            \mathrm{sign}(b)/t \\
            (|a| + |b|)t
        \end{pmatrix}
        \right\}.
    \end{equation*}
    Since 
    \begin{equation*}
        \frac{d\lip f(x_t)^2}{dt} = -\frac{4}{t^3} + 2(|a|+|b|)^2t ~~~\text{and}~~~ \frac{d^2\lip f(x_t)^2}{dt^2} = \frac{12}{t^4} + 2(|a|+|b|)^2 > 0,
    \end{equation*}
    the Lipschitz modulus is strictly minimized when
    \begin{equation*}
        |t| = \sqrt{\frac{\sqrt{2}}{|a|+|b|}}.
    \end{equation*} 
    This yields the expression of the flat global minima by \cref{cor:regular}. 
    
    We now verify the claim of \cref{prop:mf_1}. The objective $f$ is invariant under the natural action of the Lie group $G = \{ \mathrm{diag}(t,t,1/t) : t\neq 0 \}$ whose Lie algebra is $ \fg =\mathrm{span} \{\mathrm{diag}(1,1,-1) \}$. By \cref{prop:conserved}, a conserved quantity is given by $C(x) =  x_1^2+x_2^2 - x_3^2$. When $ab\neq 0$,
    the quantity
    \begin{equation*}
        C\left(\pm\begin{pmatrix}
            \mathrm{sign}(a)/t \\
            \mathrm{sign}(b)/t \\
            (|a| + |b|)t
        \end{pmatrix}\right) = (\mathrm{sign}(a)/t)^2 + (\mathrm{sign}(b)/t)^2 - [(|a| + |b|)t]^2 = 2/t^2 - (|a| + |b|)^2t^2
    \end{equation*}
    indeed cancels out iff $x_t$ is flat. (The same holds with $ab=0$, but we omit this case for brevity.) On the other hand,
    \begin{equation*}
        C(x_t) = (a^2+b^2)t^2 - 1/t^2 \neq 0
    \end{equation*}
    unless $|a| = |b|$. Indeed, $C(x_t) = 0$ for a flat minimum if and only if
    \begin{equation*}
       \sqrt{\frac{\sqrt{2}}{|a|+|b|}} = \frac{1}{\sqrt[4]{a^2+b^2}} ~ \Longleftrightarrow ~ 2(a^2+b^2) = (|a|+|b|)^2 ~ \Longleftrightarrow ~ a^2+b^2 = 2|ab| ~ \Longleftrightarrow ~ |a| = |b|.
    \end{equation*} 
    \end{proof}

\begin{example}
    \label{eg:monomial}
    The flat global minima of $f(x) = (x^\upsilon - 1)\p 2$ where $x^\upsilon = x_1^{\upsilon_1}\cdots x_n^{\upsilon_n}, \upsilon \in {\mathbb{N}^*}^n,$ are
    $$|x_i| = \frac{\sqrt{\upsilon_i}}{\sqrt{\upsilon_1^{\upsilon_1}\cdots\upsilon_n^{\upsilon_n}}^{1/|\upsilon|_1}}, ~~~ i = 1, \hdots, n,$$
    for any choice of signs such that $x^\upsilon=1$.
    Any solution to
    $$
    \left\{
    \begin{array}{ccl}
        \dot{x}_i & = & - \upsilon_n x_i (\upsilon_n x_i^2 - \upsilon_i x_n^2), ~ i = 1,\hdots,n-1, \\
    ~ \dot{x}_n & = & \sum_{i=1}\p {n-1} \upsilon_i x_n (\upsilon_n x_i^2 - \upsilon_i x_n^2),
    \end{array}
    \right.
    $$
    initialized at a global minimum is globally defined, flattens over time, and converges to a flat global minimum.
\end{example}
\begin{proof}
    Let $\overline{\upsilon} = (\upsilon_1,\hdots,\upsilon_{n-1})$.
    The objective is invariant under the natural action of the Lie group
    $$ G = \{ \mathrm{diag}(t_1^{\upsilon_n},\hdots,t_{n-1}^{\upsilon_n},t^{-\overline{\upsilon}}) : t_1,\hdots,t_{n-1} > 0  \}
    $$
    whose Lie algebra is $$ \fg =\mathrm{span} \{\mathrm{diag}(\upsilon_n,0,\hdots,0,-\upsilon_1) ,\hdots, \mathrm{diag}(0,\hdots,0, \upsilon_n,-\upsilon_{n-1}) \}.$$ By Proposition \ref{prop:conserved}, a conserved quantity is given by
    $$C(x) = (\upsilon_nx_1^2 - \upsilon_1x_n^2,\hdots,\upsilon_n x_{n-1}^2 - \upsilon_{n-1}x_n^2).$$ 
    
    The connected components of the global minima of $f$ are homogeneous $G$-spaces. Indeed, given $\overline{x}$ a global minimum of $f$, we have
    $$G \overline{x} = \{ x \in \mathbb{R}^n : x^\upsilon = 1, ~ \mathrm{sign}(x) = \mathrm{sign}(\overline{x})\}.$$
    Let $x$ be such that $x^\upsilon = 1$ and $\mathrm{sign}(x) = \mathrm{sign}(\overline{x})$. Then $x = \mathrm{diag}(t_1^{\upsilon_n},\hdots,t_{n-1}^{\upsilon_n},t^{-\overline{\upsilon}}) \overline{x}$ with $t_i = (x_i/\overline{x}_i)^{1/\upsilon_n}$ since $t_i^{\upsilon_n}\overline{x}_i = x_i$ for all $i$ and 
    $t^{-\overline{\upsilon}} \overline{x}_n = t^{-\overline{\upsilon}} \overline{x}_1^{-\upsilon_1/\upsilon_n} \cdots \overline{x}_{n-1}^{-\upsilon_{n-1}/\upsilon_n} 
        = (t_1^{\upsilon_n}\overline{x}_1)^{-\upsilon_1/\upsilon_n} \cdots (t_{n-1}^{\upsilon_n}\overline{x}_{n-1})^{-\upsilon_{n-1}/\upsilon_n}
        = x_1^{-\upsilon_1/\upsilon_n} \cdots x_{n-1}^{-\upsilon_{n-1}\upsilon_n} = x_n$.
        
    It is now possible to determine flat points. Let $F(x) = x^\upsilon - 1$ and compute
    $$F'(x) = (\upsilon_1 x^{\upsilon}/x_1,\hdots,\upsilon_n x^{\upsilon}/x_n).$$
    By \cref{fact:chain_smooth}, $\lambda_1(\nabla\p 2 f(x)) = 2|F'(x)|\p 2$ when $F(x)=0$. Given $t_1,\hdots,t_{n-1} > 0 $, let $x_t = \mathrm{diag}(t_1^{\upsilon_n},\hdots,t_{n-1}^{\upsilon_n},t^{-\overline{\upsilon}}) \overline{x}$ where $\overline{x}$ is global minimum such that $|\overline{x}_i|=1$ for all $i$. We have
    $$|F'(x_t)|\p 2 = \upsilon_1^2 t_1^{-2\upsilon_n} + \cdots + \upsilon_{n-1}^2 t_{n-1}^{-2\upsilon_n} + \upsilon_n^2 t^{2\overline{\upsilon}}$$ and $$\frac{\partial |F'(x_t)|}{\partial t_i} =
        2\upsilon_i\upsilon_n(-\upsilon_i t_i^{-2\upsilon_n} + \upsilon_n t^{2\overline{\upsilon}})/t_i.$$
Thus
$$\nabla |F'(x_t)|^2 = 0 \Longleftrightarrow \upsilon_1 t_1^{-2\upsilon_n} = \cdots = \upsilon_{n-1} t_{n-1}^{-2\upsilon_n} = \upsilon_n t^{2\overline{\upsilon}}$$
and
$$ \frac{\partial^2 |F'(x_t)|\p 2}{\partial t_i \partial t_j} =
    \left\{
    \begin{array}{cl}
        2\upsilon_i\upsilon_n((2\upsilon_n+1)\upsilon_i t_i^{-2\upsilon_n} +(2\upsilon_i-1) \upsilon_n t^{2\overline{\upsilon}})/t_i^2 & \text{if}~i=j, \\
        4\upsilon_i\upsilon_j\upsilon_n^2 t^{2\overline{\upsilon}}/(t_it_j) & \text{else}.
    \end{array}
    \right.$$
Observe that
$$ \nabla |F'(x_t)|^2 = 0 ~~~\Longrightarrow ~~~  \nabla^2 |F'(x_t)|^2 = 4\upsilon_n^2 t^{2\overline{\upsilon}} \left(\mathrm{diag}(\upsilon \oslash t)^2 + (\upsilon \oslash t)(\upsilon \oslash t)^T\right) \succ 0 $$
where $\oslash$ is the Hadamard division and $\succ$ means positive definite here. Stationary implies
$ \upsilon_1^{\upsilon_1} \cdots \upsilon_{n-1}^{\upsilon_{n-1}} t^{-2\upsilon_n\overline{\upsilon}} = (\upsilon_n t^{2\overline{\upsilon}})^{\upsilon_1+\cdots+\upsilon_{n-1}}$
so that $\upsilon_1^{\upsilon_1} \cdots \upsilon_n^{\upsilon_n} = (\upsilon_n t^{2\overline{\upsilon}})^{|v|_1}$. Thus $t^{-\overline{\upsilon}} = \upsilon_n\p{1/2}(\upsilon_1^{\upsilon_1}\cdots\upsilon_n^{\upsilon_n})^{-1/(2|\upsilon|_1)}$. Similarly $(\upsilon_i t_i^{-2\upsilon_n})^{|\upsilon|_1} = \upsilon_1^{\upsilon_1} \cdots \upsilon_n^{\upsilon_n}$ and $t_i^{\upsilon_i} = \upsilon_i\p{1/2}(\upsilon_1^{\upsilon_1}\cdots\upsilon_n^{\upsilon_n})^{-1/(2|\upsilon|_1)}$. Positive definiteness implies strict local optimality. One now concludes by \cref{cor:smooth}. Note that $$\lambda_1(\nabla\p 2 f(x_t)) = 2(\upsilon_1^2 t_1^{-2\upsilon_n} + \cdots + \upsilon_{n-1}^2 t_{n-1}^{-2\upsilon_n} + \upsilon_n^2 t^{2\overline{\upsilon}})$$ is positive and coercive.

Observe that a global minimum $x$ is flat iff $C(x)=0$. The `only if' part follows from the formula we just found, so it remains to check the `if' part. If $C_i(x) = \upsilon_n x_i^2 - \upsilon_i x_n^2 = 0$, then $|x_i| = \sqrt{\upsilon_i/\upsilon_n}|x_n|$,
    $$|x|^\upsilon = \prod_{i=1}^n \sqrt{\upsilon_i/\upsilon_n}^{\upsilon_i} |x_n|^{\upsilon_i} = 1,~~~\text{and}~~~
    |x_n|^{|\upsilon|_1} = \prod_{i=1}^n \sqrt{\upsilon_n/\upsilon_i}^{\upsilon_i} = \frac{\sqrt{\upsilon_n}^{|\upsilon|_1} }{\prod_{i=1}^n \sqrt{\upsilon_i^{\upsilon_i}}}.$$ Accordingly, let $c(x) = \|C(x)\|_F\p 2/4$. Notice that a global minimum $x$ is flat iff $\nabla c(x) = 0$. For all $x \in \R\p n$, we have
    \begin{align*}
        C_i(\nabla f(x)) & = 4(x^\upsilon-1)^2 [\upsilon_n(\upsilon_i x^{\upsilon}/x_i)^2 - \upsilon_i(\upsilon_n x^{\upsilon}/x_n)^2] \\
        & = 4 (x^\upsilon-1)^2\upsilon_i\upsilon_n(x^\upsilon/(x_ix_n))^2[\upsilon_ix_n^2-\upsilon_nx_i^2] \\
        & = -4\upsilon_i\upsilon_n((x^\upsilon-1)(x^\upsilon/(x_ix_n))^2 C_i(x).
    \end{align*}
Let $\overline{x}\in[f=0]$ be such that $C(\overline{x})\neq 0$. There is an index $i_0$ such that $C_{i_0}(\overline{x})\neq 0$. For all $x\in\R\p n$ near $\overline{x}$, we have
\begin{align*}
    \langle C(x),C(\nabla f(x))\rangle & = \sum_{i=1}\p {n-1} 4C_i(x)[-\upsilon_i\upsilon_n((x^\upsilon-1)/(x_ix_n))\p2 C_i(x)] \\
    & = -4(x^\upsilon-1)\p 2\frac{\upsilon_n x^{2\upsilon}}{x_n\p2} \sum_{i=1}\p {n-1} \frac{\upsilon_i C_i(x)\p2}{x_i\p2} 
    \leq -4(x^\upsilon-1)\p 2\frac{\upsilon_n x^{2\upsilon}}{x_n\p2}  \frac{\upsilon_{i_0}C_{i_0}(\overline{x})\p 2}{x_{i_0}\p2} \\
    & \leq -4\omega (x^\upsilon-1)\p 2 x^{2\upsilon} \sum_{i=0}\p n \frac{\upsilon_i\p 2}{x_i\p 2} = -\omega|\nabla f(x)|\p 2
\end{align*}
for some constant $\omega>0$. Together with \cref{prop:conserved}, this means that \cref{assum:conserved} holds. Applying \cref{cor:flattening}, $\lambda_1(\nabla\p 2 f)$ strictly decreases along the trajectories of $-\nabla c$. Since $\lambda_1(\nabla\p 2 f)+\delta_{[f=0]}$ is coercive, they must be bounded and hence globally defined by \cite[Proposition 2]{josz2023global}. Since $c$ is semi-algebraic, they must also converge to a point $\overline{x}$ where $\nabla c(\overline{x})=0$ \cite{lojasiewicz1984trajectoires}, i.e., a flat global minimum.
\end{proof}

To avoid overburdening the examples, we have not mentioned that all the flat minima are, in fact, globally flat.

\section*{Acknowledgments}
I would like to thank Wenqing Ouyang for his careful reading of the manuscript and valuable feedback. I also wish to thank Xiaopeng Li for useful criticism.

\bibliographystyle{abbrv}    
\bibliography{references}
\end{document}